\documentclass[11pt,a4paper]{article}
\usepackage{bbm}
\usepackage{psfrag,epic,eepic,epsfig}
\usepackage{subfig}
\usepackage{color}
\usepackage[applemac]{inputenc} 
\usepackage{amsmath,amsfonts,amssymb,amsthm,mathrsfs}
\usepackage{graphicx,graphics}
\usepackage{latexsym}
\usepackage{ae,aecompl}
 \usepackage[colorlinks=true]{hyperref}
\usepackage{enumerate}
\usepackage{comment}
\usepackage[normalem]{ulem}
\usepackage{xcolor}
\usepackage{empheq}
\usepackage[framemethod=tikz]{mdframed}
\usepackage{lipsum}
\usepackage{pifont}
\usepackage{fancybox} 
\usepackage{textcomp}
\usepackage{tikz}
\usepackage{graphicx}
\usepackage[margin=1cm,font=small]{caption}
\usepackage{array}

\definecolor{gris25}{gray}{0.75}

\definecolor{mycolor}{rgb}{0, 0, 0.1}

\newmdenv[innerlinewidth=0.5pt, roundcorner=4pt,linecolor=mycolor,innerleftmargin=6pt,
innerrightmargin=6pt,innertopmargin=6pt,innerbottommargin=6pt]{mybox}

\newcommand{\Ec}[1]{\mathbb{E} \left[#1\right]}
\newcommand{\abs}[1]{\left\lvert#1\right\rvert}
\newcommand{\Pp}[1]{\mathbb{P} \left(#1\right)}
\newcommand{\enstq}[2]{\left\lbrace#1\mathrel{}\middle|\mathrel{}#2\right\rbrace}
\newcommand{\Ppsq}[2]{\mathbb{P} \left(#1\mathrel{}\middle|\mathrel{}#2\right)}
\newcommand{\Ecsq}[2]{\mathbb{E} \left[#1\mathrel{}\middle|\mathrel{}#2\right]}
\newenvironment{steplist}
{ \begin{list}%
		{$\bullet$}%
		{\setlength{\labelwidth}{100pt}%
			\setlength{\leftmargin}{80pt}%
			\setlength{\itemsep}{\parsep}}}%
	{ \end{list} }

\newcommand{\intervalle}[4]{\mathopen{#1}#2
	\mathclose{}\mathpunct{},#3
	\mathclose{#4}}
\newcommand{\intervalleff}[2]{\intervalle{[}{#1}{#2}{]}}

\newcommand{\intervallefo}[2]{\intervalle{[}{#1}{#2}{)}}
\newcommand{\intervalleoo}[2]{\intervalle{(}{#1}{#2}{)}}

\newcommand{\ind}[1]{\mathbf{1}_{\left\lbrace #1 \right\rbrace}} 

\newcommand{\E}[1]{\ensuremath{\mathbb{E} \left[#1 \right]}}
\newcommand{\Prob}[1]{\ensuremath{\mathbb{P} \left(#1 \right)}}

\newcommand{\I}[1]{\ensuremath{\mathbbm{1}_{ \left\{ #1 \right\} }}}
\newcommand{\R}{\ensuremath{\mathbb{R}}}
\newcommand{\A}{\ensuremath{\mathbb{T}}}
\newcommand{\M}{\ensuremath{\mathbb{M}}}
\newcommand{\Z}{\ensuremath{\mathbb{Z}}}
\newcommand{\N}{\ensuremath{\mathbb{N}}}

\renewcommand{\subset}{\subseteq}
\newcommand{\convdist}{\ensuremath{\overset{\mathrm{d}}{\rightarrow}}}
\newcommand{\convprob}{\ensuremath{\stackrel{p}{\rightarrow}}}
\newcommand{\equidist}{\ensuremath{\stackrel{\mathrm{(d)}}{=}}}

\newcommand{\seg}[2]{\left[ \! \left[ #1 , #2 \right] \! \right]}
\newcommand{\oseg}[2]{\, \left] \! \left] #1 , #2 \right[ \! \right[ \,}

\newcommand{\supp}{\ensuremath{\mathrm{supp}}}

\newcommand{\cG}{\mathcal{G}}

\newcommand{\cT}{\mathcal{T}}
\newcommand{\cA}{\mathcal{A}}

\newcommand{\sG}{\mathsf{G}}
\newcommand{\sT}{\mathsf{T}}

\newtheorem{theorem}{Theorem}[section]
\newtheorem{defn}[theorem]{Definition}
\newtheorem{lem}[theorem]{Lemma}
\newtheorem{prop}[theorem]{Proposition}
\newtheorem{cor}[theorem]{Corollary}

\newtheorem{rem}[theorem]{Remark}

\date{}

\title{Stable graphs: distributions and line-breaking construction}
\author{Christina Goldschmidt\thanks{Department of Statistics and Lady Margaret Hall, University of Oxford  \newline \hspace*{0.5cm} E-mail: goldschm@stats.ox.ac.uk}\quad  B\'en\'edicte Haas\thanks{Universit\'e Sorbonne Paris Nord, LAGA, CNRS (UMR  7539) 93430 Villetaneuse, France \newline \hspace*{0.5cm} E-mail: haas@math.univ-paris13.fr} \quad and \quad Delphin S\'enizergues\thanks{Mathematics Department, University of British Columbia, 1984 Mathematics Road, Vancouver, BC V6T 1Z2  \newline \hspace*{0.5cm} E-mail: senizergues@math.ubc.ca}}

\begin{document}
\maketitle

\abstract{For $\alpha \in (1,2]$, the $\alpha$-stable graph arises as the universal scaling limit of critical random graphs with i.i.d.\ degrees having a given $\alpha$-dependent power-law tail behavior. It consists of a sequence of compact measured metric spaces (the limiting connected components), each of which is tree-like, in the sense that it consists of an $\R$-tree with finitely many vertex-identifications (which create cycles). Indeed, given their masses and numbers of vertex-identifications, these components are independent and may be constructed from a spanning $\R$-tree, which is a biased version of the $\alpha$-stable tree, with a certain number of leaves glued along their paths to the root. In this paper we investigate the geometric properties of such a component with given mass and number of vertex-identifications. We (1) obtain the distribution of its kernel and more generally of its discrete finite-dimensional marginals, and observe that these distributions are themselves related to the distributions of certain configuration models; (2) determine the distribution of the $\alpha$-stable graph as a collection of $\alpha$-stable trees glued onto its kernel; and (3) present a line-breaking construction, in the same spirit as Aldous' line-breaking construction of the Brownian continuum random tree.}

\begin{figure}[h]
	\centering
	\includegraphics[height=8.2cm]{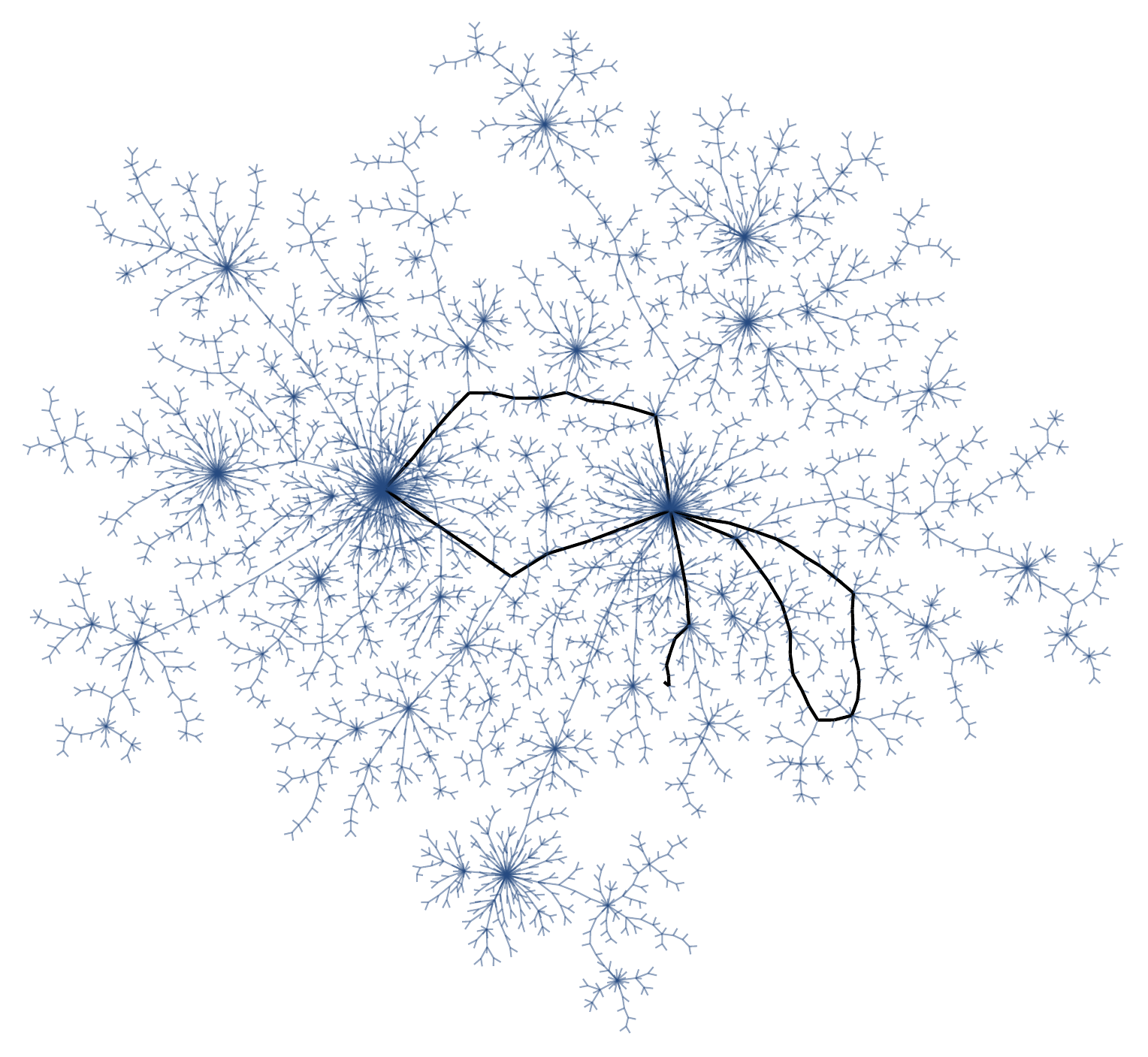}
	\caption{A simulation of a connected component of the stable graph when $\alpha=1.5$ and the surplus is 2. The cycle structure is shown in black.}\label{fig:stable bunny}
\end{figure}


\setlength\parindent{0pt}
\setlength\parskip{10pt}

\section{Introduction and main results}
\subsection{Motivation} \label{sec:motivation}

The purpose of this paper is to understand the distributional properties of the scaling limit of a critical random graph with independent and identically distributed degrees having certain power-law tail behaviour.  Let us first describe the random graph model precisely. Let $D_1, D_2, \ldots, D_n \in \mathbb N$ be independent and identically distributed random variables such that  $\E{D_1^2}<\infty$.  We build a graph with vertices labelled by $1, 2, \ldots, n$.  For $i=1,\ldots,n-1$, let vertex $i$ have degree $D_i$.  If $\sum_{i=1}^n D_i$ is even, let vertex $n$ have degree $D_n$; otherwise, let vertex $n$ have degree $D_n + 1$.  Now pick a simple graph $G_n$ uniformly at random from among those with these given vertex degrees (at least one such graph exists with probability tending to 1 as $n \to \infty$). 

Molloy and Reed \cite{MolloyReed} showed that there is a phase transition in the sizes of the connected components: if the parameter $\nu:=\mathbb E[D_1(D_1-1)]/\mathbb E[D_1]$ is larger than 1 there exists a unique \emph{giant component} of size proportional to $n$, while if $\nu$ is smaller than or equal to 1 there is no giant component. We will here tune the degree distribution so  as to be exactly at the point of the phase transition, i.e.\ $\nu = 1$.  The component size behaviour is here at its most delicate: even after performing the correct rescaling and taking a limit, there is residual randomness in the sequence of component sizes. For the questions in which we are interested, the critical case with $\E{D_1^3} < \infty$ has already been thoroughly investigated in previous work, which we summarise in Section~\ref{sec:brownian}. So we will rather assume that the degree distribution has infinite third moment and a specific power-law behaviour.  Henceforth, fix $1 <\alpha <2$ and assume that
\begin{equation}
\label{hyp:stable}
\nu = 1 \quad \text{and} \quad \Prob{D_1 = k} \sim c k^{-2-\alpha} \quad \text{as } k \to \infty,
\end{equation}
where  $c > 0$ is constant. (Note that $\nu = 1$ is equivalent to $\E{D_1^2} = 2 \E{D_1}$.)

The analogous model of a random \emph{tree} is a Galton-Watson tree with critical offspring distribution in the domain of attraction of an $\alpha$-stable law. In that case, there is a well-known scaling limit, the $\alpha$-stable tree~\cite{Du03}.  We will explore the relationship between these two models, at the level of scaling limits, in the sequel.   

It is now standard to formulate random graph scaling limits in terms of sequences of \emph{measured metric spaces}, namely metric spaces endowed with a measure.  Throughout this paper we let $(\mathscr C, \mathrm d_{\mathrm{GHP}})$ denote the set of measured isometry-equivalence classes of compact measured metric spaces equipped with the Gromov-Hausdorff-Prokhorov topology (see, for example, Section~2.1 of \cite{ABBrGoMi} for the formulation we use here)
and endow it with the associated Borel $\sigma$-algebra.  (We will often elide the difference between a measured metric space and its equivalence class but it should be understood that we are really thinking about the equivalence class.)  As we are dealing with graphs which have many components, we need a topology on sequences of (equivalence classes of) compact measured metric spaces.  For this purpose, we use the product Gromov-Hausdorff-Prokhorov topology and observe that it is standard that this yields a Polish space.

Write $C_1^n, C_2^n, \ldots$ for the vertex-sets of the components of the graph $G_n$, listed in decreasing order of size (with ties broken arbitrarily).  Set
\begin{equation}
\label{A_alpha}
A_{\alpha} =\left(\frac{c \Gamma(2-\alpha)}{\alpha(\alpha-1)}\right)^{1/(\alpha+1)}.
\end{equation}
We think of the components as metric spaces by endowing each one with a scaled version of the usual graph distance, $\mathrm d_{\text{gr}}$: let
\[
d_i^n := \frac{A_{\alpha}^2}{\mathbb E[D_1] n^{(\alpha-1)/(\alpha+1)}} \mathrm d_{\text{gr}}
\]
be the distance in $C_i^n$.  We also endow each of them with the scaled counting measure 
\[
\mu_i^n := \frac{A_{\alpha}}{\mathbb E[D_1] n^{\alpha/(\alpha+1)}} \sum_{v \in C_i^n} \delta_v.
\]
Let $\mathrm{C}_i^n = (C_i^n, d_i^n, \mu_i^n)$ be the resulting measured metric space.  We write $s(\mathrm{C}^n_i)$ for the number of \emph{surplus edges} (i.e.\ edges more than a tree) possessed by the component $\mathrm{C}^n_i$. Formally, for a connected graph $G=(V,E)$, the number of surplus edges or, more succinctly, \emph{surplus}, is defined to be $$s(G)=|E|-|V|+1.$$
The following theorem is proved in \cite{CKG17+}. 

\begin{theorem}
\label{th:C-K+G}
As $n \to \infty$, 
\[
(\mathrm{C}_1^n, \mathrm{C}_2^n, \ldots) \convdist (\mathrm{C}_1, \mathrm{C}_2, \ldots),
\]
with respect to the product Gromov-Hausdorff-Prokhorov topology, for a random sequence of measured metric spaces $(\mathrm{C}_1, \mathrm{C}_2, \ldots)$ which we call the \emph{$\alpha$-stable graph}.
\end{theorem}

(In Section~\ref{sec:brownian} below we will describe the relationship of this theorem to other work.)

Theorem~\ref{th:C-K+G} also holds in the setting of a random multigraph (i.e.\ it may contain self-loops and multiple edges) sampled from the \emph{configuration model} with i.i.d.\ degrees.  Formally, a \emph{multigraph} $G$ is an ordered pair $G = (V,E)$ where $V$ is the set of vertices and $E$ the \emph{multiset} of edges (i.e.\ elements of $\{\{u,v\},\ u\in V,\ v\in V\}$). Let $\mathrm{supp}(E)$ denote the support of $E$, i.e.\ the underlying set of distinct elements of $E$, and, for $e \in \supp(E)$, let $\mathrm{mult}(e)$ denote its multiplicity. Let $\mathrm{sl}(G)$ denote the cardinality of the multiset of self-loops.  For a vertex $v \in V$, we write $\deg(v)$ for its degree, or $\deg_G(v)$ if there is potential ambiguity over which graph we are looking at. The surplus is still defined to be $s(G)=|E|-|V|+1$, where we emphasise  that $|E| = \sum_{e \in \supp(E)} \mathrm{mult}(e)$.  Let us briefly explain the set-up of the configuration model for deterministic degrees $d_1, d_2, \ldots, d_n$ with even sum. (The configuration model was introduced in varying degrees of generality in \cite{BenderCanfield, Bollobas, Wormald}.  We refer to Chapter~7 of the recent book of van der Hofstad~\cite{RvdH} for the proofs of the claims made in this paragraph.)  To vertex $i$ we assign $d_i$ half-edges, for $1 \le i \le n$.  We give the half-edges an arbitrary labelling (so that we may distinguish them) and then choose a matching of the half-edges uniformly at random.  Two matched half-edges form an edge of the resulting structure, which is a multigraph.   Then for a particular multigraph $G$ with degrees $d_1, d_2, \ldots, d_n$, the probability that the configuration model generates $G$ is 
\begin{equation} \label{eqn:config}
\frac{\prod_{i=1}^n d_i!}{(\sum_{i=1}^n d_i - 1)!! \ 2^{\mathrm{sl}(G)} \prod_{e \in \supp(E)} \mathrm{mult}(e)!},
\end{equation}
where $a!!$ denotes the double factorial of $a$.  From this expression, it is easy to see that if there exists at least one simple graph with degrees $d_1, d_2, \ldots, d_n$ then conditioning the multigraph to be simple yields a \emph{uniform} graph with the given degree sequence. We are interested in the setting where the degrees are random variables $D_1, D_2, \ldots, D_n$ satisfying the conditions (\ref{hyp:stable}) (with the small modification mentioned above to make the sum of the degrees even).  In this case, there exists a simple graph with these degrees with probability tending to 1 as $n \to \infty$, which enables us to convert results for the configuration model into results for the uniform random graph with given degree sequence; in the setting of Theorem~\ref{th:C-K+G} the conditioning turns out not to affect the result.

The $\alpha$-stable graph is constructed using a spectrally positive $\alpha$-stable L\'evy process; we give the details, which are somewhat involved, in Section~\ref{sec:SLconfiguration}. For $i \ge 1$, write $\mathrm{C}_i = (C_i, d_{C_i}, \mu_{C_i})$, $i \ge 1$.   These measured metric spaces are \emph{$\R$-graphs} in the sense of \cite{ABBrGoMi} i.e.\ they are locally $\R$-trees, but may also possess cycles. It is possible to make sense of the surpluses of the limiting components, for which we write $s(\mathrm{C}_i)$, $i \ge 1$. It is a consequence of Theorem~\ref{th:C-K+G} that 
\begin{equation} \label{eqn:convsizes}
\frac{A_{\alpha}}{\mathbb E[D_1] n^{\alpha/(\alpha+1)}}(|C_1^n|, |C_2^n|, \ldots) \convdist (\mu_{C_1}(C_1), \mu_{C_2}(C_2), \ldots)
\end{equation}
in the product topology (in fact, this convergence can be shown to occur in $\ell^2$; see Proposition 5.6 of \cite{CKG17+}), jointly with the convergence in the sense of the product topology
\begin{equation} \label{eqn:convsurplus}
(s(\mathrm{C}_1^n), s(\mathrm{C}_2^n), \ldots) \convdist (s(\mathrm{C}_1), s(\mathrm{C}_2), \ldots)
\end{equation}
for the sequences of surplus edges. The joint law of $(\mu_{C_1}(C_1), \mu_{C_2}(C_2), \ldots)$ and  $(s(\mathrm{C}_1), s(\mathrm{C}_2), \ldots)$ is explicit in terms of the underlying $\alpha$-stable L\'evy process; see Section~\ref{sec:SLconfiguration}.  Moreover, Theorem 1.2 of \cite{CKG17+} shows that the limiting components $(\mathrm{C}_1,\mathrm{C}_2,\ldots)$ are conditionally independent given $(\mu_{C_1}(C_1), \mu_{C_2}(C_2), \ldots)$ and $(s(\mathrm{C_1}), s(\mathrm{C}_2), \ldots)$, with distributions coming from a collection of fundamental building-blocks: there exist random measured metric spaces $(\mathcal{G}^s,d^s, \mu^s)$, $s \geq 0$, where $\mu^s$ is a probability measure, such that, for all $i$, given $\mu_{C_i}(C_i)$ and $s(\mathrm{C_i})$, we have
\[
\big(C_i, d_{C_i}, \mu_{C_i} \big)\equidist \big(\mathcal{G}^{s(\mathrm{C}_i)}, \mu_{C_i}(C_i)^{1-1/\alpha} \cdot d^{s(\mathrm{C}_i)}, \mu_{C_i}(C_i) \cdot \mu^{s(\mathrm{C}_i)}\big).
\]
For $s=0$, $(\mathcal{G}^s,d^s, \mu^s)$ is simply the standard rooted $\alpha$-stable tree, the definition of which is recalled in Section~\ref{sec:stable}.  Informally, for $s \ge 1$, $(\mathcal{G}^s,d^s, \mu^s)$, is constructed by randomly choosing $s$ leaves in an $s$-biased version of this $\alpha$-stable tree, and then gluing them to randomly-chosen branch-points along their paths to the root, with probabilities proportional to the ``local time to the right" of the branch-points. (We will define these quantities in the sequel.) We will often think of the resulting $\mathbb R$-graph $\mathcal G^s$ as being rooted; in this case, the root is simply inherited from that of the $s$-biased $\alpha$-stable tree. The measure $\mu^s$ on $\mathcal{G}^s$ is then the probability measure inherited from the $s$-biased $\alpha$-stable tree.  We will often abuse notation and simply write $\mathcal{G}^s$ in place of $(\mathcal{G}^s, d^s, \mu^s)$.  For $a > 0$, we will also write $a \cdot \mathcal{G}^s$ to denote the same measured metric space with all distances scaled by $a$, i.e.\ $(\mathcal{G}^s, a  d^s, \mu^s)$.

In order to understand the geometric properties of the $\alpha$-stable graph, it therefore suffices to consider the measured metric spaces 
$$\mathcal{G}^s, s \ge 0.$$
We will call $\mathcal G^s$ the \emph{connected $\alpha$-stable graph with surplus $s$.}
Let us note immediately that $\mathcal{G}^s$ naturally inherits the Hausdorff dimension of the $\alpha$-stable tree and that, therefore, 
$$
\mathrm{dim}_{\text{H}}(\mathcal{G}^s)=\frac{\alpha}{\alpha-1} \quad \text{a.s.}
$$

Like a connected combinatorial graph, the $\R$-graph $\mathcal{G}^s$ may be viewed as a cycle structure to which pendant subtrees are attached. Let $\mathcal K^s$ be the image after the gluing procedure of the subtree spanned by the $s$ selected leaves and the root of the $s$-biased version of the $\alpha$-stable tree.  (When $s=0$, we use the convention that $\mathcal K^s$ is the empty set.) The space $\mathcal{K}^s$ encodes the rooted cycle structure of $\mathcal{G}^s$.  We refer to it as the \emph{continuous kernel} because it is a continuous analogue of the usual graph-theoretic notion of a kernel (except that it is rooted at a vertex of degree 1).  We will think of it as a rooted multigraph which is endowed with real-valued edge-lengths, and write $\mathsf{K}^s$ for the rooted multigraph \emph{without} the edge-lengths, which we call the \emph{discrete kernel}.

In order to better understand the structure of the $\R$-graph $\mathcal{G}^s$, we will approximate it by a sequence $(\mathcal{G}^s_n)_{n \ge 0}$ of multigraphs with edge-lengths, starting from the continuous kernel, $\mathcal G_{0}^s=\mathcal K^s$. Consider an infinite sample of leaves from $\mathcal{G}^s$, labelled $1,2,\ldots$.  For each $n \in \mathbb N$, let $\mathcal G_{n}^s$ be the connected subgraph of $\mathcal{G}^s$ consisting of the union of the kernel $\mathcal K^s$ and the paths from the $n$ first leaves to the root.   These are the $\R$-graph analogues of Aldous' \emph{random finite-dimensional marginals} for a continuum random tree.  For brevity, we will call them the \emph{marginals} of $\mathcal{G}^s$. In Lemma~\ref{lem:density} below, we note that $\mathcal{G}^s$ can be recovered as the completion  of $\cup_{n \geq 0} \mathcal G_{n}^s$.  We will also make extensive use of the discrete counterparts of the $\mathcal{G}_n^s$. For $n\geq 0$, let $\mathsf{G}_{n}^s$ be the \emph{combinatorial shape} of $\mathcal{G}_{n}^s$ (i.e.\ ``forget the edge-lengths", so as to obtain a finite graph with surplus $s$ and no vertices of degree 2 -- see (\ref{def:shape}) for a formal definition in the framework of trees that adapts immediately to our graphs), so that $\mathsf{K}^s=\mathsf{G}_{0}^s$. Note that the root vertex has degree 1 in all of these graphs. When $s\geq 2$, we can \emph{erase the root} in the discrete kernel (formally, we remove the root and the adjacent edge, and if this creates a vertex of degree 2 we erase it) to obtain a multigraph that we denote by $\mathsf G_{-1}^{s}$.

\subsection{Main results}
\label{sec:mainresults}

Throughout this section, we fix the surplus $s \in \Z_+$.  

Our first main results characterise the joint distributions of the discrete marginals $(\mathsf{G}_n^s)_{n \ge 0}$.  This family of random multigraphs has particularly attractive properties: for fixed $n$, the graph $\mathsf{G}_n^s$ has the distribution of a certain conditioned configuration model with i.i.d.\ random degrees, with a particular canonical degree distribution.  Moreover, as a process, $(\mathsf{G}_n^s)_{n \ge 0}$ evolves in a Markovian manner according to a simple recursive construction which is a version of Marchal's algorithm~\cite{Marchal} for building the marginals of the stable tree, $(\mathsf{G}_n^0)_{n \ge 0}$.  Although $\mathcal{G}^s$ is constructed from a biased version of the $\alpha$-stable tree, we emphasise that it was not at all obvious to us \emph{a priori} that Marchal's algorithm would generalise in this way.

An advantage of this recursive construction is that it has many urn models embedded in it, which enable us to get at different aspects of $\mathcal{G}^s$ easily.  We provide two different constructions of $\mathcal{G}^s$, which rely on relatively simple random building blocks.  The distributions of these building blocks (Beta, generalised Mittag-Leffler, Dirichlet and Poisson-Dirichlet) are defined in Section~\ref{sec:urns}, where we also recall various of their standard properties and discuss their relationships to urns.  Our two constructions are as follows.
\begin{enumerate}
\item The first takes a collection of i.i.d.\ $\alpha$-stable trees which are randomly scaled and then glued onto $\mathsf{K}^s$ in such a way that each edge of $\mathsf{K}^s$ is replaced by a tree with two marked points, and such that every vertex of $\mathsf{K}^s$ acquires a (countable) collection of pendant subtrees.
\item The second starts by replacing the edges of the kernel by line-segments of lengths with a given joint distribution, and then proceeds by recursively gluing a countable sequence of segments of random lengths onto the structure.  We call this a \emph{line-breaking construction} and obtain the limit space in the end by completion.
\end{enumerate}

These constructions generalise, in a natural way, the distributional properties and line-breaking construction proved in \cite{ABBrGo} for the components of the \emph{Brownian graph}, a term we use to mean the common scaling limit of the critical Erd\H{o}s-R\'enyi random graph~\cite{ABBrGo12} and the critical random graph with i.i.d.\ degrees having a finite third moment~\cite{BS16+} as well as various other models (see Section~\ref{sec:brownian}).  We emphasise, however, that the proofs in the stable setting are much harder, essentially due to the added complication of dealing with L\'evy processes rather than just Brownian motion.
Our line-breaking construction is the graph counterpart of the line-breaking construction of the stable trees given in \cite{GHlinebreaking}.

\subsubsection{The discrete marginals of $\mathcal{G}^s$}
\label{sec:marginalsintro}

We can recover the measured metric space $\mathcal G^s$ from the discrete marginals $\mathsf{G}_{n}^s$ by equipping them with the graph distance and the uniform distribution on their leaves, as follows.
\begin{prop}
\label{prop:approx}
$$\frac{\mathsf G_n^s}{n^{1-1/\alpha}} \underset{n \rightarrow \infty}{\overset{\mathrm{a.s.}}\longrightarrow} \alpha \cdot \mathcal G^s$$
for the Gromov-Hausdorff-Prokhorov topology.
\end{prop}
This generalises a result which says that the $\alpha$-stable tree is the (almost sure) scaling limit of its discrete marginals, see \cite{Marchal, CurienHaas}. The proof is given in Section~\ref{sec:approx}.

For any multigraph $G=(V,E)$, recall that we let $\mathrm{sl}(G)$ denote its number of self-loops, and for an element $e \in \supp(E)$, we let $\mathrm{mult}(e)$ denote its multiplicity. Let $I(G) \subset V$ denote the set of internal vertices of $G$. We say that a permutation $\tau$ of the set $I(G)$ is a \emph{symmetry} of $G$ if, after having extended $\tau$ to the identity function on the leaves, $\tau$ preserves the adjacency relations in the graph and for all $u,v \in V$, the edges $\{u,v\}$ and $\{\tau(u),\tau(v)\}$ have the same multiplicity. We let $\mathrm{Sym}(G)$ denote the set of symmetries of $G$. For $n \ge 0$, let $\mathbb M_{s,n}$ be the set of connected multigraphs with $n+1$ labelled leaves,  surplus $s$ and no vertices of degree 2. (Observe that the internal vertices are not labelled.) When $s\geq 2$, let $\mathbb M_{s,-1}$ be the set of unlabelled connected multigraphs with surplus $s$ and minimum degree at least 3.
Finally, let us define a sequence of weights by
\begin{equation}
\label{Marchal's weights}
w_0:=1, \quad w_1:=0, \quad w_2:=\alpha-1,\quad w_k:=(k-1-\alpha)\dots(2-\alpha)(\alpha-1), \quad \text{for $k\geq 3$.}
\end{equation}

Viewing the root as a leaf with label 0, we note that  $\mathsf{G}_{n}^s$ is an element of $\mathbb{M}_{s,n}$.  We can now describe the distributions of the random multigraphs $\mathsf{G}_{n}^s$.

\begin{theorem}
\label{th:distr_marginals}
Let $n\geq0$. For every connected multigraph $G=(V,E) \in \mathbb M_{s,n}$, 
	\[\Pp{\mathsf{G}_{n}^s=G}\propto\frac{\prod_{v\in I(G)} w_{\deg(v)-1}}{\abs{\mathrm{Sym}(G)}2^{\mathrm{sl}(G)}\prod_{e\in \supp(E)}\mathrm{mult}(e)!}.\]
This, in particular, gives the distribution of the kernel $\mathsf{K}^s$ when $n=0$. 
When $s\geq 2$, this expression also gives the distribution of $\mathsf{G}_{-1}^s$ on $\mathbb M_{s,-1}$.
\end{theorem}

This result is proved in Section~\ref{sec:marginals}. To illustrate it, in Figure \ref{fig:table} we give the distribution of the kernel explicitly in the case $s=2$ and $\alpha=\frac{5}{4}$. 

\begin{figure}
{\fontsize{9}{9}\selectfont 
\begin{center}
\renewcommand{\arraystretch}{2}
\begin{tabular}{|c||c|c|c|c|c|c|c|} 
        \hline
	Graph $G \in \mathbb M_{2,0}$ & {\parbox[c][3cm]{1.3cm}{\includegraphics[width=1.3cm,page=1]{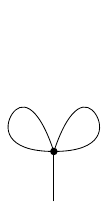}}} & \parbox[c]{1.3cm}{\includegraphics[width=1.3cm,page=2]{kernels.pdf}} & \parbox[c]{1.3cm}{\includegraphics[width=1.3cm,page=3]{kernels.pdf}} & \parbox[c]{1.3cm}{\includegraphics[width=1.3cm,page=4]{kernels.pdf}} & \parbox[c]{1.3cm}{\includegraphics[width=1.3cm,page=5]{kernels.pdf}} & \parbox[c]{1.3cm}{\includegraphics[width=1.3cm,page=6]{kernels.pdf}} & \parbox[c]{1.3cm}{\includegraphics[width=1.3cm,page=7]{kernels.pdf}}  \\ 
	\hline
	\hline
	$\displaystyle \mathrm{sl}(G)$ & 2 &1 &0 &0 &2 &1 &2 \\
	\hline 
	$\displaystyle\prod_{v\in I(G)}w_{\deg(v)-1} $ & $\frac{21}{64}$& $\frac{3}{64}$ & $\frac{3}{64}$ &$\frac{1}{64} $ & $\frac{3}{64}$ &$\frac{1}{64} $&$\frac{1}{64}$ \\
	 {{\tiny ($\alpha=5/4$)}} &  & & & & & & \\
	 \hline
	$\displaystyle\prod_{e\in E(G)}\mathrm{mult}(e)!$ &2  &2  & 6 & 2 & 1 & 2 & 1 \\
	\hline
	$\displaystyle \abs{\mathrm{Sym}(G)}$ & 1 &1  &1  & 2 &1  &1 &2  \\
	\hline
	$\displaystyle \Pp{\mathsf{K}^2=G}$ & $ \frac{1}{2} $&$ \frac{1}{7} $&$ \frac{2}{21}$&$ \frac{1}{21} $&$ \frac{1}{7}$ &$ \frac{1}{21}$ & $\frac{1}{42}$\\
	{{\tiny ($\alpha=5/4$)}} &  & & & & & & \\
	\hline
	$\displaystyle \Pp{\mathsf{K}^2_{\mathrm{Br}}=G}$ & $ 0 $&$ 0 $&$ 0$&$ \frac{2}{5} $&$ 0$ &$ \frac{2}{5}$ & $\frac{1}{5}$\\
	\hline
\end{tabular}
\end{center}
\caption{\fontsize{10}{10}\selectfont  The possible kernels for $s = 2$ with their probabilities for $\alpha = 5/4$ (given in the penultimate line).  For comparison, the last line gives the distribution of the kernel of the connected Brownian graph with surplus 2.}
\label{fig:table}}
\end{figure}

Comparing the form of the distribution of $\mathsf{G}_n^s$ with (\ref{eqn:config}) suggests a connection with a conditioned configuration model.  To make this precise, let $D^{(\alpha)}$ be a random variable on $\mathbb N$ with distribution
\begin{equation}
\label{def:Dalpha}
\mathbb P(D^{(\alpha)}=k) = \frac{2(1+\alpha)\alpha}{\alpha^2+\alpha+2} \cdot \frac{w_{k-1}}{k!}, \quad k \geq 2, \quad \text{and} \quad \mathbb P(D^{(\alpha)}=1)= \frac{2(1+\alpha)}{\alpha^2+\alpha+2}.
\end{equation}
Observe that $\mathbb{P}(D^{(\alpha)} = 2) = 0$.  We will verify in Section~\ref{sec:configembedded} that this indeed defines a probability measure which, moreover, satisfies the conditions (\ref{hyp:stable}).  Consider now the following particular instance of the configuration model.  We fix $n \ge 0$ and $m \ge n+1$ (include the case $n=-1$ if $s\geq 2$), take vertices labelled $0, 1,  \ldots, m-1$ to have i.i.d.\ degrees distributed according to $D^{(\alpha)}$ and write $\mathsf C_{n,m}^s$ for the resulting configuration multigraph conditioned to be in $\mathbb M_{n,s}$, after having forgotten the labels $n+1, n+2, \ldots, m-1$.
\begin{cor}
\label{cor:identification}
The random multigraph $\mathsf{G}_{n}^s$ conditioned to have $m$ vertices has the same law as $\mathsf C_{n,m}^s$.
\end{cor}

This again generalises the analogous result for the $\alpha$-stable tree: the combinatorial shape of the subtree obtained by sampling $n \geq 0$ leaves and the root is distributed as a planted (i.e.\ with a root of degree 1) non-ordered version of a Galton-Watson tree conditioned to have $n$ leaves, whose offspring distribution $\eta_\alpha$ has probability generating function  $z+\alpha^{-1}(1-z)^{\alpha}$. There is, of course, a connection between $D^{(\alpha)}$ and $\eta_\alpha$: if we let $\hat{D}^{(\alpha)}$ denote the size-biased version
$$\mathbb P(\hat{D}^{(\alpha)}=k):=\frac{k \mathbb P(D^{(\alpha)}=k)}{\mathbb E\big[D^{(\alpha)}\big]}, \quad k \geq 1,$$
then $\hat{D}^{(\alpha)}-1$ is distributed as $\eta_\alpha$. See Section~\ref{sec:configembedded}.

In fact, we may think of the configuration multigraph with i.i.d.\ degrees distributed as $D^{(\alpha)}$ as, in some sense, the \emph{canonical} model in the universality class of the stable graph. For this model, the law of a component conditioned to have $n+1$ leaves and surplus $s$ is exactly the same as the corresponding discrete marginal for its scaling limit, and there exists a coupling for different $n$ which is such that we get almost sure (rather than just distributional) convergence, on rescaling, to the connected $\alpha$-stable graph with surplus $s$.

We are also able to understand the joint distribution of the graphs $\mathsf{G}_{n}^s,n\geq 0$ (again, include the case $n=-1$ when $s\geq 2$): they evolve according to a multigraph version of Marchal's algorithm~\cite{Marchal} for the discrete marginals of a $\alpha$-stable tree.  Let us define a step in the algorithm. Take a multigraph $G=(V,E)\in \M_{s,n}$. Declare every edge to have weight $\alpha-1$, every internal vertex $u\in I(G)$ to have weight $\deg_G(u)-1-\alpha$ and every leaf to have weight 0. Then the total weight of $G$ is
\begin{equation}
\label{identity:degrees}
\sum_{u \in I(G)}(\deg_G(u)-1-\alpha)+ (\alpha-1) \cdot |E|=\alpha(s+n)+s-1, 
\end{equation}
which depends only on the surplus and the number of leaves of the graph. 
We use the term \emph{edge-leaf} to mean an edge with a leaf at one of its end-points.  
Choose an edge/vertex with probability proportional to its weight. Then
\begin{itemize}
	\item if it is a vertex, attach a new edge-leaf where the leaf has label $n+1$ to this vertex,
	\item if it is an edge, attach a new edge-leaf where the leaf has label $n+1$ to a newly created vertex which splits the edge into two.
\end{itemize}
We say that a sequence of graphs evolves according to Marchal's algorithm if it is Markovian and the transitions are given by one step of Marchal's algorithm.

\begin{theorem} \label{thm:MarchalAlg}
For $s \ge 0$, the sequence $(\mathsf{G}_n^s)_{n \ge 0}$ evolves according to Marchal's algorithm. For $s \geq 2$, more generally, the sequence $(\mathsf{G}_n^s)_{n \ge -1}$ evolves according to Marchal's algorithm.
\end{theorem}

See Section~\ref{sec:Marchal} for a proof. We now turn to our constructions of the limit object $\mathcal{G}^s$.

\subsubsection{Construction 1: from randomly scaled stable trees glued to the kernel} 
\label{cons1:gluing} 
Given a connected multigraph $G \in \mathbb M_{s,0}$, with $k$ edges and $k-s$ internal vertices having degrees $d_1,\ldots,d_{k-s}$, consider independent random variables
\begin{equation}
\label{distr1}
(M_1,\ldots,M_{2k-s}) \sim \mathrm{Dir} \bigg(\underbrace{\frac{\alpha-1}{\alpha},\ldots,\frac{\alpha-1}{\alpha}}_{k},\frac{d_1-1-\alpha}{\alpha},\ldots, \frac{d_{k-s}-1-\alpha}{\alpha}\bigg)
\end{equation}
and, for $1 \leq i \leq k-s$,
\begin{equation}
\label{distr2}
(\Delta_{i,j}, j \geq 1) \sim \mathrm{PD}\left(\frac{1}{\alpha},\frac{d_i-1-\alpha}{\alpha}\right),
\end{equation}
where $\mathrm{Dir}(a_1,\ldots,a_n)$ denotes the Dirichlet distribution on the $(n-1)$-dimensional simplex, with parameters $a_1>0,a_2>0,\ldots,a_n>0$, and $\mathrm{PD}(a,b)$ denotes the Poisson-Dirichlet distribution on the set of positive decreasing sequences with sum 1, with parameters $a>0,b>0$.
 
Given all of these random variables, consider independent $\alpha$-stable trees $\mathcal T_{\ell}$, $\mathcal T_{i,j}$, where $\mathcal T_{\ell}$ has mass $M_{\ell}$ and $\mathcal T_{i,j}$ has mass $M_{i+k} \cdot \Delta_{i,j}$, with $1 \leq \ell \leq k, 1 \leq i \leq k-s, j \geq 1$. For each $\ell$ let $\rho_\ell$ denote the root of $\mathcal T_{\ell}$ and $L_{\ell}$ be a uniform leaf. Similarly, let $\rho_{i,j}$ denote the root of the tree $\mathcal T_{i,j}$ for each $i,j$.
Then denote by $ e_1,\ldots, e_{k}$ the edges of $G$ in arbitrary order, with, say, $e_i=\{x_i,y_i\}$, and by $v_1,\ldots, v_{k-s}$ the internal vertices of $G$, also in arbitrary order.
Finally, let $\mathcal G(G)$ be the $\R$-graph obtained by:
\begin{enumerate}
\item[$\bullet$] replacing the edge $\{x_{\ell},y_{\ell}\}$ with the tree $\mathcal T_{\ell}$, identifying $\rho_{\ell}$ with $x_{\ell}$ and $L_{\ell}$ with $y_{\ell}$, for each $1\leq \ell \leq k$, 
\item[$\bullet$] gluing to the vertex $v_i$ the collection of stable trees $\mathcal T_{i,j},j\geq 0$, by identifying all the roots $\rho_{i,j}$ to $v_i$ (this gluing a.s.\ gives a compact metric space, see Section~\ref{sec:distrGs}), for each $1\leq i \leq k-s$.
\end{enumerate}
On an event of probability one the graph $\mathcal G(G)$ is therefore compact, and is naturally endowed with the probability measure induced by the rescaled probability measures on the $\alpha$-stable trees $\mathcal T_{\ell}$, $\mathcal T_{i,j}$, $i,j,\ell \in \mathbb N$.  We view it as a random variable in $(\mathscr C,\mathrm d_{\mathrm{GHP}})$.

\medskip

\begin{theorem}
\label{thm:distrGs} 
Given the random kernel $\mathsf{K}^s$, let $\mathcal G(\mathsf{K}^s)$ be the graph constructed above by gluing $\alpha$-stable trees along the edges and vertices of $\mathsf{K}^s$. Then  
$$
\mathcal G^s \ \overset{\mathrm{d}}= \ \mathcal G(\mathsf{K}^s),
$$
as random variables in $(\mathscr C,\mathrm d_{\mathrm{GHP}})$.
\end{theorem}

We prove Theorem~\ref{thm:distrGs} in Section~\ref{sec:distrGs} via the recursive construction of the discrete graphs  $\mathsf G_n^s,n \geq 0$. As a byproduct of the proof,  
we obtain the distribution of the continuous marginals  $\mathcal G_n^s$, which may be viewed as $\mathsf G_n^s$ with random edge-lengths. In particular, when $n=0$, we obtain the distribution of the continuous kernel $\mathcal K^s$.  

\begin{prop}
\label{cor:lengthskernel}
For $n \geq 0$, given $\mathsf G_n^s=(V,E)$, let $(L(e), e \in E)$ be the lengths of the corresponding edges in $\mathcal G_n^s$, in arbitrary order. Then,
$$
\left(\alpha \cdot L(e), e \in E \right)$$ 
is distributed as the product of three independent random variables:
\begin{equation}
\label{ref:BMLD}
\mathrm{Beta}\left(|E|,\frac{(n+s)\alpha+s-1}{\alpha-1}-|E|\right) \cdot \mathrm{ML}\left(1-\frac{1}{\alpha}, \frac{(n+s)\alpha+s-1}{\alpha} \right) \cdot \mathrm{Dir}(1,\ldots,1).
\end{equation}
\end{prop}
Here, $\mathrm{ML}(\beta,\theta)$ denotes the generalised Mittag-Leffler distribution with parameters $0 < \beta < 1$ and $\theta > -\beta$.  

\subsubsection{Construction 2: line-breaking} 
\label{cons2:LB}
Various prominent examples of random metric spaces may be obtained as the limit of a so-called line-breaking procedure that consists in gluing recursively segments of random lengths -- or more complex measured metric structures -- to obtain a growing structure. The most famous is the line-breaking construction of the Brownian continuum random tree discovered by Aldous in \cite{AldousCRT1}. We refer to \cite{ABBrGo, CH14, GHlinebreaking, RW16,  Sen17, Sen19+_scaling} for other models studied since then.

The $\R$-graph $\mathcal G^s$ may also be constructed in such a way, starting from its kernel. 
This construction makes use of an increasing $\mathbb R_+$-valued Markov chain $(R_n)_{n \geq 1}$ which is characterized by the following two properties for each $n \geq 1$: 
$$R_n \sim \mathrm{ML}\left(1-\frac{1}{\alpha},\frac{n\alpha+(s-1)}{\alpha}\right) \quad \quad \text{and} \quad \quad R_n=R_{n+1} \cdot B_n$$
where $B_n \sim \mathrm{Beta}\left(\frac{(n+1)\alpha+s-2}{\alpha-1},\frac{1}{\alpha-1}\right)$ is a random variable independent of $R_{n+1}$.
(An explicit construction of this Markov chain is given e.g.\ in \cite[Section~1.2]{GHlinebreaking}. Note that similar Markov chains arise in the scaling limits of several stochastic models, see \cite{Lancelot15, Sen19+_geometry}.)
 
For the moment, assume that $s\geq 1$. Suppose we are given $\mathsf{K}^s$ with, say, $k$ edges and internal vertices $v_1,\ldots, v_{k-s}$ having degrees $d_1,\ldots,d_{k-s}$ respectively (the order of labelling is unimportant). We first perform an initialisation step: independently of the Markov chain $(R_n)_{n \geq 1}$,
\begin{itemize}
\item sample 
\[
(\Theta_1,\ldots,\Theta_{2k-s}) \sim \mathrm{Dir}\Big(\underbrace{1,\ldots,1}_k,\frac{d_1-1-\alpha}{\alpha-1},\ldots, \frac{d_{k-s}-1-\alpha}{\alpha-1}\Big);
\]
\item assign the lengths $R_s\cdot \Theta_1,\ldots,R_s \cdot \Theta_k$ to the $k$ edges of  $\mathsf{K}^s$ (the order is again unimportant); viewing the edges as closed line-segments, this gives a metric space that we denote $\mathcal H^s_0$, with $k-s$ branch-points (i.e.\ vertices of degree at least 3) labelled $v_1,\ldots,v_{k-s}$;
\item let $\eta_0:=\lambda_{\mathcal H^s_0}+\sum_{i=1}^{k-s} (R_s \cdot \Theta_{k+i})\delta_{v_i}$, where $\lambda_{\mathcal H^s_0}$ denotes the Lebesgue measure on $\mathcal H^s_0$. 
\end{itemize} 

We now build a growing sequence of measured metric spaces $(\mathcal H^s_{n},\eta_n)_{n \geq 0}$, starting from $(\mathcal H^s_{0},\eta_0)$.  Recursively,
\begin{itemize}
\item select a point $v$ in $\mathcal H^s_{n}$ with probability proportional to $\eta_n$;
\item attach to $v$ a new closed line-segment $\sigma$ of length $(R_{n+s+1}-R_{n+s}) \cdot \beta_n$,  where $\beta_n$ has a  \linebreak $\mathrm{Beta}(1,(2-\alpha)/(\alpha-1))$-distribution and is independent of everything constructed until now; this gives $\mathcal H_{n+1}^s$;
\item let $\eta_{n+1}:=\eta_n+ (R_{n+s+1}-R_{n+s}) \cdot (1-\beta_n) \delta_{v}+\lambda_{\sigma}$, where  $\lambda_{\sigma}$ denotes the Lebesgue measure on $\sigma$.
\end{itemize}
When $s=0$ the construction works similarly except that the initialization starts at $n=1$ with $\mathcal H^0_1$ taken to be a closed segment of length $R_1$, equipped with the Lebesgue measure denoted by $\eta_1$. We have the following result, which is proved in Section~\ref{sec:LB}.

\bigskip

\begin{theorem}
\label{th:linebreaking}
The sequence $(\mathcal H^s_n, n\geq 0)$ is distributed as $(\mathcal G^s_n, n\geq 0)$. In consequence, the graph $\mathcal H^s_n$, endowed with the uniform probability on its set of leaves, converges almost surely for the Gromov-Hausdorff-Prokhorov topology to a random compact measured metric space distributed as $\mathcal G^s$. In particular, $\overline{\cup_{n\geq 0} \mathcal H^s_{n}}$ is a version of $\mathcal G^s$.
\end{theorem}

\bigskip

\begin{rem}
We adopt a ``discrete'' approach to proving Theorems~\ref{thm:distrGs} and \ref{th:linebreaking}; in other words, we make use of Marchal's algorithm and the fact that it gives us a sequence of approximations which, on rescaling, converge almost surely to the connected $\alpha$-stable graph with surplus $s$.  An alternative approach should be possible, whereby one would work directly in the continuum, but it is far from clear to us that it would be any simpler to implement.
\end{rem}

\subsection{The finite third moment case, and other related work}
\label{sec:brownian}

The case where 
$$
\E{D_1^2}=2 \E{D_1} \quad \text{and} \quad \E{D_1^3}<\infty
$$
has already been well-studied. In particular, when $\mathbb{P}(D_1 = 2) < 1$, 
if we let $\beta=\E{D_1(D_1-1)(D_1-2)}$ then Theorem~\ref{th:C-K+G} holds with 
$\alpha=2$ if we rescale the counting measure on each component by \linebreak $\beta^{-1}  \E{D_1}n^{-2/3} $ and the graph distances by $\beta \E{D_1}^{-1} n^{-1/3}$. The limiting graphs are constructed similarly to ours but using a standard Brownian motion instead of a spectrally positive $\alpha$-stable L\'evy process (with the small variation that $\beta$ appears in the change of measure). See \cite[Theorem~2.4 and Construction 3.5]{BS16+} and also \cite{CKG17+} for more details.
This \emph{Brownian graph} first appeared as the scaling limit of the critical Erd\H{o}s-R\'enyi random graph \cite{ABBrGo12} and is now known to be the universal scaling limit of various other critical random graph models. Precise analogues of our main results were already known in this Brownian case (except for Theorem~\ref{thm:MarchalAlg}).

It follows from the properties of Brownian motion that the branch-points in $\mathcal G_{\mathrm{Br}}^s$, the connected Brownian graph with surplus $s$, are then all of degree 3. Its discrete kernel $\mathsf K_{\mathrm{Br}}^s$ is therefore a $3$-regular planted multigraph, whose distribution is given below.

\begin{theorem}[{\cite[Figure (2)]{ABBrGo}} and {\cite[Theorem~7]{JKLP93}}]
\label{th:distr_brow_ker}
For a connected $3$-regular planted multigraph $G$ with surplus $s$,
	\[\Pp{\mathsf{K}_{\mathrm{Br}}^s=G}\propto\frac{1}{\abs{\mathrm{Sym}(G)}2^{\mathrm{sl}(G)}\prod_{e\in \supp(E(G))}\mathrm{mult}(e)!}.\]
\end{theorem}

(In the references given, the kernel is taken to be labelled and unrooted, but the labelling can be removed simply at the cost of the factor of $|\mathrm{Sym}(G)|^{-1}$ appearing in the above expression, and the root can be removed as detailed above.) See  Figure~\ref{fig:table} for numerical values when $s=2$. Note that the formula above corresponds to that of Theorem~\ref{th:distr_marginals} when $n=0$ and $\alpha=2$ since then 
\[
w_0=w_2=1 \quad  \text{and} \quad w_i=0 \quad \text{for all other indices } i.
\]
In fact, our proofs in Section~\ref{sec:marginals} can be adapted to recover this case and more generally to obtain the joint distribution of the marginals $\mathsf G_{n,\mathrm{Br}}^s$ via a recursive construction which is particularly simple in this case: starting from the kernel $\mathsf K_{\mathrm{Br}}^s$, at each step a new edge-leaf is attached to an edge chosen uniformly at random from among the set of edges of the pre-existing structure.  (For $s=0$, this is R\'emy's algorithm~\cite{Remy} for generating a uniform binary leaf-labelled tree.)  After $n$ steps, this gives a version of $\mathsf G_{n,\mathrm{Br}}^s$, whose distribution is specified below.

\begin{prop}
\label{th:distr_brow_mar} For every multigraph $G \in \mathbb M_{s,n}$ with internal vertices all of degree 3,
	\[\Pp{\mathsf{G}_{n,{\mathrm{Br}}}^s=G}\propto\frac{1}{\abs{\mathrm{Sym}(G)}2^{\mathrm{sl}(G)}\prod_{e\in \supp(E(G))}\mathrm{mult}(e)!}.\]
\end{prop}

As in the stable cases, these distributions are connected to configuration multigraphs. Indeed, let $D^{(\mathrm{Br})}$ denote a random variable with distribution $$\mathbb P(D^{({\mathrm{Br}})}=1)=3/4 \quad  \text{and} \quad \mathbb P(D^{(\mathrm{Br})}=3)=1/4.$$ Consider then the following particular instance of the configuration model.  We fix $n \ge 0$, $m \ge n+1$ and take vertices labelled $0, 1,  \ldots, m-1$ to have i.i.d.\ degrees distributed according to $D^{(\mathrm{Br})}$. We then write $\mathsf C_{n,m}^s$ for the resulting configuration multigraph conditioned to be in $\mathbb M_{s,n}$, after having forgotten the labels $n+1, n+2, \ldots, m-1$. 

\medskip

\begin{cor}
The random multigraph $\mathsf{G}_{n,\mathrm{Br}}^s$ conditioned to have $m$ vertices has the same law as $\mathsf C_{n,m}^s$.
\end{cor}

The paper \cite{ABBrGo} is devoted to the study of the distribution of $\mathcal G_{\mathrm{Br}}^s$ for $s \ge 0$. In particular, it is shown there that a version of $\mathcal G_{\mathrm{Br}}^s$ can be recovered by gluing appropriately rescaled Brownian continuum random trees along the edges of $\mathsf K_{\mathrm{Br}}^s$ (\cite[Procedure 1]{ABBrGo}) or via a line-breaking construction (\cite[Procedure 2 $\&$ Theorem~4]{ABBrGo}).

Let us turn now to other related work.  The study of scaling limits for critical random graph models was initiated by Aldous in \cite{AldousCritRG}, where he proved in particular the convergence of the sizes and surpluses of the largest components of the Erd\H{o}s-R\'enyi random graph in the critical window, as well as a similar result for the sizes of the largest components in an inhomogeneous random graph model.  This was followed soon afterwards by Aldous and Limic \cite{AldousLimic}, who explored the possible scaling limits for the sizes of the components in a ``rank-one'' inhomogeneous random graph, with the limiting sizes encoded as the lengths of excursions above past-minima of a so-called thinned L\'evy process.

In \cite{ABBrGo12}, it was shown that Aldous' result for the sizes and surpluses of the largest components in a critical Erd\H{o}s-R\'enyi random graph could be extended to include also the metric structure of the limiting components; the limiting object is what we refer to here as the \emph{Brownian graph}.  Since that paper, progress has been made in several directions.  One direction has been to demonstrate the \emph{universality} of the Brownian graph (first in terms of component sizes, and then in terms of the full metric structure).  This has been done for  critical rank-one inhomogeneous random graphs \cite{Turova, BhvdHvL3rdMo, BhamidiSenWang17}, for critical Achlioptas processes with bounded size rules \cite{AugMultCoal}, for critical configuration models with finite third moment degrees \cite{NachmiasPeres,Joseph14, Riordan12, DhvdHvLS17, BS16+} and in great generality in \cite{BBSW14+}.

Another line of enquiry, into which the present paper fits, is the investigation of other universality classes, generally those with power law degree distributions.  This has been pursued in the setting of rank-one inhomogeneous random graphs with power-law degrees in \cite{vdH13, BhvdHvLNovel, BhvdHSen18} and with very general weights by \cite{BrDuWa18,BrDuWa20}.
The configuration model with power-law degrees has been treated by \cite{Joseph14, DvdHvLS16, BDvdHS17,BhDhvdHSen2020}.  The last four papers are the most directly related to the topic of the present paper, and so we will discuss them in a little more detail.

In \cite{Joseph14}, Joseph considers the configuration model with i.i.d.\ degrees satisfying the same conditions as us, and proves the convergence in distribution of the component sizes (\ref{eqn:convsizes}).  (He leaves the equivalent convergence in the setting of the graph conditioned to be simple as a conjecture, but this is not hard to prove; see \cite{CKG17+} for the details.)  The results of \cite{CKG17+} in Theorem~\ref{th:C-K+G} thus directly generalise those of Joseph.  Dhara, van der Hofstad, van Leeuwaarden and Sen~\cite{DvdHvLS16} and Bhamidi, Dhara, van der Hofstad, and Sen~\cite{BDvdHS17,BhDhvdHSen2020} consider the component sizes and metric structure respectively for critical percolation on supercritical configuration models with degree sequences satisfying a certain power-law condition. The paper \cite{BDvdHS17} proves a metric space scaling limit, where the limit components are derived from the thinned L\'evy processes mentioned above. This scaling limit is proved in the product Gromov-weak topology, and the result is improved to a convergence in the product Gromov-Hausdorff-Prokhorov sense in \cite{BhDhvdHSen2020}.  This result is in principle somewhat more general in scope than that of \cite{CKG17+}, in that it covers a whole family of deterministic degree sequences; however, it is restricted to the case of critical percolation on a supercritical configuration model, whereas \cite{CKG17+} applies directly to a critical configuration model.   In principle, it should nonetheless be possible to view the stable graph as an appropriately annealed version of the scaling limit of \cite{BDvdHS17}.  However, it is for the moment unclear how to prove independently that the two objects obtained  must be the same.  The limit spaces obtained in \cite{BDvdHS17} are \emph{a priori} much less easy to understand than ours; the advantage of the i.i.d.\ setting is that we get very nice absolute continuity relations with the stable trees which are already well understood.  Obtaining analogous results in the setting of \cite{BDvdHS17} seems much more challenging.  (See \cite{CKG17+} for a more in-depth discussion of these issues and for a list of open problems.)

\subsection{Perspectives}

As discussed above, the results of this paper provide heavy-tailed analogues of those in \cite{ABBrGo}, which have been applied in other contexts.  Firstly, the decomposition into a continuous kernel with explicit distribution plus pendant subtrees played a key role in the proof of the existence of a scaling limit for the minimum spanning tree of the complete graph on $n$ vertices in \cite{ABBrGoMi}.  More specifically, assign the edges of the complete graph i.i.d.\ random edge-weights with $\mathrm{Exp}(1)$ distribution.  Now find the spanning tree $M_n$ of the graph with minimum total edge-weight.  (The law of $M_n$ does not depend on the weight distribution as long as it is non-atomic.)  Think of $M_n$ as a measured metric space in the usual way by endowing it with the graph distance $d_n$ and the uniform probability measure $\mu_n$ on its vertices. The main result of \cite{ABBrGoMi} is that
\[
(M_n, n^{-1/3} d_n, \mu_n) \convdist (\mathcal{M}, d, \mu)
\]
as $n \to \infty$, in the Gromov-Hausdorff-Prokhorov sense, where the limit space $(\mathcal{M},d,\mu)$ is a random measured $\R$-tree having Minkowski dimension 3 almost surely.  This convergence has, up to a constant factor, recently been shown by Addario-Berry and Sen~\cite{ABSen} to hold also for the MST of a uniform random 3-regular (simple) graph or for the MST of a 3-regular configuration model.

Following a scheme of proof similar to that developed in \cite{ABBrGoMi}, it may be possible to use the results of the present paper together with those of \cite{CKG17+} to prove an analogous scaling limit for the minimum spanning tree of the following model.  First, generate a uniform random graph (or configuration model) with i.i.d.\ degrees $D_1, D_2, \ldots, D_n$ with the same power-law tail behaviour as discussed above, but now in the supercritical setting $\nu > 1$.  For the purposes of this discussion, let us also assume that $\Prob{D_1 \ge 3} = 1$.  Under this condition, the graph not only has a giant component, but that component contains all of the vertices with probability tending to 1 \cite[Lemma~1.2]{ChatterjeeDurrett}.  As before, assign the edges of this graph i.i.d.\ random weights with Exp(1) distribution and find the minimum spanning tree $M_n$. Then we conjecture that in this setting we will have
\[
(M_n, n^{-(\alpha-1)/(\alpha+1)} d_n, \mu_n) \convdist (\mathcal{M}, d, \mu),
\]
for some measured $\R$-tree $(\mathcal{M},d,\mu)$. This conjecture will be the topic of future work. 

Another application of the results of \cite{ABBrGo} has been in the context of random maps. The Brownian versions of the graphs $\mathcal{G}^s, s \ge 0$ arise as scaling limits of unicellular random maps on various compact surfaces.  The results of \cite{ABBrGo} have, in particular, been used to study Voronoi cells in these objects.  More specifically, for a surface $S$, let $(\mathcal{U}(S), d, \mu)$ be the continuum random unicellular map on $S$ \cite{ABACFG}, endowed with its mass measure $\mu$, and let $X_1, X_2, \ldots, X_k$ be independent random points sampled from $\mu$.  Let $V_1, V_2, \ldots, V_k$ be the Voronoi cells with centres $X_1, \ldots, X_k$.  Then in \cite{ABACFG} it is shown that
\[
\left(\mu(V_1), \ldots, \mu(V_k)\right) \sim \mathrm{Dir}(1,1,\ldots,1).
\]
In other words, the Voronoi cells of uniform points provide a way to split the mass of the space up uniformly.  In principle, there should exist ``stable'' analogues of this result (in which the mass-split will no longer be uniform).

\subsection{Organisation of the paper}

Section~\ref{sec:SL} is devoted to background on stable trees, and to the description of the distribution of the limiting sequence of metric spaces arising in Theorem~\ref{th:C-K+G} in terms of a spectrally positive $\alpha$-stable L\'evy process. In particular, we give a precise description of the elementary building-blocks $\mathcal G^s, s \geq 0$.
We then enter the core of the paper with Section~\ref{sec:marginals} which is dedicated to the proof of the joint distribution of the discrete marginals $\mathsf G_n^s, n \geq  0$ (Theorems~\ref{th:distr_marginals} and \ref{thm:MarchalAlg}), including the connection to a configuration model stated in Corollary~\ref{cor:identification}. Section~\ref{sec:final} is devoted to the proofs of the construction of the $\R$-graph $\mathcal G^s$ from randomly scaled trees glued to its kernel and of its line-breaking construction (Theorem~\ref{thm:distrGs}, Proposition~\ref{cor:lengthskernel} and Theorem~\ref{th:linebreaking}, as well as Proposition~\ref{prop:approx}). Finally,  in the appendix, Section~\ref{sec:urns}, we recall the definitions and some properties of various distributions (generalized Mittag-Leffler, Beta, Dirichlet and Poisson-Dirichlet), as well as some classical urn model asymptotics, which are used at various points in the paper.

\section{The stable graphs}
\label{sec:SL}

We begin in Section~\ref{sec:stable} with some necessary background on stable trees. In particular, we recall Marchal's algorithm for constructing the discrete ordered marginals, and use it to obtain the joint distribution of various aspects (lengths, weights, local times) of the continuous marginals, which we will need later on. In Section \ref{sec:SLconfiguration}, we turn to the distribution of the limiting sequence of metric spaces arising in Theorem~\ref{th:C-K+G} and in particular to the construction of the stable graphs. 

Throughout this section, we fix $\alpha \in (1,2)$.

\subsection{Background on stable trees}
\label{sec:stable}

\subsubsection{Construction and properties}  \label{subsec:alphastable}

The $\alpha$-stable tree was introduced by Duquesne and Le Gall \cite{DuquesneLeGall}, building on earlier work of Le Gall and Le Jan \cite{LeGallLeJan}. 
Our presentation of this material owes much to that of Curien and Kortchemski~\cite{CRLoop13}, which relies in turn on various key results from Miermont~\cite{Mier05}. 

First, let $\xi$ be a spectrally positive $\alpha$-stable L\'evy process with Laplace exponent
\[
\E{\exp(-\lambda \xi_t)} = \exp(t \lambda^{\alpha}), \quad \lambda \ge 0, \quad t \ge 0.
\]
Now consider a reflected version of this L\'evy process, namely $(\xi_t - \inf_{0 \le s \le t} \xi_s, t \ge 0)$.  It is standard that this process has an associated excursion theory, and that one can make sense of an excursion conditioned to have length 1.  We will write $X$ for this excursion of length 1, and observe that, thanks to the scaling property of $\xi$ we may obtain the law of an excursion conditioned to have length $x > 0$ via $(x^{1/\alpha} X(t/x), 0 \le t \le x)$. See Chaumont~\cite{Chaumont} for more details.

To a normalised excursion $X$ we may associate an $\R$-tree.  In order to do this, we first derive from $X$ a \emph{height function} $H$, defined as follows: for $t \in [0,1]$,
\[
H(t) = \lim_{\varepsilon \to 0^+} \frac{1}{\varepsilon} \int_0^t \I{X(s) < \inf_{s \le r \le t}X(r) + \varepsilon} \mathrm{d}s.
\]
The process $H$ possesses a continuous modification such that $H(0) = H(1) = 0$ and $H(t) > 0$ for $t \in (0,1)$, which we consider in the sequel (see Duquesne and Le Gall~\cite{DuquesneLeGall} for more details).
We then obtain an $\R$-tree in a standard way from $H$ by first defining a pseudo-distance $d$ on $\R_+$ via
\[
d(s,t) = H(s) + H(t) - 2 \inf_{s \wedge t \le r \le s \vee t} H(r).
\]
Now define an equivalence relation $\sim$ by declaring $s \sim t$ if $d(s,t) = 0$.  Then let $\mathcal{T}$ be the metric space obtained by endowing $[0,1]/\sim$ with the image of $d$ under the quotienting operation.  Let us write $\pi: [0,1] \to \mathcal{T}$ for the projection map.  We additionally endow $\mathcal{T}$ with the push-forward of the Lebesgue measure on $[0,1]$ under $\pi$, which is denoted by $\mu$.  The point $\rho := \pi(0) = \pi(1)$ is naturally interpreted as a \emph{root} for the tree.  We will refer to the random variable $(\mathcal{T},d,\mu)$ as the (standard) \emph{$\alpha$-stable tree}.  In the usual notation, for points $x,y \in \mathcal{T}$, we will write $\seg{x}{y}$ for the path between $x$ and $y$ in $\mathcal{T}$, and $\oseg{x}{y}$ for $\seg{x}{y} \setminus \{x,y\}$.  (These are isometric to closed and open line-segments of length $d(x,y)$, respectively.) We can use the root to endow the tree $\mathcal{T}$ with a \emph{genealogical order}: we say $x \preceq y$ if $x \in \seg{\rho}{y}$.  We define the \emph{degree}, $\mathrm{deg}(x)$, of a point $x \in \mathcal{T}$ to be the number of connected components into which its removal splits the space. If there is any potential ambiguity over which metric space we are working in, we will write $\deg_{\cT}(x)$. The \emph{branchpoints} are those with degree strictly greater than 2 and the leaves are those with degree 1; we write $\mathrm{Br}(\mathcal{T}) = \{x \in \mathcal{T}: \mathrm{deg}(x) > 2\}$ and $\mathrm{Leaf}(\mathcal{T}) = \{x \in \mathcal{T}: \mathrm{deg}(x) = 1\}$.  We observe that the distance $d$ induces a natural length measure on the tree $\mathcal{T}$, for which we write $\lambda$.

 We also define a partial order $\preceq$ on $[0,1]$ by declaring
	\begin{equation}\label{def:partial order on 0,1}
		s \preceq t \quad \text{if} \quad s \le t \text{ and } X(s-) \leq \inf_{s \le r \le t} X(r).
	\end{equation} 
(We take as a convention that $X(0-) = 0$.) This partial order is compatible with the genealogical order on $\mathcal{T}$ in the sense that for $x,y \in \mathcal{T}$, $x \preceq y$ if and only if there exist $s,t \in [0,1]$ such that $x = \pi(s)$ and $y = \pi(t)$ and $s \preceq t$.

We will require various properties of $\mathcal{T}$ in the sequel.  We will make use of the fact that the law of $\mathcal{T}$ is invariant under re-rooting at a random point with distribution $\mu$ \cite{HaasPitmanWinkel, DLG09}. So we will sometimes think of the tree as unrooted and regenerate a root from $\mu$ when necessary.  Another key feature of $\mathcal{T}$ is that its branchpoints are all of infinite degree, almost surely.  By Proposition 2 of Miermont~\cite{Mier05}, $x \in \mathrm{Br}(\mathcal{T})$ if and only if there exists a unique $s \in [0,1]$ such that $x = \pi(s)$ and $\Delta X(s) = X(s) - X(s-) > 0$. For all other values $r\in\intervalleff{0}{1}$ such that $\pi(r)=\pi(s)=x$, we have $\inf_{s\leq u \leq r} X(u)=X(r)\geq X(s-)$. For such $s$ associated to a branchpoint $x = \pi(s)$, we will define $N(x) := \Delta X(s)$.  By Miermont's equation \cite[Eq.\ (1)]{Mier05}, for all $x \in \mathrm{Br}(\mathcal T$) this quantity may be almost surely  recovered as
\[
N(x) = \lim_{\varepsilon \to 0+} \frac{1}{\varepsilon} \mu(\{y \in \mathcal T:x \in \seg{\rho}{y}, d(x,y) < \varepsilon\}),
\]
and so $N(x)$ gives a renormalised notion of the degree of $x$.  We will refer to this quantity as the \emph{local time} of $x$, since it plays that role with respect to $H$.

For any $s,t \in [0,1]$ such that $\pi(s) \in \mathrm{Br}(\mathcal{T})$ and $s \preceq t$, we also define the \emph{local time of $\pi(s)$ to the right of $\pi(t)$} to be
\[
N^{\mathrm{right}}(\pi(s),\pi(t)) = \inf_{s \le u \le t} X(u) - X(s-).
\]
Then $N^{\mathrm{right}}(\pi(s),\pi(t)) \in [0,N(\pi(s))]$ is a measure of how far through the descendants of $\pi(s)$ we are when we visit $\pi(t)$.  (Indeed, since $\pi(s) \in \mathrm{Br}(\mathcal{T})$, if $s \preceq t$ and $s \preceq u$ with $N^{\mathrm{right}}(\pi(s), \pi(t)) > N^{\mathrm{right}}(\pi(s), \pi(u))$ then necessarily $t < u$.)
By Corollary 3.4 of \cite{CRLoop13}, we can express $X(t)$ as the sum of the atoms of local time along the path from the root to $\pi(t)$:
\begin{equation}\label{eq:local time on the right of the path from the root}
X(t)=\sum_{0\preceq s \preceq t} N^{\mathrm{right}}(\pi(s),\pi(t)),
\end{equation}
almost surely for all $t\in\intervalleff{0}{1}$.
For any $s \preceq t$, we define the local time along the path $\oseg{\pi(s)}{\pi(t)}$ by
\begin{align*}
N(\oseg{\pi(s)}{\pi(t)}) & :=  \sum_{b \in \mathrm{Br}(\mathcal{T}) \cap \oseg{\pi(s)}{\pi(t)}} N(b), \\
\intertext{and the local time to the right along the path $\oseg{\pi(s)}{\pi(t)}$ by}
N^{\mathrm{right}}(\oseg{\pi(s)}{\pi(t)}) & := \sum_{b \in \mathrm{Br}(\mathcal{T}) \cap \oseg{\pi(s)}{\pi(t)}} N^{\mathrm{right}}(b,\pi(t)) = X(t-) - X(s),
\end{align*}
where we observe that all of these sums are over countable sets.

\subsubsection{Marchal's algorithm for ordered trees}
\label{sec:MarchalTree}

Consider an infinite sample of leaves from $(\mathcal T,d,\mu)$ obtained as the images of i.i.d.\ uniform random variables $U_1, U_2, \ldots$ on $[0,1]$ under the quotienting.  These leaves, which we label $1, 2, \ldots$, inherit an order from $[0,1]$.  For $n \in \mathbb N$, let $\mathcal T^{\mathrm{ord}}_n$ be an ordered leaf-labelled version of the subtree of $\mathcal T$ spanned by the root and the first $n$ leaves (the order being inherited from the leaves) and $\mathsf T^{\mathrm{ord}}_n$ its combinatorial shape, also with leaf-labels.  Formally, 
$$\mathsf T^{\mathrm{ord}}_n=\mathrm{shape}(\mathcal T^{\mathrm{ord}}_n)$$
where, for any compact rooted (say at $\rho$) real tree $\tau$ (possibly ordered), $\mathrm{shape}(\tau)$ is the (possibly ordered) rooted discrete tree $(V,E)$ with no vertex of degree 2 except possibly the root, where
\begin{equation}
\label{def:shape}
V=\{\rho\} \cup \{v \in \tau\backslash \{\rho\}:\mathrm{deg}_{\tau}(v) \neq 2\} \ \text{and} \ E=\Big\{\{u,v\}:\ u,v\in V,\ \mathrm{deg}_{\tau}(w)=2, \forall w \in ]]u,v[[ \text{ and } \rho \notin ]]u,v[[\Big\}.
\end{equation}
We define the shape of a discrete tree similarly. Note that, in fact, all of the trees we consider will have a root of degree 1: they are \emph{planted}.

For any $n\geq 1$, we denote by $\A_n$ the set of planted ordered finite trees with $n$ labelled leaves, with labels from $1$ to $n$, and no vertex of degree $2$. The root is thought of as a leaf with label $0$. In \cite[Section 3]{DuquesneLeGall}, Duquesne and Le Gall show that  for each tree $T\in\A_n$ with set of internal vertices $I(T)$, 
\begin{equation}
\label{plane tree distribution}
\mathbb P\big(\mathsf T^{\mathrm{ord}}_n=T\big)\propto \prod_{u\in I(T)}\frac{w_{\mathrm{deg}_T(u)-1}}{(\mathrm{deg}_T(u)-1)!},
\end{equation}
where the weights $(w_k,k\geq 0)$ were defined in (\ref{Marchal's weights}). In other words, $\mathsf T^{\mathrm{ord}}_n$ is distributed as a planted version of a Galton-Watson tree with offspring distribution $\eta_{\alpha}$ as defined in Section \ref{sec:marginalsintro} (below Corollary \ref{cor:identification}), conditioned on having $n$ leaves uniformly labelled from $1$ to $n$.

Building on this result, in \cite{Marchal} Marchal proposed  a recursive construction of a sequence with the same law as $(\mathsf T^{\mathrm{ord}}_n,n \geq 1)$.  (In fact, Marchal gave a construction of the non-ordered versions of the trees $\mathsf T^{\mathrm{ord}}_n,n \geq 1$ but combined with \cite[Section 2.3]{Marchal} we easily obtain an ordered version.) For any $n \geq 1$ and any $T\in\A_n$, we construct randomly a tree in $\A_{n+1}$ as follows.
\begin{enumerate}
\item[(1)] Assign to every edge of $T$ a weight $\alpha-1$ and every internal vertex $u$ a weight $\deg_T(u)-1-\alpha$; the other vertices have weight 0;
\item[(2)] Choose an edge/vertex with probability proportional to its weight and then
\begin{itemize}
	\item if it is a vertex, choose a uniform corner around this vertex, attach a new edge-leaf in this corner and give the leaf the label $n+1$,
	\item if it is an edge, create a new vertex which splits the edge into two edges, and attach an edge-leaf with leaf labelled $n+1$ pointing to the left/right with probability $1/2$.
\end{itemize}
\end{enumerate}
If we start with the unique element of $\A_1$ and apply this procedure recursively, we obtain a sequence of trees distributed as $(\mathsf T^{\mathrm{ord}}_n,n \geq 1)$.

\noindent \textbf{Asymptotic behaviour.} Consider now the discrete trees as metric spaces, endowed with the graph distance. Fix $k$ and for each $k \leq n$ let $\mathsf T^{\mathrm{ord}}_{k}(n)$ be the subtree of $\mathsf T^{\mathrm{ord}}_n$ spanned by the leaves with labels $1,2,\ldots, k$ and the root. Hence, $\mathsf T^{\mathrm{ord}}_k=\mathrm{shape}(\mathsf T^{\mathrm{ord}}_{k}(n))$ but the distances in $\mathsf T^{\mathrm{ord}}_{k}(n)$ are inherited from those in $\mathsf T^{\mathrm{ord}}_n$. We may therefore view $\mathsf T^{\mathrm{ord}}_{k}(n)$ as a discrete tree having the same vertex- and edge-sets as $\mathsf T^{\mathrm{ord}}_k$, but where the edges now have lengths. Similarly for $\mathcal T_k^{\mathrm{ord}}$. Again from Marchal~\cite{Marchal}, we have
\begin{equation}
\label{cv:fdtrees}
\frac{\mathsf T^{\mathrm{ord}}_{k}(n)}{n^{1-1/\alpha}} \ \underset{n \rightarrow \infty}{\overset{\mathrm{a.s.}} \longrightarrow} \  \alpha \cdot \mathcal T^{\mathrm{ord}}_k,
\end{equation}
as $n \to \infty$, where the convergence means that the rescaled lengths of the edges of $\mathsf T^{\mathrm{ord}}_{k}(n)$ converge to the lengths, multiplied by $\alpha$, of the corresponding edges in $\mathcal T^{\mathrm{ord}}_k$. 
This convergence of random finite-dimensional marginals can be improved when considering trees as metric spaces (i.e.\ we forget the order) equipped with probability measures. Indeed, if $\mathsf T_n$ denotes the unordered version of $\mathsf T^{\mathrm{ord}}_{n}$, with leaves still labelled $0,1,2,\ldots,n$ ($0$ is the root), $\mu_n$ the uniform probability measure on these leaves, then we have that 
\begin{equation}
\label{cv:GHPtrees}
\left( \frac{\mathsf T_n}{n^{1-1/\alpha}}, \mu_n,0,\ldots,k\right)  \ \underset{n \rightarrow \infty}{\overset{\mathrm{a.s.}} \longrightarrow} \  \alpha \cdot (\mathcal T,\mu, 0,\ldots,k)
\end{equation}
for the $(k+1)$-pointed Gromov-Hausdorff-Prokhorov topology on the set of measured $(k+1)$-pointed compact trees, for each integer $k$. (See e.g.\ \cite[Section 6.4]{MiermontTessellations} for a definition of this topology.) The convergence (\ref{cv:GHPtrees}) was first proved in probability in \cite[Corollary 24]{HMPW} and then improved to an almost sure convergence in \cite[Section 2.4]{CurienHaas}.

Suppose now that $\mathsf T^{\mathrm{ord}}_k$ has edge-set $E(\mathsf T^{\mathrm{ord}}_k)$, labelled arbitrarily as $e_i,1 \leq i\leq \abs{E(\mathsf T^{\mathrm{ord}}_k)}$, and internal vertices $I(\mathsf T^{\mathrm{ord}}_k)$, labelled arbitrarily as $v_j$, $1 \leq j \leq \abs{I(\mathsf T^{\mathrm{ord}}_k)}$.  As discussed above, for $k \le n$, the internal vertices $I(\mathsf T^{\mathrm{ord}}_k)$ all have counterparts in $\mathsf T^{\mathrm{ord}}_{k}(n)$, which we will also call $v_j$, $1 \leq j \leq \abs{I(\mathsf T^{\mathrm{ord}}_k)}$.  To each edge $e_i \in E(\mathsf T^{\mathrm{ord}}_k)$ there corresponds a path $\gamma_i$ in $\mathsf T^{\mathrm{ord}}_{k}(n)$ whose endpoints are elements of $\{v_j, 1 \le j \le |I(\mathsf T^{\mathrm{ord}}_k)|\} \cup \{0,1,\ldots,k\}$. Write $\gamma_i^{\circ}$ for the same path with its endpoints removed ($\gamma_i^{\circ}$ may be empty).  Since $\mathsf T^{\mathrm{ord}}_{k}(n) \subset \mathsf T^{\mathrm{ord}}_n$, we refer to the corresponding vertices and paths in $\mathsf T^{\mathrm{ord}}_n$ by the same names.  

\begin{figure}
\begin{center}
\includegraphics[width=12cm]{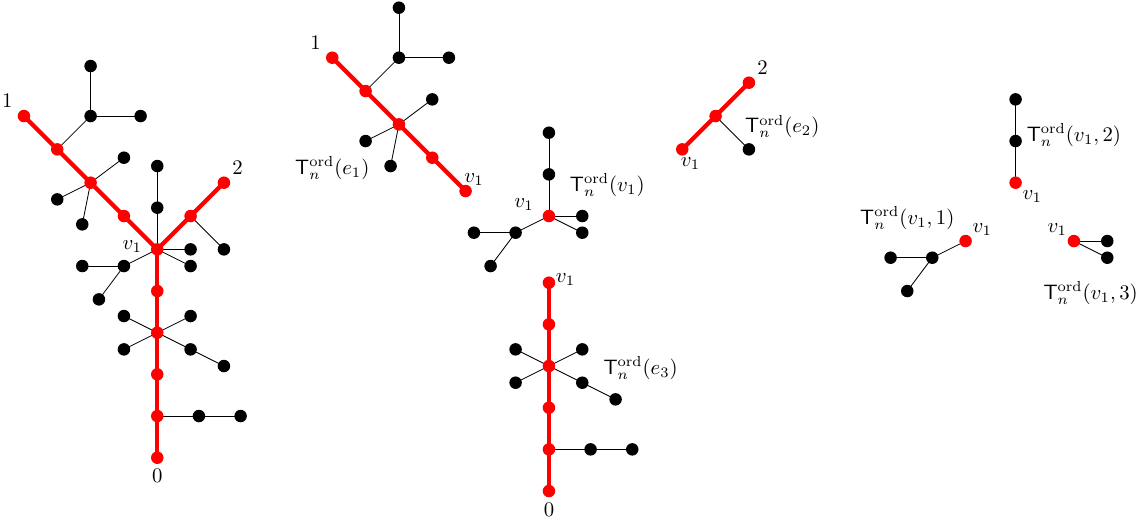}
\end{center}
\caption{Left: the tree $\mathsf T^{\mathrm{ord}}_{n}$ for $n =18$ (leaf-labels $3, \ldots, 18$ are suppressed for purposes of readability).  $\mathsf T^{\mathrm{ord}}_2(n)$ is emphasised in red and bold. The tree $\mathsf T_2^{\mathrm{ord}}$ has a single internal vertex called $v_1$ and edges $e_1 = \{v_1, 1\}$, $e_2 = \{v_1,2\}$ and $e_3=\{v_1,0\}$. The corresponding paths in $\mathsf T_{2}^{\mathrm{ord}}(n)$ have lengths 4, 2 and 5 respectively. Middle: the subtrees $\mathsf T_{n}^{\mathrm{ord}}(e_1)$, $\mathsf T_{n}^{\mathrm{ord}}(e_2)$, $\mathsf T_{n}^{\mathrm{ord}}(e_3)$ and $\mathsf T_{n}^{\mathrm{ord}}(v_1)$. Right: the subtrees $\mathsf T_n^{\mathrm{ord}}(v_1,1)$, $\mathsf T_n^{\mathrm{ord}}(v_1,2)$ and $\mathsf T_n^{\mathrm{ord}}(v_1,3)$.}
\label{fig:subtrees}
\end{figure}

We will now give names to certain important subtrees of $\mathsf T^{\mathrm{ord}}_n$ and refer the reader to Figure~\ref{fig:subtrees} for an illustration.  For each vertex $v \in V(\mathsf{T}^{\mathrm{ord}}_n)$, the unique directed path from $v$ to $0$ has a first point $\mathrm{int}(v)$ of intersection with $\mathsf{T}^{\mathrm{ord}}_k(n)$.  For $1 \le j \le |I(\mathsf T_k^{\mathrm{ord}})|$, let $\mathsf{T}_n^{\mathrm{ord}}(v_j)$ be the subtree induced by the set of vertices $\{v: \mathrm{int}(v) = v_j\}$ and rooted at $v_j$.  If $\mathrm{int}(v) \notin \{v_j: 1 \le j \le |I(\mathsf T_k^{\mathrm{ord}})|\}$ then $\mathrm{int}(v)$ belongs to $\gamma_i^{\circ}$ for some $1 \le i \le |E(\mathsf T_k^{\mathrm{ord}})|$.  Let $\mathsf{T}_n^{\mathrm{ord}}(e_i)$ be the subtree of $\mathsf{T}_n^{\mathrm{ord}}$ induced by the vertices $\{v \in V(\mathsf{T}^{\mathrm{ord}}_n): \mathrm{int}(v) \in \gamma_i^{\circ}\} \cup \gamma_i$ and rooted at the endpoint of $\gamma_i$ closest to the root of $\mathsf T_n^{\mathrm{ord}}$.

If $\mathrm{deg}_{\mathsf{T}_k^{\mathrm{ord}}}(v_j) = d_j$ then $\mathsf{T}_n^{\mathrm{ord}}(v_j)$ can be split up into separate subtrees descending from the $d_j$ different corners of $v_j$.  We list these subtrees in clockwise order from the root as $\mathsf{T}_n^{\mathrm{ord}}(v_j, \ell)$, $1 \le \ell \le d_j$.

For each $e_i, 1 \le i \le |E(\mathsf{T}_k^{\mathrm{ord}})|$ then denote by
\begin{itemize}
\item $L_{n}(e_i)$  the length of $\gamma_i$ in $\mathsf T^{\mathrm{ord}}_{k}(n)$,
\item $M_{n}(e_i)$ the number of leaves in the subtree $\mathsf T_n^{\mathrm{ord}}(e_i)$,
\item $N_{n}(e_i)$ the number of edges of $\mathsf T^{\mathrm{ord}}_n(e_i)$ adjacent to $\gamma_i$,
\item $N^{\mathrm{right}}_{n}(e_i)$ the number of edges of $\mathsf T^{\mathrm{ord}}_n(e_i)$ attached to the right of $\gamma_i$,
\item $N_n(e_i, \ell)$ the degree $-2$ of the $\ell$th largest branchpoint along the path $\gamma_i$ in $\mathsf T^{\mathrm{ord}}_n(e_i)$, for $\ell \ge 1$, with ties broken arbitrarily,
\item $N^{\mathrm{right}}_n(e_i, \ell)$ the degree to the right of the $\ell$th largest branchpoint along the path $\gamma_i$ in $\mathsf T^{\mathrm{ord}}_n(e_i)$, for $\ell \ge 1$ (with the same labelling as in the previous point).
\item $L_n(e_i,\ell)$ the distance from the $\ell$th largest branchpoint of $\gamma_i$ to the root (endpoint nearest 0 in $\mathsf T^{\mathrm{ord}}_n$) of $\mathsf T^{\mathrm{ord}}_n(e_i)$, $\ell\geq 1$, again with the same labelling.
\end{itemize}
Observe that $N_n(e_i) = \sum_{\ell \ge 1} N_n(e_i, \ell)$ and $N^{\mathrm{right}}_{n}(e_i) = \sum_{\ell \ge 1} N^{\mathrm{right}}_n(e_i, \ell)$.

Similarly, for each vertex $v_j$, $1 \leq j \leq \abs{I(\mathsf T^{\mathrm{ord}}_n)}$, denote by
\begin{itemize}
\item $N_{n}(v_j)$ the degree of $v_j$ in $\mathsf T^{\mathrm{ord}}_n$ (i.e.\ $\mathrm{deg}_{\mathsf T_n^{\mathrm{ord}}}(v_j)$),
\item $N_{n}(v_j,\ell)$ the degree of $v_j$ in $\mathsf{T}_n^{\mathrm{ord}}$ in the $\ell$th corner counting clockwise from the root, for $1 \leq \ell \leq \mathrm{deg}_{\mathsf T_k^{\mathrm{ord}}}(v_j)$, 
\item $M_{n}(v_j)$ the number of leaves in $\mathsf{T}_n^{\mathrm{ord}}(v_j)$,
\item $M_{n}(v_j,\ell)$ the number of leaves in $\mathsf{T}_n^{\mathrm{ord}}(v_j, \ell)$, for $1 \leq \ell \leq \mathrm{deg}_{\mathsf T_k^{\mathrm{ord}}}(v_j)$. 
\end{itemize}
We use the same edge- and vertex-labels for the corresponding parts of $\mathcal T^{\mathrm{ord}}_k$.
Since $\mathcal T^{\mathrm{ord}}_k$ is (an ordered version of) a subset of $\mathcal T$, we have that $e_i$ corresponds to an open path $\oseg{x_{i,1}}{x_{i,2}}$ for some pair of points $x_{i,1}, x_{i,2} \in \mathcal{T}$ such that $x_{i,1} \preceq x_{i,2}$.   
Let $L(e_i)=d(x_{i,1}, x_{i,2})$ be the length of this path. We will abuse notation somewhat by writing $N(e_i)$ and $N^{\mathrm{right}}(e_i)$ instead of $N(\oseg{x_{i,1}}{x_{i,2}})$ and $N^{\mathrm{right}}(\oseg{x_{i,1}}{x_{i,2}})$ for the local time of the edge and the local time to the right of the edge respectively.
For $\ell \ge 1$, we will write $N(e_i, \ell)$ for the local time of the $\ell$th largest branchpoint along $\oseg{x_{i,1}}{x_{i,2}}$ (with ties broken arbitrarily), $N^{\mathrm{right}}(e_i, \ell)$ for the local time to the right at the same branchpoint, and $L(e_i,\ell)$ for the distance from that branchpoint to the lower endpoint $x_{i,1}$ of $e_i$.
Each vertex $v_j$ corresponds to some point of $\mathcal{T}$, which by abuse of notation we will also call $v_j$.  (Note that, of course, we must have $\{v_j: 1 \le j \le \abs{I(\mathsf{T}_k^{\mathrm{ord}})}\} \cup \{0,1,\ldots,k\}= \{x_{i,p}: 1 \le i \le \abs{E(\mathsf{T}_k^{\mathrm{ord}})}, p=1,2\}$.)

Let $\mathcal{T}(e_i)$ be the subtree of $\mathcal{T}$ containing $\seg{x_{i,1}}{x_{i,2}}$, formally defined by 
\[
\mathcal{T}(e_i) = \{x \in \mathcal{T}: \seg{\rho}{x} \cap \oseg{x_{i,1}}{x_{i,2}} \neq \emptyset, x_{i,2} \notin \seg{\rho}{x}\} \cup \{x_{i,1},x_{i,2}\}.
\]
Let $M(e_i) = \mu(\mathcal{T}(e_i))$.  Let $\mathcal{T}(v_j)$ be the subtree of $\mathcal{T}$ attached to $v_j$, namely 
\[
\mathcal{T}(v_j) = \{x \in \mathcal{T}: v_j \in \seg{\rho}{x}, \oseg{v_j}{x} \cap \seg{x_{i,1}}{x_{i,2}} = \emptyset \text{ for all } 1 \le i \le \abs{E(\mathsf{T}_n^{\mathrm{ord}})}\}.
\]
Let $M(v_j) = \mu(\mathcal{T}(v_j))$.  As in the discrete case, we can split up $\mathcal{T}(v_j)$ into subtrees sitting in the $\mathrm{deg}_{\mathsf{T}_k^{\mathrm{ord}}}(v_j)$ corners of $v_j$.  We call these $\mathcal{T}(v_j, \ell)$ for $1 \le \ell \le \mathrm{deg}_{\mathsf{T}_k^{\mathrm{ord}}}(v_j)$.  Let 
\[
N(v_j, \ell) = \lim_{\epsilon \to 0+} \frac{1}{\epsilon} \mu(\{y \in \mathcal{T}(v_j, \ell): d(v_j,y) < \epsilon\}).
\]

\begin{lem}
\label{lem:grandeurslimites}
We have the almost sure joint convergence, for $1 \leq i\leq  \abs{E(\mathsf T_k^{\mathrm{ord}})}$ and $\ell \ge 1$,
\begin{align*}
& \frac{L_{n}(e_i)}{n^{1-1/\alpha}}  \ \underset{n \rightarrow \infty} \longrightarrow \  \alpha \cdot  L(e_i), 
\qquad\frac{M_n(e_i)}{n} \ \underset{n \rightarrow \infty} \longrightarrow \ M(e_i),  \\
&  \frac{N_{n}(e_i)}{n^{1/\alpha}}  \ \underset{n \rightarrow \infty} \longrightarrow \ N(e_i), \qquad \frac{N^{\mathrm{right}}_{n}(e_i)}{n^{1/\alpha}}  \ \underset{n \rightarrow \infty} \longrightarrow \ N^{\mathrm{right}}(e_i),  \\
& \frac{N_{n}(e_i,\ell)}{n^{1/\alpha}}  \ \underset{n \rightarrow \infty} \longrightarrow \ N(e_i,\ell), \qquad \frac{N^{\mathrm{right}}_{n}(e_i, \ell)}{n^{1/\alpha}}  \ \underset{n \rightarrow \infty} \longrightarrow \ N^{\mathrm{right}}(e_i,\ell),
\end{align*}
and for $1 \leq j \leq \abs{I(\mathsf T_k^{\mathrm{ord}})}$, $1 \leq \ell \leq \mathrm{deg}_{\mathsf T_k^{\mathrm{ord}}}(v_j)$,
\begin{align*}
& \frac{M_{n}(v_j)}{n} \ \underset{n \rightarrow \infty} \longrightarrow \ M(v_j),  \\
& \frac{N_n(v_j)}{n^{1/\alpha}}  \ \underset{n \rightarrow \infty} \longrightarrow \  {N(v_j)},  \qquad \frac{N_n(v_j, \ell)}{n^{1/\alpha}}  \ \underset{n \rightarrow \infty} \longrightarrow \  {N(v_j, \ell)}.  
\end{align*}
\end{lem}

\begin{proof}
The convergence of the lengths is Marchal's result (\ref{cv:fdtrees}).
The convergence of the local times is proved in Dieuleveut \cite[Lemma~2.7 \& Lemma~2.8]{Dieuleveut15}. Finally, the convergences of the subtree masses are an immediate consequence of the strong law of large numbers. Note that since we are dealing with a countable collection of random variables, these convergences indeed hold simultaneously almost surely.
\end{proof}

\subsubsection{Marginals of the stable tree}
\label{sec:marginalsstable}

We now state explicitly the joint distributions of all of the limit quantities in Lemma~\ref{lem:grandeurslimites}.

\begin{prop} \label{prop:jointlaws}
Conditionally on $\mathsf{T}_k^{\mathrm{ord}}$ with $\abs{E(\mathsf{T}_k^{\mathrm{ord}})} = m$ and $\abs{I(\mathsf{T}_k^{\mathrm{ord}})} = r := m-k$, with $\mathrm{deg}_{\mathsf{T}_k^{\mathrm{ord}}}(v_j) = d_j$ for $1 \le j \le r$, we have jointly
\begin{align*}
\left( M(e_1), \ldots, M(e_{m}), M(v_1), \ldots, M(v_r) \right) & \equidist (D_1, D_2, \ldots, D_{m+r})\\
\left(N(e_1), \ldots, N(e_m), N(v_1), \ldots, N(v_r) \right) & \equidist (D_1^{1/\alpha} R_1, \ldots, D_{m+r}^{1/\alpha} R_{m+r}) \\
\alpha \cdot \left( L(e_1), \ldots, L(e_m) \right) & \equidist  (D_1^{1-1/\alpha} R_1^{\alpha-1} \bar{R}_1, D_2^{1-1/\alpha} R_2^{\alpha-1} \bar{R}_2, \ldots, D_m^{1-1/\alpha} R_m^{\alpha-1} \bar{R}_m), 
\end{align*}
where the following elements are independent:
\begin{itemize}
\item $(D_1, \ldots, D_m, D_{m+1}, \ldots, D_{m+r}) \sim \mathrm{Dir}(1-1/\alpha, \ldots, 1-1/\alpha, (d_1-1-\alpha)/\alpha, \ldots, (d_r-1-\alpha)/\alpha)$;
\item $R_1, R_2, \ldots, R_{m+r}$ are mutually independent with $R_1, \ldots, R_m \sim \mathrm{ML}(1/\alpha,1-1/\alpha)$ and $R_{m+i} \sim \mathrm{ML}(1/\alpha,(d_i-1-\alpha)/\alpha)$ for $ 1 \le i \le r$;
\item $\bar{R}_1, \bar{R}_2, \ldots, \bar{R}_m$ are i.i.d.\ $\mathrm{ML}(\alpha-1,\alpha-1)$.
\end{itemize}
Moreover, we have $R_i^{\alpha-1}\bar{R}_i \sim  \mathrm{ML}(1-1/\alpha, 1-1/\alpha)$ for $1 \le i \le m$.

The random variables $N^{\mathrm{right}}(e_i,\ell)/N(e_i, \ell)$ and $L(e_i,\ell)/L(e_i)$ for $1 \le i \le m$, $\ell \ge 1$, the random sequences $(N(e_i, \ell)/N(e_i), \ell \ge 1)$ for $1 \le i \le m$,  and the random vectors $(N(v_j, \ell)/N(v_j), 1 \le \ell \le d_j)$ for $1 \le j \le r$ are mutually independent, and are also independent of $N(e_i), 1 \le i \le m$ and $N(v_j), 1 \le j \le r$. Moreover, we have
\[
\left(\frac{N(e_i,\ell)}{N(e_i)}, \ell \ge 1\right) \sim \mathrm{PD}(\alpha-1,\alpha-1),\quad 1 \le i \le m,
\]
\[
\frac{N^{\mathrm{right}}(e_i, \ell)}{N(e_i, \ell)} \sim \mathrm{U}[0,1], \quad 1 \le i \le m, \quad \ell \ge 1,
\]
\begin{align*}
	\frac{L(e_i,\ell)}{L(e_i)}\sim \mathrm{U}[0,1], \quad 1 \le i \le m, \quad \ell \ge 1,
\end{align*}
and
\[
\left(\frac{N(v_j, \ell)}{N(v_j)}, 1 \le \ell \le d_j\right) \sim \mathrm{Dir}(1,1,\ldots,1), \quad 1 \le j \le r.
\]
\end{prop}

The distributional results for the masses, lengths and total local times may be read off from~\cite{GHlinebreaking}, although the precise dependence between lengths and local times is left somewhat implicit there.  Related results appeared earlier in \cite{HaasPitmanWinkel}. We give a complete proof of Proposition~\ref{prop:jointlaws} via an urn model which we now introduce.

Suppose we have $k$ colours such that each colour has three \emph{types}: $a$, $b$ and $c$.  Let $X_i^a(n)$, $X_i^b(n)$ and $X_i^c(n)$ be the weights of the three types of colour $i$ in the urn at step $n$, respectively, for $1 \le i \le k$.  At each step we draw a colour with probability proportional to its weight in the urn.  If we pick the colour $i$ type $a$, we add weight $\alpha-1$ to colour $i$ type $a$, $2-\alpha$ to colour $i$ type $b$ and $\alpha-1$ to colour $i$ type $c$ (recall that $\alpha \in (1,2)$). If we pick colour $i$ type $b$, we add $1$ to colour $i$ type $b$ and $\alpha-1$ to colour $i$ type $c$. If we pick colour $i$ type $c$, we simply add weight $\alpha$ to colour $i$ type $c$.  We start with
\[
X^a_i(0) = \gamma_i, \quad X^b_i(0) = 0,  \quad X^c_i(0) = 0, \quad 1 \le i \le k.
\]

\begin{prop} \label{prop:urn}
As $n \to \infty$, we have the following almost sure limits:
\begin{align*}
\frac{1}{(\alpha-1) n^{1-1/\alpha}}(X^a_1(n), \ldots, X^a_k(n)) & \to (D_1^{1-1/\alpha} R_1^{\alpha-1} \bar{R}_1, \ldots, D_k^{1-1/\alpha} R_k^{\alpha-1} \bar{R}_k)\\
\frac{1}{n^{1/\alpha}} (X^b_1(n), \ldots, X^b_k(n)) & \to (D_1^{1/\alpha} R_1, \ldots, D_k^{1/\alpha} R_k) \\
\frac{1}{\alpha n} (X^c_1(n), \ldots, X^c_k(n)) & \to (D_1, D_2, \ldots, D_k),
\end{align*}
where the sequences $(D_1, \ldots, D_k)$, $(R_1, \ldots, R_k)$ and $(\bar{R}_1, \ldots, \bar{R}_k)$  are independent; we have $(D_1, \ldots, D_k) \sim \mathrm{Dir}(\gamma_1/\alpha, \ldots, \gamma_k/\alpha)$; the random variables $R_1, \ldots, R_k$ are mutually independent with $R_i \sim \mathrm{ML}(1/\alpha, \gamma_i/\alpha)$; and the random variables $\bar{R}_1, \ldots,\bar{R}_k$ are mutually independent with $\bar{R_i} \sim\mathrm{ML}(\alpha-1,\gamma_i)$.
\end{prop}

The proof of Proposition~\ref{prop:urn} appears in Section~\ref{sec:urnproof}.

\begin{proof}[Proof of Proposition~\ref{prop:jointlaws}]
We make use of Marchal's algorithm.  Recall that we are given an ordered tree $\mathsf{T}_k^{\mathrm{ord}}$ with $k$ leaves labelled $1,2,\ldots,k$, $m$ edges and $r$ internal vertices with degrees $d_1, \ldots, d_r$.  Let us set
\[
\gamma_1 = \cdots = \gamma_m = \alpha -1
\]
and
\[
\gamma_{m+1} = d_1 - 1 - \alpha, \ldots, \gamma_{m+r} = d_r - 1 - \alpha.
\]
We then have $ \sum_{i=1}^{m+r}\gamma_i =  \alpha n - 1$.

We now show that the the urn process from Proposition~\ref{prop:urn} naturally occurs within our tree evolving according to Marchal's algorithm.  Colours $1, 2, \ldots, m$ represent the different edges of $\mathsf{T}_k^{\mathrm{ord}}$ and colours $m+1, \ldots, m+r$ represent the different vertices.  For edge $e_i$ of $\mathsf{T}_k^{\mathrm{ord}}$, type $a$ corresponds to the weight of edges inserted along $e_i$; type $b$ corresponds to the weight at vertices along $e_i$; and type $c$ corresponds to the weight in vertices and edges in pendant subtrees hanging off $e_i$ (excluding their roots along $e_i$).  So $X^a_i(n)=(\alpha-1) L_n(e_i)$, $X_i^b(n) = N_n(e_i) +(1-\alpha)(L_n(e_i)-1)$ and $X^c(n) = \alpha M_n(e_i) - N_n(e_i)$.  For vertex $v_j$ of $\mathsf{T}_k^{\mathrm{ord}}$, types $a$ and $b$ together correspond to the weight at $v_j$ and type $c$ corresponds to the weight in edges and vertices in subtrees hanging from $v_j$.  So $X^a_{m+j}(n) + X^b_{m+j}(n) = N_n(v_j) -1 - \alpha$ and $X_{m+j}^c(n) = \alpha M_n(v_j) - N_n(v_j)+d_j$. Applying Proposition~\ref{prop:urn} and Lemma~\ref{lem:grandeurslimites} then yields the claimed distributions for the $L(e_i)$, $N(e_i)$, $M(e_i)$, $N(v_j)$ and $M(v_j)$.

We now turn to $N_n(e_i,\ell), \ell \ge 1$, the ordered numbers of edges attached to the branchpoints along $e_i$.  Independently for $1 \le i \le m$, let $(C_{i, \ell}(n), \ell \ge 1)$ be a Chinese restaurant process with $\beta = \theta = \alpha -1$.  This evolves in exactly the same way as Marchal's algorithm adds new edges along $e_i$.  In particular, we have
\[
(N_n(e_i,\ell), \ell \ge 1) = (C^{\downarrow}_{i,\ell}(N_n(e_i)), \ell \ge 1).
\]
By again composing limits, it follows that
\[
\left(\frac{N(e_i,\ell)}{N(e_i)}, \ell \ge 1\right) \sim \mathrm{PD}(\alpha-1,\alpha-1),
\]
independently for $1 \le i \le m$ and independently of everything else.

Let us now consider how the local time is distributed among the corners of the vertices $v_j$.  This again follows from an urn argument: for the vertex $v_j$ which has degree $d_j$, consider an urn with $d_j$ colours, one corresponding to each corner, $(A_{m+j,1}(n), \ldots A_{m+j, d_j}(n))_{n \ge 0}$.  Start the urn from a single ball of each colour.  Then whenever we insert an edge into the corresponding corner, we increase the number of positions into which we can insert new edges by 1.  Hence, we have precisely P\'olya's urn (see Section~\ref{sec:urns} for a definition) and so by Theorem~\ref{thm:urn2},
\[
\frac{1}{n} (A_{m+j,1}(n), \ldots, A_{m+j,d_j}(n)) \to (\Delta_1, \ldots, \Delta_{d_j})
\]
almost surely, where $(\Delta_1, \ldots, \Delta_d) \sim \mathrm{Dir}(1,1,\ldots,1)$.  We have
\[
(N_n(v_j,\ell), 1 \le \ell \le d_j) = (A_{m+j,\ell}(N_n(v_j)) - 1, 1 \le \ell \le d_j)
\]
and it follows that
\[
\left(\frac{N(v_j,\ell)}{N(v_j)}, 1 \le \ell \le d_j\right) \sim \mathrm{Dir}(1,1,\ldots,1),
\]
independently for $1 \le j \le r$ and independently of everything else.

A similar argument works for the local time to the left and right of the $\ell$th largest vertex along an edge $e_i$: start a two-colour urn $(A_{i,\ell,1}(n), A_{i,\ell,2}(n))_{n \ge 0}$ from one ball of each colour and at each step add a single ball of the picked colour.  Then, again by Theorem~\ref{thm:urn2},
\[
\frac{1}{n} (A_{i,1}(n), A_{i,2}(n)) \to (\Delta,1-\Delta)
\]
almost surely, where $\Delta \sim \mathrm{U}[0,1]$.  We get
\[
N^{\mathrm{right}}_n(e_i,\ell) = A_{i, 2}(N^n(e_i, \ell)) - 1
\]
and so it follows that
\[
\frac{N^{\mathrm{right}}(e_i,\ell)}{N(e_i,\ell)} \sim \mathrm{U}[0,1],
\]
independently for $1 \le i \le m$ and $\ell \ge 1$.
\end{proof}

\begin{rem}\label{rem:distribution temps local aretes et noeuds}
Let $N(T):=N(e_1)+\dots+N(e_m)+N(v_1)+\dots+N(v_r)$. Using Remark~\ref{rem:identityMLDir} below,  we observe the following distributional relation: we have $N(T)\sim \mathrm{ML}(1/\alpha, k-1/\alpha)$ and, independently,
\begin{align*}
\left(\frac{N(e_1)}{N(T)},\dots,\frac{N(e_m)}{N(T)},\frac{N(v_1)}{N(T)},\dots,\frac{N(v_r)}{N(T)}\right)\sim \mathrm{Dir}(\alpha-1,\dots,\alpha-1,d_1-1-\alpha,\dots,d_r-1-\alpha).
\end{align*}
\end{rem}

\subsection{Construction of the stable graphs}
\label{sec:SLconfiguration}

\textbf{Construction from \cite{CKG17+}.}
Returning now to the setting of our graphs, we wish to specify the distribution of the limiting sequence  $\mathrm{C}_i = (C_i, d_{C_i}, \mu_{C_i})$, $i \ge 1$ arising in Theorem \ref{th:C-K+G}.  The details of the following can be found in the paper \cite{CKG17+}.
Our graph notation was introduced in Section~\ref{sec:motivation} and the processes $\xi, X, H$ were introduced in Section~\ref{subsec:alphastable}.

We first define a real-valued process $\tilde{\xi}$ via a change of measure from the L\'evy process $\xi$.  To this end, we observe first that $\left(\exp\left(-\int_0^t s \mathrm{d} \xi_s - \frac{t^{\alpha+1}}{(\alpha+1)}\right), t \ge 0\right)$ is a martingale. Now for each $t \ge 0$ and any suitable test-function $f: \mathbb{D}([0,t], \R) \to \R$, define $\tilde{\xi}$ by
\[
\E{f(\tilde{\xi}_s, 0 \le s \le t)} = \E{\exp \left(-\int_0^t s \mathrm{d} \xi_s - \frac{t^{\alpha+1}}{(\alpha+1)} \right) f(\xi_s, 0 \le s \le t)}.
\]
Superimpose a Poisson point process of rate $A_{\alpha}^{-1}$ (as defined in (\ref{A_alpha})) in the region $\{(t,y) \in \R_+ \times \R_+: y \le \tilde \xi_t - \inf_{0 \le s \le t} \tilde \xi_s\}$.  Then the limiting components $\mathrm{C}_i,i\geq 1$ are encoded by the excursions of the reflected process $(\tilde \xi_t - \inf_{0 \le s \le t} \tilde \xi_s, t \ge 0)$ above 0 and the Poisson points falling under each such excursion.  The total masses of the measures $\mu_{C_1}(C_1), \mu_{C_2}(C_2), \ldots$ are given by the lengths of the excursions of $\tilde{\xi}$ above its running infimum.  The surpluses $s(\mathrm{C}_1), s(\mathrm{C}_2), \ldots$ are given by the the number of Poisson points falling under corresponding excursions.
Then, the limiting components $(\mathrm{C}_1, \mathrm{C}_2, \ldots)$ are conditionally independent given the sequences $(\mu_{C_1}(C_1), \mu_{C_2}(C_2), \ldots)$ and $(s(\mathrm{C}_1), s(\mathrm{C}_2), \ldots)$, with
\[
\big(C_i, d_{C_i}, \mu_{C_i} \big)\equidist \big(\mathcal{G}^{s(\mathrm{C}_i)}, \mu_{C_i}(C_i)^{1-1/\alpha} \cdot d^{s(\mathrm{C}_i)}, \mu_{C_i}(C_i) \cdot \mu^{s(\mathrm{C}_i)}\big).
\]

\textbf{Construction of the connected $\alpha$-stable graph with surplus $s$.}
For $s \ge 0$, it remains to describe the connected stable graph, $(\mathcal{G}^s, d^s, \mu^s)$ with surplus $s$.  Just as the stable tree is encoded by a normalised excursion of $\xi$, the space $\mathcal{G}^s$ has a spanning tree which is encoded by a normalised excursion of $\tilde{\xi}$ conditioned to contain $s$ Poisson points.  This turns out to be distributed as follows.  First sample excursions $X^s$ and $H^s$ with joint law specified by
\[
\E{f(X^s(t), H^s(t), 0 \le t \le 1)} = \frac{\E{\left(\int_0^{1} X(u) \mathrm{d}u \right)^s f(X(t), H(t), 0 \le t \le 1)}}{\E{\left(\int_0^{1} X(u) \mathrm{d}u \right)^s}}.
\]
Let $\mathcal{T}^s$ be the $\R$-tree encoded by $H^s$ and let $\pi^s: [0,1] \to \mathcal{T}^s$ be its canonical projection.  If $s = 0$, then $X^s$ is a standard stable excursion and $H^s$ is its corresponding height process i.e.\ $\mathcal{T}^0 \equidist \mathcal{T}$. In this case, we simply set $\mathcal{G}^0 = \mathcal{T}^0$.  If, on the other hand, $s \ge 1$, conditionally on $X^s$ and $H^s$, sample conditionally independent points $V^s_1, V^s_2, \ldots, V^s_s$ from $[0,1]$, each having density
\[
\frac{X^s(u)}{\int_0^1 X^s(t) \mathrm{d}t}, \quad u \in [0,1].
\]
Then, for $1 \le k \le s$, let $Y^s_{k}$ be uniformly distributed on the interval $\intervalleff{0}{X^s(V^s_k)}$, independently for all $k$, and let $B^s_{k} = \inf \{t \ge V^s_{k}: X^s(t) = Y^s_{k} \}$.  
We obtain $\mathcal{G}^s$ from $\mathcal{T}^s$ by identifying the pairs of points $(\pi^s(V^s_k), \pi^s(B^s_k))$ for $1 \le k \le s$.  (This is achieved formally by a further  straightforward quotienting operation which we do not detail here.) 

In fact, using the notation of Section~\ref{subsec:alphastable} for the tree $\cT^s$ (which is absolutely continuous with respect to $\cT$), this last operation corresponds to identifying the leaf $\pi^s(V^s_k)$ with a branchpoint on its ancestral line $\oseg{\rho}{\pi^s(V^s_k)}$, independently for $1\leq k \leq s$. As a consequence of the discussion in Section~\ref{sec:stable}, the point $\pi^s(B_k)$ is such that
\[
\pi^s(B^s_k)=\pi^s(A^s_k), \quad \text{where} \quad A^s_{k} = \sup \Big\{t \le V^s_{k}: X^s(t) \leq Y_k^s \Big\}.
\]
Along with equation \eqref{eq:local time on the right of the path from the root}, this ensures that each branchpoint $b\in \oseg{\rho}{\pi^s(V^s_k)}$ is chosen with probability equal to 
\begin{align*}
\frac{N^{\mathrm{right}}(b,\pi^s(V^s_k))}{N^{\mathrm{right}}(\oseg{\rho}{\pi^s(V^s_k)})}= \frac{N^{\mathrm{right}}(b,\pi^s(V^s_k))}{X(V^s_k)},
\end{align*}
as claimed in the introduction. We view $\cG^s$ as a measured metric space by endowing it with $\mu^s$, the image of the Lebesgue measure on $[0,1]$ by the projection $\pi^s$.  

\textbf{Continuous and discrete marginals.}
Recall the definition for any $n\geq 0$ of the continuous marginals $\cG_n^s$ from the introduction: $\cG_n^s$ is  the union of the kernel $\mathcal{K}^s$ and the paths from $n$ leaves to the root, where the leaves are taken i.i.d\ under the measure carried by $\cG^s$. Indeed, the kernel is the image of the subtree of $\cT^s$ spanned by the $s$ selected leaves after the gluing procedure.

Let $(U_i)_{i\geq 1}$ be a sequence of i.i.d.\ $\mathrm{U}[0,1]$ random variables independent of $X^s$, and let $n\geq 0$. In the construction described above, let $\cT_{s,n}^{s}$ be the ordered subtree of $\cT^s$ spanned by the root and the leaves corresponding to the real numbers $V^s_1, \dots , V_s^s,U_1,\dots, U_n$, and $\cT_{s,n}^{s,\mathrm{ord}}$ its ordered version. Since $\pi^s(U_1),\dots, \pi^s(U_n)$ are (by definition) distributed according to the probability measure carried by $\cG^s$, the image of $\cT_{s,n}^s$ after the gluing procedure is a version of the continuous marginal $\cG_n^s$ (and the discrete marginal $\sG_n^s$ is then the combinatorial shape of the continuous marginal $\cG_n^s$). 

For future purposes, we also define $\sT_{s,n}^{s,\mathrm{ord}}$ the discrete counterpart of $\cT_{s,n}^{s,\mathrm{ord}}$. By convention, we consider that the $s$ first leaves are unlabelled and the $n$ leaves corresponding to $U_1, \dots , U_n$ inherit the label of their uniform variable.

\textbf{Unbiasing.}
Let $(X;V_1,V_2,\dots, V_s,Y_1,\dots,Y_s)$ be the unbiased excursion endowed with
\begin{itemize} 
\item $V_1,\dots, V_s$ i.i.d.\ $\mathrm{U}[0,1]$ random variables
\item $Y_1,\dots, Y_s$ which are conditionally independent given $(X;V_1,V_2,\dots V_s)$, with $Y_k \sim \mathrm{U}[0,X(V_k)]$.
\end{itemize}
We call $(X;V_1,V_2,\dots V_s,Y_1,\dots,Y_s)$ the \emph{unbiased counterpart} of $(X^s;V^s_1,\dots,V^s_s,Y^s_1,\dots Y^s_s)$. Any random object defined as a measurable function $f(X^s;(V^s_k)_{1\leq k \leq s},(Y^s_k)_{1\leq k\leq s},(U_i)_{i\geq 1})$ then also has an unbiased counterpart, $f(X;(V_k)_{1\leq k \leq s},(Y_k)_{1\leq k\leq s},(U_i)_{i\geq 1})$ and vice versa. 
Using the fact that, conditionally on $(X;V_1,V_2,\dots V_s)$, the random variables $Y_1,\dots,Y_s$ have the same distribution as $Y^s_1,\dots,Y^s_s$ conditionally on $(X^s;V^s_1,V^s_2,\dots V^s_s)$, we observe that 
\begin{align}\label{eq:change of measure}
& \E{f(X^s;(V^s_k)_{1\leq k \leq s},(Y^s_k)_{1\leq k\leq s},(U_i)_{i\geq 1})}  \notag \\
& = \frac{\E{\int_{\intervalleff{0}{1}^s}\mathrm{d}v_1\ldots \mathrm{d}v_s  \frac{X(v_1)\ldots X(v_s)}{\left( \int_0^1 X(t) \mathrm{d}t \right)^s} \int_{0}^{X(v_1)}\frac{\mathrm{d}y_1}{X(v_1)}\ldots\int_{0}^{X(v_s)}\frac{\mathrm{d}y_s}{X(v_s)} f(X; (v_k),(y_k),(U_i)) \left( \int_0^1 X(t) \mathrm{d}t \right)^s}}{ \E{\left( \int_0^1 X(t) \mathrm{d}t \right)^s} } \notag \\
& =  \frac{\E{f(X; (V_k)_{1\leq k \leq s}, (Y_k)_{1\leq k\leq s},(U_i)_{i\geq 1}) X(V_1)X(V_2)\ldots X(V_s)}}{ \E{X(V_1) X(V_2) \ldots X(V_s)}}. 
\end{align}
In particular, this allows us to compute quantities in the unbiased setting in order to understand the biased one. We define $\widehat{\mathcal{G}}^s$ to be the unbiased counterpart of $\cG^s$ and $\widehat{\mathcal{G}}^s_n$ to be the unbiased counterpart of $\mathcal{G}^s_n$ and $\widehat{\sG}^s_n$ to be the unbiased counterpart of $\sG_n^s$.  Similarly, $\widehat{\cT}_{s,n}^{s,\mathrm{ord}}$ is the unbiased counterpart of $\cT_{s,n}^{s,\mathrm{ord}}$ which, modulo the labelling of the leaves, has the same distribution as $\cT_{s+n}^{\mathrm{ord}}$.


\section{Distribution of the marginals $\mathsf{G}_{n}^s$}
\label{sec:marginals}


Let $s \geq 0$. The goal of this section is to identify the joint distribution of the marginals $\mathsf{G}^s_{n}$, for $n \geq 0$ (and for $n \ge -1$ if $s \ge 2$). By definition, for any $n \geq 0$, the random graph $\mathsf{G}^s_{n}$ is an element of $\M_{s,n}$, the set of connected multigraphs with surplus $s$, with $n+1$ labelled leaves, unlabelled internal vertices and no vertex of degree $2$. 
To perform our calculations, it will be convenient to consider versions of this multigraph with some additional structure, namely cyclic orderings of the half-edges around each vertex. 
We denote by $\M_{s,n}^{\mathrm{ord}}$ the set of such graphs and we emphasise here that the orderings around different vertices need not be compatible with one another: the elements of $\M^{\mathrm{ord}}_{s,n}$ are not necessarily planar. The advantage is that this additional structure breaks the symmetries present in elements of $\M_{s,n}$. (For $n=-1$ the cyclic ordering is insufficient to break all the symmetries and we will rather label the internal vertices.)

We will begin in Section~\ref{sec:cycling} by computing the number of possible cyclic orderings of the half-edges around the different vertices of a graph $G\in \M_{s,n}$. Then, in Section~\ref{sec:bijectiontreegluing}, we will describe the elements of $\M_{s,n}^{\mathrm{ord}}$ as ordered trees with $n$ labelled and $s$ unlabelled leaves together with a ``gluing plan'', that specifies how to glue each unlabelled leaf ``to the right" of the ancestral path of that leaf. This description corresponds to the one we have for $\mathsf{G}^s_{n}$, and we compute in Section~\ref{sec:calculations} the distribution of the tree and the corresponding gluing plan, which then yields the distribution of $\mathsf{G}^s_{n}$ claimed in Theorem \ref{th:distr_marginals}. In Section~\ref{sec:Marchal}, we show that the sequence $\left(\mathsf{G}^s_{n}\right)_{n \ge 0}$ evolves according to Marchal's algorithm (Theorem \ref{thm:MarchalAlg}). In Section~\ref{sec:unrooted}, we extend this to $(\mathsf{G}_{n}^s)_{n \ge -1}$ for $s \ge 2$.  Finally, Section~\ref{sec:configembedded} is devoted to the proof of Corollary~\ref{cor:identification}, which identifies the distribution of $\mathsf{G}^s_{n}$ with that of a specific configuration model with i.i.d.\ random degrees.

We recall the following notation from the introduction. For each $G=(V(G),E(G))\in \M_{s,n}$, we denote $I(G)\subset V(G)$ the set of internal vertices of $G$ (vertices of degree $3$ or more), $\deg(v) = \deg_G(v)$ the degree of a vertex $v \in V(G)$, $\mathrm{sl}(G)$ the number of self-loops, $\mathrm{mult}(e)$ the multiplicity of the element $e \in \supp(E)$ and $\mathrm{Sym}(G)$ the set of permutations of vertices of $G$ that are the identity on the leaves and that preserve the adjacency relations (with multiplicity).

\subsection{Cyclic orderings of half-edges}\label{sec:structure multigraph}
\label{sec:cycling}

Let $n\geq 0$. In this section we compute the number of possible cyclic orderings of the half-edges around each vertex of $G$, for each $G\in \M_{s,n}$ (we emphasise that Lemma \ref{lem:co} is false when $n=-1$ and $s\geq 2$). Let $\psi:\M_{s,n}^{\mathrm{ord}}\rightarrow \M_{s,n}$ be the map that forgets the cyclic ordering around the vertices.

\begin{lem}
\label{lem:co}
For each $G \in \M_{s,n}$,
\begin{equation*}
\label{eq:nb of cyclic orderings} \left|\psi^{-1}(G) \right|=\frac{\prod_{v\in I(G)}(\deg(v) -1)!}{\abs{\mathrm{Sym}(G)} 2^{\mathrm{sl}(G)} \prod_{e\in \supp(E(G))}\mathrm{mult}(e)!}.
\end{equation*}
\end{lem}

\begin{proof}
It is convenient to consider versions of $G$ with labelled internal vertices. The number of possible labellings is
\begin{equation} \label{eqn:symmetries}
\frac{\abs{I(G)}!}{\abs{\mathrm{Sym}(G)}}.
\end{equation}
Indeed, let $\tilde{G}$ denote an arbitrarily labelled version of $G$. The symmetric group $\mathfrak{S}_{\abs{I(G)}}$ acts on the set of multigraphs with $\abs{I(G)}$ internal labels by permuting those labels. The number of labellings we seek is thus the number of elements of the orbit of $\tilde{G}$ under this action. This is just $\abs{I(G)}!$ divided by the cardinality of the stabilizer of $\tilde{G}$. Any permutation $\sigma\in \mathfrak{S}_{\abs{I(G)}}$ that fixes $\tilde{G}$ corresponds to a permutation $\tau\in\mathrm{Sym}(G)$, and (\ref{eqn:symmetries}) follows.

Now, to compute  $\left|\psi^{-1}(G) \right|$,  we first label everything then forget the labels we do not need.
\begin{itemize}
	\item Consider version of $G$ with labelled internal vertices: from the preceding paragraph, there are $\frac{\abs{I(G)}!}{\abs{\mathrm{Sym}(G)}}$ possible labellings.
	\item For each $e=\{u,v\}\in \supp(E(G))$, in order to distinguish between the $\mathrm{mult}(e)$ edges joining $u$ and $v$, number them from $1$ to $\mathrm{mult}(e)$.
	\item Give every self-loop an orientation.
	\item Endow the multigraph with a cyclic ordering around each vertex. For each $v\in I(G)$ we have $(\deg(v) -1)!$ possibilities for an ordering of the half edges adjacent to $v$. (The half-edges are distinguishable because the self-loops are oriented.)
	\item Forget the orientation on the self-loops. This transformation is $2^{\mathrm{sl}(G)}$-to-$1$ since with the ordering around the vertices, every orientation is distinguishable.
	\item Forget the labelling of the edges. This transformation is $\left(\prod_{e\in \supp(E(G))} \mathrm{mult}(e)!\right)$-to-$1$.
	\item Forget the labelling of the internal vertices. With the cyclic ordering around the vertices every vertex is distinguishable, and so this map is $\abs{I(G)}!$-to-$1$.
\end{itemize}
(We emphasise here the importance of the fact that our multigraphs are planted in distinguishing edges and vertices.)
We obtain a multigraph in $\M_{s,n}^{\mathrm{ord}}$ whose image by $\psi$ is $G$. By the previous considerations, the number of such multigraphs is indeed given by the claimed formula.
\end{proof}

\subsection{Ordered multigraphs and the depth-first tree}
\label{sec:bijectiontreegluing}

We still consider integers $n\geq 0$.

\textbf{Ordered trees with paired leaves.}
Let $\A_{s,n}$ be the set of planted ordered trees with no vertices of degree $2$ that have $s$ unlabelled leaves and $n$ labelled leaves, with labels from $1$ to $n$. Let $\A_{s,n}^{\mathrm{pair}}$ be the set of ordered trees with no vertices of degree 2 that have $n$ labelled uncoloured leaves, $s$ red leaves labelled $1$ to $s$ in clockwise order from the root, and $s$ blue leaves also labelled from $1$ to $s$.  We think of the red and blue leaves labelled $i$ as forming a pair, and impose the condition that the blue leaf labelled $i$ must lie to the right of the ancestral line of the red leaf labelled $i$, for $1 \le i \le s$.

We first describe how every ordered multigraph $G\in \M_{s,n}^{\mathrm{ord}}$ is equivalent to an element of $\A_{s,n}^{\mathrm{pair}}$.
We define two natural maps on $\A_{s,n}^{\mathrm{pair}}$. Let
\[\textsc{Glue}:\A_{s,n}^{\mathrm{pair}}\rightarrow\M_{s,n}^{\mathrm{ord}}\]
be the map that, for each red leaf $i$ identifies $i$ with its blue pair and then contracts the resulting path containing a vertex of degree 2 into a single edge.
Let \[\textsc{Erase}:\A_{s,n}^{\mathrm{pair}} \rightarrow \mathrm \A_{s,n}\] be the map that erases the blue leaves and their adjacent edges, then contracts any path of degree 2 vertices into a single edge, and finally forgets the labelling and colour of the red leaves.

\begin{figure}
	\centering
	\includegraphics[height=5cm]{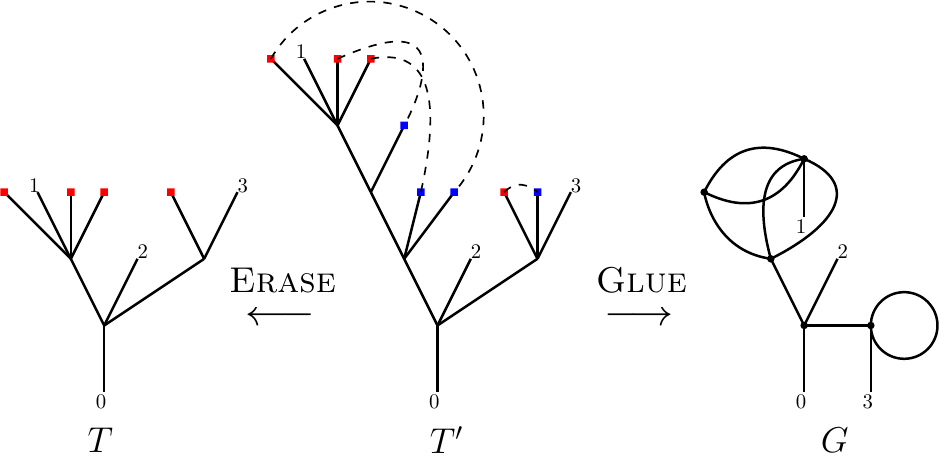}
	\caption{The operations \textsc{Glue} and \textsc{Erase} applied to a tree $T'$. Here, $T'$ is the depth-first tree of $G$, and $T$ is the base tree.}
	\label{fig:test}
\end{figure}

\textbf{Reverse construction: the depth-first tree.} 
Let $G\in \M_{s,n}^{\mathrm{ord}}$. We imagine that each edge of $G$ is made up of two half-edges, one attached to each end-point. We say that two half-edges are \emph{adjacent} if they have a common end-point. We describe a procedure that explores all the half-edges of the graph in a deterministic manner and disconnects exactly $s$ edges in order to transform $G$ into a tree. At each step $i$ of the algorithm, we will have an ordered stack of \emph{active} half-edges $A_i$ and a current surplus $s_i$. We write $h_0$ for the unique half-edge connected to the leaf with label $0$.
\newpage
\begin{steplist}
	\item[{\sc Initialization}] $A_0=(h_0)$, $s_0 = 0$.
	\item[{\sc Step $i$}\hspace{38pt}]$(0\leq i\leq |E(G)|-1)$: Let $h_{i}$ be the half-edge at the top of the stack $A_i$. Let $\hat{h}_i$ be the half-edge to which it is attached. If $\hat{h}_i\notin A_i$, remove $h_i$ from the stack and put the half-edges adjacent to $\hat{h}_i$ on the top of the stack, in clockwise order top to bottom. If $\hat{h}_i\in A_i$, first increment $s_i$, then remove both $h_i$ and $\hat{h}_i$ from the stack, disconnect them, attach a red leaf labelled $s_i$ to $h_i$ and attach a blue leaf labelled $s_i$ to $\hat{h}_i$.
\end{steplist} 
It is straightforward to check that this algorithm produces a tree in $\A_{s,n}^{\mathrm{pair}}$, which we call the \emph{depth-first tree}, and denote by $\textsc{Dep}(G)$. (Note that this is a variant of the notion of depth-first tree introduced in \cite{ABBrGo12}.) We have $\textsc{Dep}(G)=G$ if and only if $G$ is a tree i.e.\ $s=0$.  The following lemma is then straightforward.

\begin{lem}
	The maps $\textsc{Glue}:\A_{s,n}^{\mathrm{pair}}\rightarrow\M_{s,n}^{\mathrm{ord}}$ and $\textsc{Dep}:\M_{s,n}^{\mathrm{ord}}\rightarrow\A_{s,n}^{\mathrm{pair}}$ are reciprocal bijections.
\end{lem}

For a multigraph $G$, call $\textsc{Erase}(\textsc{Dep}(G))$ the \emph{base tree}.

\textbf{Gluing plans.}
Consider $T\in \A_{s,n}$. We now aim to describe the set $\textsc{Erase}^{-1}\left(\{T\}\right)$.  This is the set of possible depth-first trees $T'$ obtainable from a fixed base tree $T$. As usual, we write $I(T)$ for the internal vertices of $T$ and $E(T)$ for its edges. A vertex $v \in I(T)$ of degree $d=\mathrm{deg}_T(v)$ possesses $d$ \emph{corners}, which we call $c_{v,1}, \ldots, c_{v,d}$ in clockwise order from the root.  We write $C(T)$ for the set of corners of $T$. The \emph{ancestral path} of a vertex is its unique path to the root. For the $k$th unlabelled leaf of $T$ in clockwise order, let $\mathcal{A}(k)$ be the set of edges and corners that lie immediately to the right of its ancestral path, for $1 \le k \le s$.

Now let $T'\in \textsc{Erase}^{-1}\left(\{T\}\right)$.  The internal vertices of $T$ each have a counterpart in $T'$, for which we use the same name.  The red leaves of $T'$ correspond to the unlabelled leaves of $T$.  A blue leaf is attached by its incident edge either into one of the corners of an internal vertex of $T$, or to an internal vertex of $T'$ which disappears when the blue leaves are removed and paths of internal vertices of degree 2 are contracted into a single edge.  For each $e \in E(T)$ let $a_e$ be the number of additional vertices along the path in $T'$ which get contracted to yield the edge $e$ by $\textsc{Erase}$.  If $a_e \neq 0$, we will list these additional vertices as $v_{e, 1}, \ldots, v_{e,a_e}$ in decreasing order of distance from the root.

For each $v \in I(T)$, let $S_{v, \ell}$ be the set of labels of blue leaves attached to corner $c_{v, \ell}$, for $1 \le \ell \le \mathrm{deg}_T(v)$.  (Any or all of these sets may be empty; in particular, $S_{v,1}$ is always empty because a blue leaf must lie to the right of the ancestral line of the corresponding red leaf.)  If $S_{v,\ell}$ is non-empty, let $\sigma_{v,\ell}$ be the permutation of its elements which gives the clockwise ordering of the blue leaves in corner $c_{v,\ell}$; if it is empty, let $\sigma_{v,\ell}$ be the unique permutation of the empty set. For each $e \in E(T)$ such that $a_e \neq 0$, we let $S_{e,i}$ be the set of labels of blue leaves attached to vertex $v_{e,i}$ in $T'$, for $1 \le i \le a_e$.  These sets can not be empty.  Let $\sigma_{e,i}$ be the permutation of the elements of $S_{e,i}$ giving the clockwise ordering of the blue leaves attached to $v_{e,i}$ (note that these are necessarily attached to the right of $e$).  Observe that the collection of sets
\[
\left\{ S_{v,\ell} : v \in I(T), 1 \le \ell \le \mathrm{deg}_T(v), S_{v,\ell} \neq \emptyset\right\} \cup \left\{ S_{e,i}: e \in E(T), 1 \le i \le a_e\right\}
\]
partitions $\{1,2,\ldots,s\}$.  This induces a \emph{gluing function} $g: \{1,2,\ldots,s\} \to (I(T) \cup E(T)) \times \N$ as follows.  For $1 \le k \le s$, if $k \in S_{v,\ell}$ set $g(k) = (v,\ell)$; if $k \in S_{e,i}$  set $g(k) = (e,i)$.

See Figure \ref{fig:partition} for an illustration. 
This leads us to the formal definition of a gluing plan.

\begin{defn}\label{def:gluing plan}
	We say that $\Delta= \left( \left((S_{v,\ell}, \sigma_{v,\ell})_{1 \le \ell \le \mathrm{deg}_T(v)} \right)_{v \in I(T)}, \left( (S_{e,i},\sigma_{e,i})_{1\le i \le a_e} \right)_{e\in E(T)} \right)$ is a \emph{gluing plan for $T$} if the following properties are satisfied.
	\begin{enumerate}
	\item For all $v \in I(T)$ and all $1 \le \ell \le \mathrm{deg}_T(v)$, we have $S_{v,\ell} \subseteq \{1,2,\ldots,s\}$ and $\sigma_{v,\ell}$ is a permutation of $S_{v,\ell}$.
	\item For all $e\in E$ and all $1 \le i \le a_e$, the set $S_{e,i} \subseteq \{1,2,\ldots,s\}$ is non-empty and $\sigma_{e,i}$ is a permutation of $S_{e,i}$.
	\item The sets $\left\{ S_{v,\ell} : v \in I(T), 1 \le \ell \le \mathrm{deg}_T(v), S_{v,i} \neq \emptyset\right\}$ and $\left\{ S_{e,i}: e \in E(T), 1 \le i \le a_e\right\}$ partition \linebreak $\{1,2,\ldots,s\}$.
	\item The induced gluing function $g: \{1,2, \ldots,s\} \rightarrow (I(T)\cup E(T)) \times \N$ is such that if $g(k) = (v,\ell)$ then $c_{v,\ell} \in\cA(k)$ and if $g(k) = (e,i)$ then $e \in \cA(k)$, for all $1 \le k \le s$. 
	\end{enumerate}
\end{defn}

It is straightforward to see that we can completely encode a tree $T'\in \textsc{Erase}^{-1}(\{T\})$ by its gluing plan, and that conversely, every gluing plan for $T$ encodes a tree $T'\in\textsc{Erase}^{-1}(\{T\})$.

\begin{lem}
\label{eq:bijection gluing plan}
\begin{equation*}
\M_{s,n}^{\mathrm{ord}}\simeq \A_{s,n}^{\mathrm{pair}}\simeq \enstq{(T,\Delta)}{T\in \A_{s,n} \text{ and $\Delta$ is a gluing plan for $T$}}.
\end{equation*}
\end{lem}

\begin{figure}
	\centering
	\subfloat[In a corner]{\includegraphics[width=4cm]{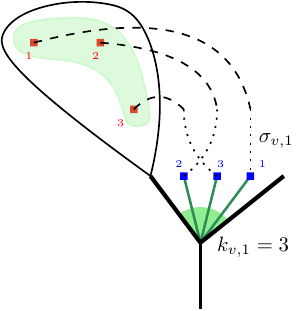}}
	\qquad \subfloat[On an edge]{\includegraphics[width=6cm]{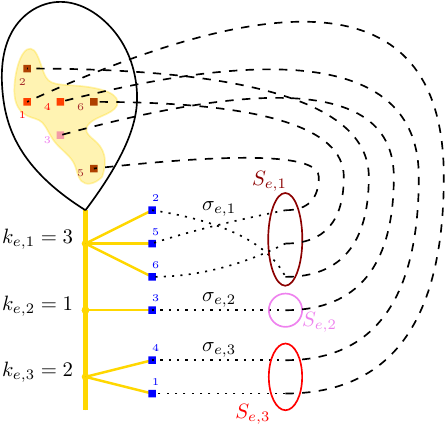}}
	\caption{Definition of a gluing plan}	\label{fig:partition}
\end{figure}

Suppose $T \in \A_{s,n}$ and that
\[
\Delta=\left( \left((S_{v,\ell}, \sigma_{v,\ell})_{1 \le \ell \le \mathrm{deg}_T(v)} \right)_{v \in I(T)}, \left( (S_{e,i},\sigma_{e,i})_{1\le i \le a_e} \right)_{e\in E(T)} \right)
\]
is a gluing plan for the base tree $T$.  We let $k_{v,\ell} = |S_{v,\ell}|$ be the number of blue leaves attached into corner $c_{v,\ell}$ and $k_v = \sum_{\ell=1}^{\mathrm{deg}_T(v)} k_{v,\ell}$ be the total number of blue leaves attached to $v$.  We let $k_{e,i} = |S_{e,i}|$ be the number of blue leaves attached to the $i$th vertex inserted along $e$ and let $k_e = \sum_{i=1}^{a_e} k_{e,i}$ be the total number of blue leaves attached to vertices along $e$.
We call the family of numbers 
\[
\left( (k_v,(k_{v,\ell})_{1\leq \ell \leq \deg_T(v)})_{v\in I(T)}, (k_e,a_e,(k_{e,i})_{1\leq i\leq a_e})_{e\in E(T)} \right)
\]
the \emph{type} of the gluing plan $\Delta$.
\begin{rem}
	Suppose that $G \in \M_{s,n}^{\mathrm{ord}}$ corresponds to $(T,\Delta)$.  The degrees in $G$ depend only on $T$ and the type of the gluing plan $\Delta$. 
	For an internal vertex $v$ of $G$ that was already present in $I(T)$, its degree in $G$ is $\deg_G(v)=\deg_T(v)+k_v$. The internal vertices of $G$ that do not correspond to internal vertices of $T$ are the ones that were created along the edges of $T$ during the gluing procedure. For each $e\in E(T)$, there are $a_e$ newly-created vertices along the edge $e$, having degrees $2+k_{e,1}, 2+k_{e,2}, \dots , 2+k_{e,a_e}$.
\end{rem}

\subsection{The distribution of $\sG^s_{n}$}
\label{sec:calculations}

The goal of this section is to prove Theorem~\ref{th:distr_marginals}~for $n\geq 0$, which states that for every connected multigraph $G\in \mathbb M_{s,n}$, 
\[\Pp{\mathsf{G}_{n}^s=G}\propto\frac{\prod_{v\in I(G)} w_{\deg(v)-1}}{\abs{\mathrm{Sym}(G)}2^{\mathrm{sl}(G)}\prod_{e\in \supp(E(G))}\mathrm{mult}(e)!},\]
where the weights $(w_k)_{k\geq 0}$ are defined in \eqref{Marchal's weights}.

Recall the construction of the random \R-graph $\cG^{s}$ using a tilted excursion and biased chosen points $(X^s;V_1^s,\dots,V_s^s)$ from Section \ref{sec:SLconfiguration}. Recall also the definitions of $\cT_{s,n}^{s,\mathrm{ord}}$ (and its discrete version $\sT_{s,n}^{s,\mathrm{ord}}$) and $\cG_{s,n}^s$ (and its discrete version $\sG_{n}^{s}$), using an additional sequence of i.i.d.\ uniform random variables $(U_i)_{i\geq 1}$.   
In order to apply the results of the previous section, we want to work with ordered versions of our graphs.  In particular, we will get an ordered version $\sG_n^{s,\mathrm{ord}}$ of $\sG_n^s$ by applying a gluing plan to the base tree $\sT_{s,n}^{s,\mathrm{ord}}$. The change of measure \eqref{eq:change of measure} allows us to make calculations using the unbiased excursion with uniform points $(X;V_1,\dots,V_s, U_1, \ldots, U_n)$. So we will define and work instead with an unbiased version $\widehat{\mathsf{G}}_{n}^{s,\mathrm{ord}}$, derived from the unbiased version $\widehat{\mathsf{T}}_{s,n}^{s,\mathrm{ord}}$ of $\mathsf{T}_{s,n}^{s,\mathrm{ord}}$.

\textbf{Construction of $\widehat{\sG}_n^{s,\mathrm{ord}}$.}
We define $\widehat{\sG}_n^{s,\mathrm{ord}}$ via a random gluing plan $\Delta$ for $\widehat{\sT}_{s,n}^{s,\mathrm{ord}}$. 
Conditionally on $\widehat{\sT}^{s,\mathrm{ord}}_{s,n}=T\in \mathbb{T}_{s,n}$, let
\begin{align*}
W(T)&:=\{(v,\ell): v\in I(T),\ 1 \leq \ell \leq \deg_T(v)\} \cup \{(e,j) : e\in E(T),\ j \geq 1\} \subset(I(T) \cup E(T)) \times\N.
\end{align*}
This indexes all the atoms of local time in the corners (as usual, ordered clockwise around each internal vertices) and along the edges (ordered by decreasing local time in this instance) of the ordered tree $\cT_{s,n}^{s,\mathrm{ord}}$. We will often abuse notation and think of the elements of $W(T)$ as the atoms themselves. In fact, the tree $\cT_{s,n}^{s,\mathrm{ord}}$ has, up to the labelling of the leaves, the same distribution as $\cT_{s+n}^{\mathrm{ord}}$, so using the discussion just before Lemma \ref{lem:grandeurslimites}, we can decompose the whole (unbiased) stable tree as
\begin{align*}
\cT_{s,n}^{\mathrm{ord}}\cup \bigcup_{w\in W(T)}\cT(w).
\end{align*}
In order to define our gluing plan, we need to be a little careful about labelling.  For $1 \le k \le s$, let $l_k \in \{1,2,\ldots,s\}$ be the position of $V_k$ in the increasing ordering of $V_1, \ldots, V_s$ i.e.\ $l_k = \#\{1 \le j \le s: V_j \le V_k\}$. This gives the relative planar position of the (unlabelled) leaf in $T$ corresponding to $V_k$.  Recall that $\mathcal{A}(l_k)$ is then the set of edges and corners that lie immediately to the right of the ancestral path of this leaf. Almost surely, the value $B_k = \inf \{t \ge V_{k}: X(t) = Y_{k}\}$ is such that there exists an element $w_k \in W(T)$ along the ancestral line of  the leaf in $T$ corresponding to $V_k$, such for $\epsilon$ small enough, the canonical projection of an $\epsilon$-neighbourhood around $B_k$ lies completely within some subtree hanging off $\mathcal{T}_{s,n}^{\mathrm{ord}}$ i.e.\
\begin{align*}
\pi\left(\intervalleoo{B_k-\epsilon}{B_k+\epsilon}\right)\subset \cT(w_k).
\end{align*}
For $1\leq k\leq s$, for the $j$th largest atom of local time along an edge $e\in \cA(l_k)$ and every corner $(v,\ell)\in\cA(l_k)$ on the right of the ancestral path of the root to $l_k$, conditionally on $(X;V_1,V_2,\dots V_s,U_1,U_2,\dots U_n)$ we have
\begin{align*}
w_k = 
\begin{cases}
(v,\ell) & \text{with probability $\frac{N(v,\ell)}{X(V_k)}$, for $v \in I(T), 1 \le \ell \le \deg_T(v)$},\\
 (e,j) &  \text{with probability $\frac{N^{\mathrm{right}}(e,j)}{X(V_k)}$, for $e \in E(T), j \ge 1$,}
\end{cases}
\end{align*}
independently for all $k$.  For each edge $e\in E(T)$, let $a_e$ be the number of distinct atoms of local time which appear among $w_1, \ldots, w_s$.  If $a_e \ge 1$, we denote by $j_1,j_2,\dots j_{a_e}$ the values in the set $\{j\geq 1: (e,j)\in \{w_1, \ldots, w_s\}\}$ (that is, the indices of the atoms along $e$ that receive at least one gluing) listed now in decreasing order of height i.e.\ such that $L(e, j_1) > L(e, j_2) > \cdots > L(e, j_{a_e})$.  The probability that for any fixed set $\{j_1, \ldots, j_{a_e}\}$ of distinct indices we have $L(e, j_1) > L(e, j_2) > \cdots > L(e, j_{a_e})$ is $1/a_e!$, since the random variables $L(e, j_1), \ldots, L(e, j_{a_e})$ are exchangeable and distinct with probability 1, by Proposition~\ref{prop:jointlaws}. Moreover, again by Proposition~\ref{prop:jointlaws}, these random variables are independent of the local times. For $1 \le k \le s$, let
\[
g(l_k) = \begin{cases}
		(v,\ell) & \text{ if $w_k = (v,\ell)$ for some $v \in I(T)$ and some $1 \le \ell \le \deg_T(v)$} \\
		(e,i) & \text{ if $w_k = (e, j_{i})$ for some $e \in E(T)$ and some $1 \le i \le a_e$}.
            \end{cases}
\]
This is the required gluing function for $T$.  We now derive the full gluing plan.
For $e \in E(T)$ such that $a_e \ge 1$ and $1 \le i \le a_e$, let $S_{e,i} = g^{-1}(\{(e,i)\})$ be the set of leaves mapped to the $i$th atom in decreasing order of height along the edge $e$.  Define a permutation $\sigma_{e,i}$ of $S_{e,i}$ by 
\[
\sigma_{e,i}(l_k) =  \#\{1 \le j \le s: l_j \in S_{e,i}, Y_j \ge Y_k\}.
\]
Similarly, for any $(v,\ell)\in C(T)$, we define $S_{v,\ell} = g^{-1}(\{(v,\ell)\})$ and a permutation $\sigma_{v,\ell}$ of $S_{v,\ell}$ by 
\[
\sigma_{v,\ell}(l_k)= \#\{1\leq j\leq s: l_j\in S_{v,\ell}, Y_j \geq Y_k\}.
\]
Since $Y_1, \ldots, Y_k$ are conditionally independent given $(X; V_1, \ldots, V_s, U_1, \ldots, U_n)$, we see that the permutations are conditionally independent.  Conditionally on corresponding to the same atom of local time, the relative ordering of the associated $Y_k$'s is uniform, so that the permutations are all uniform on their label-sets. By construction,
\[
\Delta=\left( \left((S_{v,\ell}, \sigma_{v,\ell})_{1 \le \ell \le \mathrm{deg}_T(v)} \right)_{v \in I(T)}, \left( (S_{e,i},\sigma_{e,i})_{1\le i \le a_e} \right)_{e\in E(T)} \right)
\]
is a gluing plan for $T$. We call $\widehat{\mathsf{G}}_{n}^{s,\mathrm{ord}}$ the corresponding (random) multigraph in $\M_{s,n}^{\mathrm{ord}}$, obtained via the bijection of Lemma~\ref{eq:bijection gluing plan}. 

For $n \ge 1$, let $\N^{n, \neq} = \{(j_1, \ldots, j_n) \in \N^n: \text{ $j_1, j_2, \ldots, j_n$ are distinct.} \}$.

\begin{prop}\label{prop:proba of a gluing}
Fix $T \in \A_{s,n}$ and suppose that $G \in \M_{s,n}^{\mathrm{ord}}$ is obtained from $T$ by a gluing plan $\Delta$. Conditionally on $(X; V_1, \ldots, V_s, U_1, \ldots, U_n)$ such that $\widehat{\mathsf{T}}_{s,n}^{s,\mathrm{ord}}=T$, the probability that $\widehat{\mathsf{G}}_{n}^{s,\mathrm{ord}}$ is equal to $G$ depends only on the type of the gluing plan $\Delta$.
	Indeed, for any gluing plan of type 
	\[
	\left(\left(k_v,(k_{v,\ell})_{1\leq\ell\leq\deg_T(v)}\right)_{v\in V(T)}, \left(k_e,a_e,(k_{e,i})_{1\leq i \leq a_e}\right)_{e\in E(T)}\right),
	\]
	this conditional probability is 
\begin{equation}\label{eq:proba gluing plan}
\frac{1}{X(V_1)X(V_2)\dots X(V_s)} \left(\prod_{v\in I(T)}\prod_{\ell=1}^{\deg_T(v)}\frac{N(v,\ell)^{k_{v,\ell}}}{k_{v,\ell}!}\right)
\cdot\left(\prod_{e\in E(T)}\sum_{(j_1,\dots,j_{a_e})\in \N^{a_e,\neq}} \frac{1}{a_e!} \prod_{i=1}^{a_e} \frac{N^{\mathrm{right}}(e,j_{i})^{k_{e,i}}} {k_{e,i}!} \right).
\end{equation}
\end{prop}

\begin{proof}
We reason conditionally on $(X; V_1, \ldots, V_s, U_1, \ldots, U_n)$. Observe that the tree $\widehat{\mathsf{T}}_{s,n}^{s,\mathrm{ord}}$ and random variables $\left(N^{\mathrm{right}}(e,j) : e\in E(T), \ j\geq 1\right)$ and $\left(N(v,\ell) :  v\in I(T), 1\leq \ell \leq \deg_T(v)\right)$ are measurable functions of these quantities, as are the relative orderings of the atoms of local time along an edge. The remaining randomness lies in the random variables $Y_1, \ldots, Y_s$.  Consider first a vertex $v \in I(T)$ and $1 \le \ell \le \deg_T(v)$.  The probability that the leaves among $l_1, \ldots, l_s$ with indices in $S_{v,\ell}$ (where $|S_{v,\ell}| = k_{v,\ell}$) are glued into corner $c_{v,\ell}$ is
\[
\frac{N(v,\ell)^{k_{v,\ell}}}{\prod_{l_j \in S_{v,\ell}} X(V_j)}.
\]
Now consider an edge $e \in E(T)$ and fixed $a_e \ge 1$.  The probability that the leaves among $l_1, \ldots, l_s$ with indices in the sets $S_{e, 1}, \ldots, S_{e,a_e}$ (with $|S_{e,i}| = k_{e,i}$) are grouped together in the gluing, in that top-to-bottom order, is given by summing over $(j_1, \ldots, j_{a_e}) \in \N^{a_e, \neq}$, corresponding to different ordered collections of atoms of local time along the edge $e$, and multiplying by the probability $1/a_e!$ that this vector is such that $L(e,j_1) > L(e, j_2) > \cdots > L(e, j_{a_e})$:
\[
\sum_{(j_1, \ldots, j_{a_e}) \in \N^{a_e, \neq}} \frac{1}{a_e!} \prod_{i=1}^{a_e} \frac{N^{\mathrm{right}}(e, j_i)^{k_{e, i}}} {\prod_{l_j \in S_{e,i}} X(V_{j}) }.
\]
The corners and edges all behave independently, and so multiplying everything together, we obtain that the probability of seeing the particular sets $((S_{v,\ell})_{1 \le \ell \le \deg_T(v)})_{v \in I(T)}, ((S_{e,i})_{1 \le i \le a_e})_{e \in E(T)}$ in the random gluing plan is 
\begin{equation}\label{eq:proba choice atoms}
\frac{1}{X(V_1)X(V_2)\dots X(V_s)} \cdot \left( \prod_{v\in I(T)} \prod_{\ell=1}^{\deg_T(v)}N(v,\ell)^{k_{v,\ell}}\right)
 \cdot \left( \prod_{e\in E(T)} \sum_{(j_1, \ldots, j_{a_e}) \in \N^{a_e, \neq}}\frac{1}{a_e!} \prod_{i = 1}^{a_e}N^{\mathrm{right}}(e,j_i)^{k_{e,i}}\right).
\end{equation} 
Since the permutations $(\sigma_{v,\ell})_{v \in I(T), 1 \le \ell \le \deg_T(v)}$ and $(\sigma_{e,i})_{e\in E(T), 1\leq i \leq a_e}$ are uniform and independent given the sets $((S_{v,\ell})_{1 \le \ell \le \deg_T(v)})_{v \in I(T)}$ and $((S_{e,\ell})_{1 \le \ell \le a_e})_{e \in E(T)}$, we see that each particular collection of permutations arises with conditional probability
\[
\left(\prod_{v\in I(T)}\prod_{\ell=1}^{\deg_T(v)}\frac{1}{k_{v,\ell}!}\right)\cdot\left(\prod_{e\in E(T)}\frac{1}{k_{e,1}!\dots k_{e,a_e}!}\right).
\]
Multiplying \eqref{eq:proba choice atoms} by this quantity gives the desired result.
\end{proof}

Recall that $\widehat{\mathsf{G}}_{n}^{s,\mathrm{ord}}$  is an ordered version of $\widehat{\mathsf{G}}_{n}^s$. We denote by $\mathsf{G}_{n}^{s,\mathrm{ord}}$ the corresponding ordered version in the $s$-biased case.

\textbf{The distribution of $\mathsf{G}_{n}^{s,\mathrm{ord}}$.} We will show that for any ordered multigraph $G\in\M_{s,n}^{\mathrm{ord}}$,
\begin{equation}\label{eq:distribution ordered multigraphs}\Pp{\mathsf{G}_{n}^{s,\mathrm{ord}}=G}\propto\prod_{v\in I(G)}\frac{w_{\deg_G(v)-1}}{(\deg_G(v)-1)!}.
\end{equation}
Fix $G \in\M_{s,n}^{\mathrm{ord}}$.
As previously mentioned, the only way to obtain $G$ by gluing the $s$ unlabelled leaves of a tree $T\in\A_{s,n}$ onto their ancestral paths is if the tree $T$ is the base-tree of $G$, i.e.\ if $T=\textsc{Erase}(\textsc{Dep}(G))$. Let $C_s:=\Ec{X(V_1)\dots X(V_s)}^{-1}$. Then using the change of measure formula \eqref{eq:change of measure}, we have 
\begin{align}\label{eq:unbiasing}
\Pp{\mathsf{G}_{n}^{s,\mathrm{ord}}=G}
&=C_s \cdot \Ec{\ind{\widehat{\mathsf{G}}_{n}^{s,\mathrm{ord}}=G}X(V_1)X(V_2)\dots X(V_s)}\notag\\
&=C_s \cdot \Pp{\widehat{\mathsf{T}}_{s,n}^{s,\mathrm{ord}}=T}\Ecsq{\ind{\widehat{\mathsf{G}}_{n}^{s,\mathrm{ord}}=G}X(V_1)X(V_2)\dots X(V_s)}{\widehat{\mathsf{T}}_{s,n}^{s,\mathrm{ord}}=T}.
\end{align}
Observe here again that, apart from the labels on the leaves, the tree $\widehat{\mathsf{T}}_{s,n}^{s,\mathrm{ord}}$ has exactly the same distribution as $\sT^{\mathrm{ord}}_{s+n}$ defined at the beginning of Section \ref{sec:MarchalTree}. So by \eqref{plane tree distribution}, we have
\begin{equation} \label{eq:tree}
\Prob{\widehat{\mathsf{T}}_{s,n}^{s,\mathrm{ord}}=T} \propto \prod_{v\in I(T)}\frac{w_{\deg_T(v)-1}}{(\deg_T(v)-1)!}.
\end{equation}
We then calculate 
\[
\Ecsq{\ind{\widehat{\mathsf{G}}_{n}^{s,\mathrm{ord}}=G}X(V_1)X(V_2)\dots X(V_s)}{\widehat{\mathsf{T}}_{s,n}^{s,\mathrm{ord}}=T}
\]
by taking expectations in the formula of Proposition~\ref{prop:proba of a gluing} conditionally on the event $\{\widehat{\mathsf{T}}_{s,n}^{s,\mathrm{ord}}=T\}$.  Recall that we fixed $T=\textsc{Erase}(\textsc{Dep}(G))$.
Using Proposition \ref{prop:jointlaws} and Remark~\ref{rem:distribution temps local aretes et noeuds}, we know explicitly the (conditional) distributions of each of the terms in \eqref{eq:proba gluing plan}. Using the independence stated there, we get
\begin{align*}
& \Ecsq{\ind{\widehat{\mathsf{G}}_{n}^{s,\mathrm{ord}}=G}X(V_1)X(V_2)\dots X(V_s)}{\widehat{\mathsf{T}}_{s,n}^{s,\mathrm{ord}}=T} \\
& = \E{\left(\prod_{v\in I(T)}\prod_{\ell=1}^{\deg_T(v)}\frac{N(v,\ell)^{k_{v,\ell}}}{k_{v,\ell}!}\right)
\cdot\left(\prod_{e\in E(T)}\sum_{(j_1,\dots,j_{a_e})\in \N^{a_e,\neq}} \frac{1}{a_e!} \prod_{i=1}^{a_e} \frac{N^{\mathrm{right}}(e,j_{i})^{k_{e,i}}} {k_{e,i}!} \right)} \\
& = \E{N(T)^s} \E{\prod_{v\in I(T)} \left(\frac{N(v)}{N(T)}\right)^{k_v} \prod_{e \in E(T)} \left(\frac{N(e)}{N(T)} \right)^{k_e} } \prod_{v \in I(T)} \E{\prod_{\ell=1}^{\deg_T(v)} \frac{1}{k_{v,\ell}!} \left(\frac{N(v,\ell)}{N(v)} \right)^{k_{v, \ell}} }  \\
& \qquad \times \prod_{e \in E(T)} \E{ \sum_{(j_1, \ldots, j_{a_e}) \in \N^{a_e, \neq}} \frac{1}{a_e!}\prod_{i=1}^{a_e}  \left(\frac{N^{\mathrm{right}}(e, j_i)}{N(e)}\right)^{k_{e,i}} \frac{1}{k_{e,i}!}}
\end{align*}
We now compute the different terms in this product separately.

Using Remark \ref{rem:distribution temps local aretes et noeuds}, we have \[N(T)=N(e_1)+N(e_2)+\dots N(e_{\abs{E(T)}})+N(v_1)+\dots N(v_{\abs{I(T)}}) \sim \mathrm{ML}\left(1/\alpha;n+s-1/\alpha \right),\]
so we get
\[\Ec{N(T)^s}= \frac{\Gamma(n+s-1/\alpha)\Gamma((n+s)\alpha +s-1)}{\Gamma((n+s)\alpha -1)\Gamma(n+s+(s-1)/\alpha)}.\]

Using Remark~\ref{rem:distribution temps local aretes et noeuds} again, 
\[
\left(\frac{N(e_1)}{N(T)},\dots,\frac{N(e_{\abs{E(T)}})}{N(T)},\frac{N(v_1)}{N(T)}, \dots , \frac{N(v_{\abs{I(T)}})}{N(T)}\right) \sim \mathrm{Dir}(\alpha-1,\dots,\alpha-1,d_1-1-\alpha, \dots, d_r-1-\alpha).
\]
Note that $|I(T)| = |E(T)| - n - s$ and $\sum_{v \in I(T)} \deg_T(v) = 2 |E(T)| - n - s -1$, which yield that
\[
(\alpha-1)|E(T)| + \sum_{v \in I(T)} (\deg_T(v) - 1 - \alpha) = (n+s) \alpha - 1.
\]
So \eqref{moments dirichlet} gives
\begin{align*}&\Ec{\prod_{v\in I(T)}\left(\frac{N(v)}{N(T)}\right)^{k_v} \prod_{e\in E(T)}\left(\frac{N(e)}{N(T)}\right)^{k_e} }\\
&\qquad =\frac{\Gamma((n+s)\alpha-1)}{\Gamma((n+s)\alpha+s-1)} \cdot\prod_{v\in I(T)}\frac{\Gamma(\deg_T(v)+k_v-1-\alpha)}{\Gamma(\deg_T(v)-1-\alpha)} \cdot \prod_{e\in E(T)}\frac{\Gamma(\alpha-1+k_e)}{\Gamma(\alpha-1)}.
\end{align*}
Let $v\in I(T)$. Proposition \ref{prop:jointlaws} gives
\[\left(\frac{N(v,1)}{N(v)},\dots, \frac{N(v,\deg_T(v))}{N(v)}\right)\sim \mathrm{Dir}(1,\dots,1),\]
and then \eqref{moments dirichlet} yields 
\begin{align*}\Ec{\prod_{\ell=1}^{\deg_T(v)}\frac{1}{k_{v,\ell}!}\left(\frac{N(v,\ell)}{N(v)}\right)^{k_{v,\ell}}}&=\frac{\Gamma(\deg_T(v))}{\Gamma(\deg_T(v)+k_v)}\cdot \left(\prod_{\ell =1}^{\deg_T(v)}\frac{\Gamma(k_{v,\ell}+1)}{\Gamma(1)}\right)\cdot  \left(\prod_{\ell=1}^{\deg_T(v)} \frac{1}{k_{v,\ell}!}\right)\\
&=\frac{(\deg_T(v)-1)!}{(\deg_T(v)+k_v-1)!}.
\end{align*}
Let $e\in E(T)$. Using Proposition \ref{prop:jointlaws}, we have 
\[
\left(\frac{N(e,j)}{N(e)}\right)_{j\geq 1}\sim \mathrm{PD}(\alpha-1,\alpha-1), \quad \text{and} \quad \left(\frac{N^{\mathrm{right}}(e,j)}{N(e,j)}\right)_{j\geq 1} \text{ are i.i.d.\ $\mathrm{U}[0,1]$,}
\]
so using Lemma~\ref{lem:pd}, and the fact that $\Ec{U^p}=1/(p+1)$ for $U \sim \mathrm{U}\intervalleff{0}{1}$, we get
\[
\Ec{ \sum_{(j_1,\dots,j_{a_e})\in \N^{a_e,\neq}}
\left( \frac{N^{\mathrm{right}}(e,j_1)}{N(e)} \right)^{k_{e,1}} \cdots \left( \frac{N^{\mathrm{right}}(e, j_{a_e})}{N(e)} \right)^{k_{e,a_e}} } 
= \left(\prod_{i=1}^{a_e} \frac{w_{k_{e,i}+1}}{k_{e,i}+1}\right)\cdot \frac{\Gamma(\alpha-1)}{\Gamma(k_e+\alpha-1)}\cdot a_e!.
\]
Multiplying this by the combinatorial factor $\displaystyle \frac{1}{a_e!k_{e,1}!\dots k_{e,a_e}!}$, we get
\[
\prod_{i=1}^{a_e} \frac{w_{k_{e,i}+1}}{(k_{e,i}+1)!}\cdot \frac{\Gamma(\alpha-1)}{\Gamma(k_e+\alpha-1)}.
\]
So, multiplying everything together, we get
\begin{align}\label{eqn:proba of gluing}
&\Ecsq{\ind{\widehat{\mathsf{G}}_{n}^{s,\mathrm{ord}}=G}X(V_1)X(V_2)\dots X(V_s)}{\widehat{\mathsf{T}}_{s,n}^{s,\mathrm{ord}}=T} \notag \\
&=\frac{\Gamma(n+s-1/\alpha)}{\Gamma(n+s + (s-1)/\alpha)}\cdot\left(\prod_{e\in E(T)}\prod_{i=1}^{a_e} \frac{w_{k_{e,i}+1}}{(k_{e,i}+1)!}\right) \cdot\prod_{v\in I(T)}\frac{\Gamma(\deg_T(v)+k_v-1-\alpha)}{(\deg_T(v)+k_v -1)!}\frac{(\deg_T(v)-1)!}{\Gamma(\deg_T(v)-1-\alpha)}.
\end{align}
Now, if we fix an ordered multigraph $G\in \M_{s,n}^{\mathrm{ord}}$, from \eqref{eq:unbiasing} and \eqref{eq:tree} we get
\[
\Pp{\mathsf{G}_{n}^{s,\mathrm{ord}}=G} \propto \prod_{v \in I(T)} \frac{w_{\deg_T(v)-1} \ \Gamma(\deg_T(v) + k_v - 1- \alpha)}{(\deg_T(v) + k_v - 1)! \ \Gamma(\deg_T(v) - 1 - \alpha)} \cdot \left(\prod_{e\in E(T)}\prod_{i=1}^{a_e} \frac{w_{k_{e,i}+1}}{(k_{e,i}+1)!}\right).
\]
Observe finally that every \emph{new} internal vertex in $G$ corresponds to some $e\in E(T)$ and some $1\leq i \leq a_e$, and has degree $k_{e,i}+2$. For a vertex $v\in I(T)$, its degree in $G$ is $\deg_G(v)=\deg_T(v)+k_v$. Moreover, 
\[
w_{\deg_T(v)+k_v-1}=w_{\deg_T(v)-1}\cdot\frac{ \Gamma(\deg_T(v)+k_v-1-\alpha)}{\Gamma(\deg_T(v)-1-\alpha)}.
\]
Putting everything together, we indeed get (\ref{eq:distribution ordered multigraphs}).

\bigskip

We have now assembled all of the ingredients needed for the proof of Theorem~\ref{th:distr_marginals}.

\begin{proof}[Proof of Theorem~\ref{th:distr_marginals}]
Take a multigraph $G\in \M_{s,n}$ with internal vertices $I(G)$, edge multiset $E(G)$ and a number $\mathrm{sl}(G)$ of self-loops. From Lemma \ref{eq:nb of cyclic orderings}, the number of corresponding ordered multigraphs is 
\begin{equation*}\frac{\prod_{v\in I(G)}(\mathrm{deg}(v)-1)!}{\abs{\mathrm{Sym}(G)}2^{\mathrm{sl}(G)}\prod_{e\in \supp(E(G))}\mathrm{mult}(e)!}.
\end{equation*}
Combining this with \eqref{eq:distribution ordered multigraphs}, we get that for any multigraph $G\in \M_{s,n}$, 
	\[\Pp{\mathsf{G}^s_{n}=G}\propto\frac{\prod_{v\in I(G)} w_{\deg(v)-1}}{\abs{\mathrm{Sym}(G)}2^{\mathrm{sl}(G)}\prod_{e\in \supp(E(G))}\mathrm{mult}(e)!},\]
as claimed.
\end{proof}

\subsection{The distribution of $(\sG^s_{n}, n \ge 0)$ as a process}
\label{sec:Marchal}

We now turn to the proof of Theorem~\ref{thm:MarchalAlg}, which says that the sequence $(\sG^s_{n},n\geq 0)$ evolves according to the multigraph version of Marchal's algorithm given in Section~\ref{sec:marginalsintro}. Again, it is easier to work with multigraphs having cyclic orderings of the half-edges around each vertex in order to break symmetries.  Recall from Section~\ref{sec:calculations} that $\mathsf G^{s,\mathrm{ord}}_{n}$ denotes a version of $\mathsf G^{s}_n$ with cyclic orderings around the vertices built from the trees $\mathsf T^{\mathrm{s,ord}}_{s,n}$.  We observe that there is a natural coupling of $\mathsf T_{s,n}^{s,\mathrm{ord}}$ for $n \ge 0$ obtained by repeatedly sampling new uniform leaves. Let $(\mathsf G^{s}_n, n\geq 0)$ and $(\mathsf G^{s,\mathrm{ord}}_{n}, n\geq 0)$ be built from this coupled version of the base trees.  Note that, for all $n$, $\mathsf G^{s,\mathrm{ord}}_{n}$ is obtained from $\mathsf G^{s,\mathrm{ord}}_{n+1}$ by erasing the leaf labelled $n+1$ together with the edge to which it is connected. Recall also from (\ref{eq:distribution ordered multigraphs}) that the distribution of $\mathsf G^{s,\mathrm{ord}}_{n}$ is 
\begin{align*}
\Pp{\mathsf{G}_{n}^{s,\mathrm{ord}}=G}=c_{s,n}\cdot \prod_{v \in I(G)}\frac{w_{\deg_G(v)-1}}{(\deg_G(v)-1)!}, \quad \forall G \in \M_{s,n}^{\mathrm{ord}},
\end{align*}
where $c_{s,n}$ is the normalizing constant. We need an ordered counterpart of Marchal's algorithm for graphs with cyclic orderings around vertices. Starting from a graph $G\in\M_{s,n}^{\mathrm{ord}}$ and assigning to its edges and vertices the weights of Marchal's algorithm, we decide that (1) if a vertex is selected, we  glue the new edge-leaf in a corner chosen uniformly around this vertex, while (2) if an edge is selected, we place the new edge-leaf on the right or on the left of the selected edge each with probability $1/2$.

We will prove Theorem~\ref{thm:MarchalAlg} together with the following result.

\begin{prop}
\label{prop:Marchalordered}
The sequence $(\mathsf G^{s,\mathrm{ord}}_{n}, n\geq 0)$ is Markovian, with transitions given by the ordered version of Marchal's algorithm.
\end{prop}

\begin{proof}[Proof of Proposition~\ref{prop:Marchalordered} and Theorem~\ref{thm:MarchalAlg}]
The Markov property of $(\mathsf G^s_n,n \geq 0)$ and $(\mathsf G^{s,\mathrm{ord}}_{n},n \geq 0)$ is immediate since the backward transitions are deterministic. Now fix $n$ and let $G^{\mathrm{ord}} \in \M^{\mathrm{ord}}_{s,n}$ and $H^{\mathrm{ord}} \in \M^{\mathrm{ord}}_{s,n+1}$ be such that $G^{\mathrm{ord}}$ is obtained from $H^{\mathrm{ord}}$ by erasing the leaf labelled $n+1$ and the adjacent edge. Note that  the internal vertices of our graphs are mutually distinguishable since the graphs are planted, with cyclic orderings around internal vertices. Then,
\begin{equation*}
\Pp{\mathsf G^{s,\mathrm{ord}}_{n+1}=H^{\mathrm{ord}} \ |\ \mathsf G^{s,\mathrm{ord}}_{n}=G^{\mathrm{ord}}}
= \frac{\Pp{\mathsf G^{s,\mathrm{ord}}_{n+1}=H^{\mathrm{ord}}}}{\Pp{\mathsf G^{s,\mathrm{ord}}_{n}=G^{\mathrm{ord}}}} 
= \frac{c_{s,n+1}}{c_{s,n}} \cdot \frac{\displaystyle  \prod_{v\in I(H^{\mathrm{ord}})}\frac{w_{\deg_{H^{\mathrm{ord}}}(v)-1}}{(\deg_{H^{\mathrm{ord}}}(v)-1)!}}{\displaystyle \prod_{v\in I(G^{\mathrm{ord}})}\frac{w_{\deg_{G^{\mathrm{ord}}}(v)-1}}{(\deg_{G^{\mathrm{ord}}}(v)-1)!}}.
\end{equation*}
Now there are two different cases, (a) and (b) below.
\vspace{-0.4cm}
	\begin{itemize}
\item[(a)] The leaf $n+1$ of $H^{\mathrm{ord}}$ is attached to a vertex $v$ of $H^{\mathrm{ord}}$ that has a degree greater or equal to $4$.  In this case, $v$ corresponds to a vertex of $G^{\mathrm{ord}}$, still denoted by $v$,  and $I(H^{\mathrm{ord}})=I(G^{\mathrm{ord}})$, $\deg_{G^{\mathrm{ord}}}(v)=\deg_{H^{\mathrm{ord}}}(v)-1$ and the degree of any other internal vertex is identical in $G^{\mathrm{ord}}$ and $H^{\mathrm{ord}}$. Since
\[
w_{\deg_{H^{\mathrm{ord}}}(v)-1}=w_{\deg_{G^{\mathrm{ord}}}(v)}=(\deg_{G^{\mathrm{ord}}}(v) -1-\alpha)w_{\deg_{G^{\mathrm{ord}}}(v)-1},
\]
together with the above expression for $\Pp{\mathsf G^{s,\mathrm{ord}}_{n+1}=H^{\mathrm{ord}} \ | \ \mathsf G^{s,\mathrm{ord}}_{n}=G^{\mathrm{ord}}}$ this implies that
\begin{equation}
\label{eq:marchal algo on vertex}
\Pp{\mathsf G^{s,\mathrm{ord}}_{n+1}=H^{\mathrm{ord}} \ |\ \mathsf G^{s,\mathrm{ord}}_{n}=G^{\mathrm{ord}}}=\frac{c_{s,n+1}}{c_{s,n}} \cdot \frac{\deg_{G^{\mathrm{ord}}}(v) -1-\alpha}{\deg_{G^{\mathrm{ord}}}(v)}.
\end{equation}
\item[(b)] The vertex $v$ has degree $3$ in $H^{\mathrm{ord}}$ and is erased when erasing the leaf $n+1$ and the adjacent edge. In this case $I(H^{\mathrm{ord}})=I(G^{\mathrm{ord}})\cup \{v\}$ and 
\begin{equation}
\label{eq:marchal algo on edge}
\Pp{\mathsf G^{s,\mathrm{ord}}_{n+1}=H^{\mathrm{ord}}\ | \ \mathsf G^{s,\mathrm{ord}}_{n}=G^{\mathrm{ord}}}=\frac{c_{s,n+1}}{c_{s,n}} \cdot \frac{\alpha-1}{2}.
\end{equation}
\end{itemize}
Proposition~\ref{prop:Marchalordered} follows immediately. 

This argument also gives the transition probabilities of the process $(\mathsf G^s_n,n \geq 0)$. Recall the function
$\psi:\M^{\mathrm{ord}}_{s,n}\rightarrow \M_{s,n}$ that forgets the cyclic ordering around vertices. We have that
\begin{equation}
\label{eq:formule_probas_totales}
\Pp{\mathsf G^s_{n+1}=H\ |\ \mathsf G^s_{n}=G}=\sum_{G^{\mathrm{ord}} \in \psi^{-1}(G)}\Pp{\mathsf G^s_{n+1}=H\ |\ \mathsf G^{s,\mathrm{ord}}_{n}=G^{\mathrm{ord}}} \Pp{\mathsf G^{s,\mathrm{ord}}_{n}=G^{\mathrm{ord}} \ |\ \mathsf G^s_{n}=G}.
\end{equation}
If $H$ is obtained from $G$ by attaching a leaf-edge to a vertex $v$ of $G$, then, from (\ref{eq:marchal algo on vertex}), we get
$$
\Pp{\mathsf G^s_{n+1}=H\ |\ \mathsf G^{s,\mathrm{ord}}_{n}=G^{\mathrm{ord}}} =\deg_G(v) \cdot \frac{c_{s,n+1}}{c_{s,n}} \cdot \frac{\deg_G(v) -1-\alpha}{\deg_G(v)} \quad \text{for all }G^{\mathrm{ord}} \in \psi^{-1}(G).
$$  
With (\ref{eq:formule_probas_totales}), this gives
$$
\Pp{\mathsf G^s_{n+1}=H\ |\ \mathsf G^s_{n}=G}=\frac{c_{s,n+1}}{c_{s,n}} \cdot (\deg_G(v) -1-\alpha).
$$
Similarly, from (\ref{eq:marchal algo on edge}) and (\ref{eq:formule_probas_totales}), we get that when $G'$ is obtained from $G$ by attaching a leaf-edge to the middle of an edge of $G$, we have
$$
\Pp{\mathsf G^s_{n+1}=H \ |\ \mathsf G^s_{n}=G}=\frac{c_{s,n+1}}{c_{s,n}} \cdot (\alpha-1).
$$
Theorem~\ref{thm:MarchalAlg} follows.
\end{proof}

\subsection{The unrooted kernel $\mathsf G_{-1}^s$}
\label{sec:unrooted}

In this section, we fix $s \geq 2$. Our goal is to prove that the distribution of $\mathsf G_{-1}^s$ is that given in Theorem~\ref{th:distr_marginals}, and that the conditional probability of $\mathsf G_{0}^s$ given $\mathsf G_{-1}^s$ is given by a step in Marchal's algorithm. We cannot proceed as before since the use of cyclic orderings around vertices is not sufficient to break all the symmetries in the unrooted graph $\mathsf G^{s}_{-1}$. We instead label the internal vertices: let $\mathsf G_{0}^{s,\mathrm{lab}}$ denote a version of $\mathsf G^s_0$ with internal vertices labelled uniformly from $1$ to $|V(\mathsf  G^s_0)|$. 

For any connected multigraph $G$ (labelled or not) we write 
$$
w(G):=\frac{\prod_{v\in I(G)} w_{\deg(v)-1}}{|I(G)|!~2^{\mathrm{sl}(G)}\prod_{e\in \supp(E(G))}\mathrm{mult}(e)!},
$$
with the usual notation. From Theorem~\ref{th:distr_marginals} and (\ref{eqn:symmetries}), we know that the distribution of the labelled graph $\mathsf G_{0}^{s,\mathrm{lab}}$ is
\begin{equation}
\label{eq:distretiq_0}
\mathbb P\big(\mathsf G_{0}^{s,\mathrm{lab}}=G\big)=\tilde{c}_{s,0} \cdot w(G),
\end{equation}
where $\tilde{c}_{s,0}$ is the normalising constant.

Let $H^{\mathrm{lab}}$ and $G^{\mathrm{lab}}$ be labelled versions of multigraphs in $\mathbb M_{s,0}$ and $\mathbb M_{s,-1}$ respectively that are \emph{compatible} in the sense that removing the root and the adjacent edge (in the following, we will use the word \emph{root-edge}) in $H^{\mathrm{lab}}$ gives a graph which, after an increasing mapping of the labelling  to $\{1,\ldots,|V(G^{\mathrm{lab}})|\}$, is $G^{\mathrm{lab}}$. 
We then distinguish 2 cases, precisely one of which occurs.
\begin{enumerate} 
\item[(a)] The root-edge in $H^{\mathrm{lab}}$ is attached to a vertex $v$ of degree $\deg_{H^{\mathrm{lab}}}(v) \geq 4$, in which case 
$$
w(H^{\mathrm{lab}})=\frac{w_{\deg_{H^{\mathrm{lab}}}(v)-1}}{w_{\deg_{G^{\mathrm{lab}}}(v)-1}}\cdot w(G^{\mathrm{lab}})=(\deg_{G^{\mathrm{lab}}}(v)-1-\alpha) \cdot w(G^{\mathrm{lab}}).
$$ 
Note that, given $G^{\mathrm{lab}}$ and a vertex $v$ of $G^{\mathrm{lab}}$, there is a unique graph $H^{\mathrm{lab}}$ which has its root-edge attached to $v$ and is compatible with $G^{\mathrm{lab}}$.
\item[(b)] The root-edge is attached to a vertex $v$ of degree $\deg_{H^{\mathrm{lab}}}(v) = 3$. Its deletion either ``creates" an edge $e$ of $G^{\mathrm{lab}}$ (possibly a self-loop, erasing then at the same time an edge of multiplicity 2 in $H^{\mathrm{lab}}$) or increases by 1 the multiplicity of an edge $e \in \mathrm{supp}(H^{\mathrm{lab}})$ (possibly a multiple self-loop, erasing, again,  at the same time an edge of multiplicity 2 in $H^{\mathrm{lab}}$). In all cases, $$w(H^{\mathrm{lab}})=\frac{w_{\deg_{H^{\mathrm{lab}}}(v)-1} \cdot \mathrm{mult}(e)}{|I(G^{\mathrm{lab}})|+1}\cdot w(G^{\mathrm{lab}})=\frac{(\alpha-1)\cdot \mathrm{mult}(e)}{|I(G^{\mathrm{lab}})|+1}\cdot w(G^{\mathrm{lab}}),$$
where $\mathrm{mult}(e)$ refers here to the multiplicity of $e$ seen as an element of  $\mathrm{supp}(G^{\mathrm{lab}})$. Note that given an edge $e$ of $G^{\mathrm{lab}}$, there are exactly $|I(G^{\mathrm{lab}})|+1$ graphs $H^{\mathrm{lab}}$ with the root-edge attached in the middle of (a copy of) $e$ that are compatible with $G^{\mathrm{lab}}$.
\end{enumerate}

From this, (\ref{eq:distretiq_0}) and the fact that the sum of the Marchal weights is $(s-1)(\alpha+1)$ for any graph in $\mathbb M_{s,-1}$ (see (\ref{identity:degrees})), we obtain the distribution of $G_{-1}^{s,\mathrm{lab}}$:
\begin{align*}
\Pp{\mathsf  G_{-1}^{s,\mathrm{lab}}=G^{\mathrm{lab}}}&=\sum_{\substack{\text{$H^{\mathrm{lab}}$ compatible} \\ \text{with $G^{\mathrm{lab}}$}}} \Pp{\mathsf  G_{0}^{s,\mathrm{lab}}=H^{\mathrm{lab}}} \\
& = \tilde{c}_{s,0}\sum_{\substack{\text{$H^{\mathrm{lab}}$ compatible} \\ \text{with $G^{\mathrm{lab}}$}}} w(H^{\mathrm{lab}}) \\
&= \tilde{c}_{s,0} \cdot \left(\sum_{v \in I(G^{\mathrm{lab}})} (\deg_{G^{\mathrm{lab}}}(v)-1-\alpha) + \sum_{e \in \mathrm{supp}(E(G^{\mathrm{lab}}))}  \mathrm{mult}(e)(\alpha-1) \right) \cdot w(G^{\mathrm{lab}}) \\
&= \tilde{c}_{s,-1} \cdot w(G^{\mathrm{lab}}),
\end{align*}
where $\tilde{c}_{s,-1}=\tilde{c}_{s,0}(s-1)(\alpha+1)$.
Together with  (\ref{eqn:symmetries}), which holds for graphs of $\mathbb M_{s,-1}$, this implies that $\mathsf G_{-1}^s$ has the required distribution. Next, to get the conditional distribution of  $\mathsf G^s_0$ given $\mathsf G^s_{-1}$ we write, for $H \in \mathbb M_{s,0}$ and $G \in \mathbb M_{s,-1}$,
\begin{equation*}
\Pp{\mathsf G^s_0 =H \ | \ \mathsf G^s_{-1} =G}
= \sum_{\substack{\text{$G^{\mathrm{lab}}$ a labelled} \\ \text{version of $G$}}} \frac{\Pp{\mathsf G^s_0 =H, \mathsf G^{s,\mathrm{lab}}_{-1} =G^{\mathrm{lab}}}}{\Pp {\mathsf G^{s,\mathrm{lab}}_{-1} 
=G^{\mathrm{lab}}}} \Pp{\mathsf G^{s,\mathrm{lab}}_{-1} =G^{\mathrm{lab}}  \  |  \  \mathsf G^s_{-1} =G}.
\end{equation*}
From the remarks above, we see that when $H$ is obtained from $G$ by gluing the root-edge to a vertex $v$ of $G$, we get
$$
 \frac{\Pp{\mathsf G^s_0 =H, \mathsf G^{s,\mathrm{lab}}_{-1} =G^{\mathrm{lab}}}}{\Pp{ \mathsf G^{s,\mathrm{lab}}_{-1} =G^{\mathrm{lab}}}}  = \frac{\tilde{c}_{s,0}}{\tilde{c}_{s,-1}} \cdot \frac{w(H)}{w(G)}=\frac{\tilde{c}_{s,0}}{\tilde{c}_{s,-1}} \cdot \left( \deg_G(v)-1-\alpha\right),
$$  
for all labelled versions $\mathsf G^{\mathrm{lab}}$.  If, on the other hand, $H$ is obtained from $G$ by gluing the root-edge to (a copy of) an edge $e \in \mathrm{supp}(G)$, 
$$
 \frac{\Pp{\mathsf G^s_0 =H,  \mathsf G^{s,\mathrm{lab}}_{-1} =G^{\mathrm{lab}}}}{\Pp{ \mathsf G^{s,\mathrm{lab}}_{-1} =G^{\mathrm{lab}}}}  =(|I(G)|+1) \cdot \frac{\tilde{c}_{s,0}}{\tilde{c}_{s,-1}} \cdot  \frac{w(H)}{w(G)}=\frac{\tilde{c}_{s,0}}{\tilde{c}_{s,-1}} \cdot  (\alpha-1) \cdot \mathrm{mult(e)}.
$$  
Putting everything together, we see that we do indeed obtain the transition probabilities corresponding to a step of Marchal's algorithm.

\subsection{The configuration model embedded in a limit component} 
\label{sec:configembedded}

The goal of this subsection is to prove Corollary \ref{cor:identification} where we identify for each $n \geq 0$ (and $n=-1$ if $s\geq 2)$ the distribution of $\mathsf G_n^s$ with that of a specific configuration model. 

\textbf{Two probability distributions.}
In Section 3 of Duquesne and Le Gall~\cite{DuquesneLeGall}, it is shown that the rooted subtree obtained by sampling $n \geq 0$ leaves in the $\alpha$-stable tree is distributed as a planted (non-ordered version of a) Galton-Watson tree conditioned to have $n$ leaves, with critical offspring distribution $\eta_\alpha$ satisfying
$$
\eta_\alpha(k)=\frac{w_k}{k!}, \quad k \geq 2, \quad \quad \eta_\alpha(1)=0, \quad \quad \eta_\alpha(0)=\frac{1}{\alpha},
$$
or, equivalently, with probability generating function $z+\alpha^{-1}(1-z)^{\alpha}, z \in (0,1],$ as already mentioned in Section \ref{sec:marginalsintro}.
Note that $\eta_{\alpha}(k) \sim_{k \rightarrow \infty} c k^{-1-\alpha}$ for some constant $c>0$, by Stirling's approximation. 
Now consider the random variable $D^{(\alpha)}$ with distribution introduced in (\ref{def:Dalpha}), and note that it is indeed a probability distribution since
\begin{align*}
\sum_{k \geq 2}\frac{w_k}{k!}&=\frac{(\alpha-1)}{2}+\sum_{k \geq 3}\frac{(k-1-\alpha)w_{k-1}}{k!} = \frac{(\alpha-1)}{2}+\sum_{k \geq 3}\frac{w_{k-1}}{(k-1)!}-(1+\alpha) \sum_{k \geq 3}\frac{w_{k-1}}{k!},
\end{align*}
which implies that 
$$
\sum_{k \geq 2}\frac{w_{k-1}}{k!}+\frac{1}{\alpha}= \frac{(\alpha-1)}{2(1+\alpha)}+\frac{1}{\alpha}=\frac{\alpha^2+\alpha+2}{2(1+\alpha)\alpha}.
$$
It is straightforward to see that $\mathbb E[D^{(\alpha)}]=2$.  Moreover, if we consider the biased version
$$\mathbb P(\hat{D}^{(\alpha)}=k):=\frac{k \mathbb P(D^{(\alpha)}=k)}{\mathbb E\big[D^{(\alpha)}\big]}, \quad k\geq 1$$
we immediately get that $\hat{D}^{(\alpha)}-1$ has the same distribution as $\eta_\alpha$. This in particular implies that $D^{(\alpha)}$ satisfies the conditions (\ref{hyp:stable}).

\textbf{The stable configuration model.} Fix $n \geq 0$ if $s\in \{0,1\}$ or $n \geq -1$ if $s \geq 2$. Then fix $m \ge n+1$ and consider the multigraph $\mathsf C_{m}$ sampled from the configuration model with i.i.d.\ degrees $D^{(\alpha)}_0,\ldots,D^{(\alpha)}_{m-1}$ distributed as $D^{(\alpha)}$. From Proposition~7.7 in \cite{RvdH}, we have that 
$$
\mathbb P\left(\mathsf C_{m} = G \ \Big| \ D^{(\alpha)}_i=d_i, 0 \leq i \leq m-1\right)=\frac{1}{(\sum_{0\leq i \leq m-1} d_i-1)!!}\cdot\frac{\prod_{0 \leq i \leq m-1} d_i!}{2^{\mathrm{sl}(G)}\prod_{e\in \supp(E)}\mathrm{mult}(e)!},
$$
for every multigraph $G=(V,E)$ with $m$ \emph{labelled} vertices of respective degrees $d_0,\ldots,d_{m-1}$ such that $\sum_{0\leq i \leq m-1} d_i$ is even. Hence, the distribution of  $\mathsf C_{m}$  is given for each such multigraph by
$$
\mathbb P( \mathsf C_{m}=G)= \left(\frac{2(1+\alpha)\alpha}{\alpha^2+\alpha+2}\right)^m \cdot \frac{1}{(\sum_{0\leq i \leq m-1} d_i-1)!!}\cdot\frac{\prod_{0 \leq i \leq m-1} d_i!}{2^{\mathrm{sl}(G)}\prod_{e\in \supp(E)}\mathrm{mult}(e)!} \cdot\frac{1}{\alpha^{\#\{i:d_i=1\}}}\cdot \prod_{i=0}^{m-1}\frac{w_{d_i-1}}{d_i!}.
$$
On the event $\{\mathsf C_m \text{ is connected},  s(\mathsf C_m) = s\}$, the sum $\sum_{0\leq i \leq m-1} d_i$ depends only on $m$ and $s$.  Conditioning additionally on $\{D^{(\alpha)}_0 = \cdots = D^{(\alpha)}_{n} = 1, D^{(\alpha)}_i \neq 1, n+1 \le i \le m-1\}$, we have $\#\{i:D^{(\alpha)}_i=1\} = n+1$. Forgetting the labels $n+1, \ldots, m-1$ (which we now know belong to internal vertices), we obtain a factor of $(m-n-1)!/|\mathrm{Sym}(G)|$.  (See (\ref{eqn:symmetries}) for further discussion.)
Together with Theorem~\ref{th:distr_marginals} this implies Corollary \ref{cor:identification}.

\section{Two simple constructions of the graph $\mathcal G^s$}
\label{sec:final}


Let $s \geq 1$. We start by proving in Section~\ref{sec:approx} that the (measured) $\mathbb R$-graph $\mathcal G^s$ is the almost sure limit of rescaled versions of its combinatorial shapes $\mathsf G_n^s, n \geq 0$ equipped with the uniform distribution on their leaves.  Together with the algorithmic construction of the graphs  $\mathsf G_n^s, n \geq 0$ (Theorem~\ref{thm:MarchalAlg}) and some urn model asymptotics recalled in the Appendix, this will lead us to the two alternative constructions of $\mathcal G^s$ presented in the introduction: in Section~\ref{sec:distrGs}, we prove Theorem~\ref{thm:distrGs} and Proposition~\ref{cor:lengthskernel}, giving the distribution of $\mathcal G^s$ as a collection of rescaled $\alpha$-stable trees appropriately glued onto the kernel $\mathsf K^s$; Section~\ref{sec:LB} is then devoted to the line-breaking construction of Theorem~\ref{th:linebreaking}.

\subsection{The graph as the scaling limit of its marginals}
\label{sec:approx}

Recall from Section \ref{sec:SLconfiguration} that $\mathcal G^s$ is constructed from $\mathcal T^s$, a biased version of the $\alpha$-stable tree, by gluing appropriately $s$ marked leaves onto randomly selected branch-points. Recall also that $X^s$ denotes the $s$-biased stable excursion from which $\mathcal T^s$ is built, that $\pi^s(V_1^s), \ldots, \pi^s(V_s^s)$ are the $s$ leaves to be glued and that $\pi^s(U_i), i \ge 1$ are i.i.d.\ uniform leaves. For all $n\geq 1$, $\mathcal T^s_{s,n}$ then denotes the subtree of $\mathcal T^s$ spanned by the root and the leaves $\pi^s(V_1^s), \ldots, \pi^s(V_s^s), \pi^s(U_1),\ldots, \pi^s(U_n)$ and we let $\mathsf T^{s}_{s,n}$ be its combinatorial shape. Finally, recall that $\mathcal G_n^s$ is the connected subgraph of $\mathcal G^s$ consisting of the union of the kernel and the paths from the leaves $\pi^s(U_1), \ldots, \pi^s(U_n)$ to the root, for all $n \geq 0$, and that the finite graph $\mathsf G_n^s$ denotes the combinatorial shape of $\mathcal G_n^s$. We will use the following observation: for all $n$ larger than some finite random variable, $\mathcal G_n^s$ is obtained from $\mathcal T^s_{s,n}$ by an appropriate gluing of the $s$ leaves  $\pi^s(V_1^s),\ldots,\pi^s(V_s^s)$ to some of its \emph{internal} vertices (for small $n$, it may be that we instead glue some leaves along edges of $\mathcal T^s_{s,n}$).

The goal of this section is to prove Proposition~\ref{prop:approx}: when the graph $\mathsf G_n^s$ is equipped with the uniform distribution on its leaves, 
\begin{equation}
\label{cv:GHPgraph}
\frac{\mathsf G_n^s}{n^{1-1/\alpha}} \underset{n \rightarrow \infty}{\overset{\mathrm{a.s.}}\longrightarrow}  \alpha \cdot\mathcal G^s
\end{equation}
for the Gromov-Hausdorff-Prokhorov topology. With this aim in mind, we first observe that $\mathcal G^s$ can be recovered from the completion of the union of its continuous marginals.  

\begin{lem}
\label{lem:density}
With probability one,
$$
\mathcal G^s= \overline{\cup_{n \geq 0} \mathcal G^s_n}
$$
and consequently $\mathcal G^s$ is the a.s. limit of $\mathcal G^s_n$
in $(\mathscr C,\mathrm{d}_{\mathrm{GHP}})$, when the graph $\mathcal G_n^s$ is endowed with the uniform distribution on its leaves for $n \ge 1$.
\end{lem}

Indeed, it is well-known that the $\alpha$-stable tree is almost surely the completion of the union of its continuous marginals, which entails a similar result for the biased version $\mathcal T^s$ and then for the graph $\mathcal G^s$, using its construction from $\mathcal T^s$. The measures can then be incorporated by using the strong law of large numbers.

\begin{proof}[Proof of Proposition~\ref{prop:approx}] We make use of the fact (\ref{cv:GHPtrees}) that the $\alpha$-stable tree is the almost sure scaling limit of its discrete marginals. We refer the reader to the book of Burago, Burago and Ivanov~\cite{BBI01} for background on the notions of a \emph{correspondence} and its \emph{distortion}, which are used here for the proof.  

By Lemma \ref{lem:density}, it suffices to prove that almost surely
$$
\mathrm{d}_{\mathrm{GHP}}\left(\frac{\mathsf G_n^s}{n^{1-1/\alpha}}, \alpha \cdot \mathcal G_n^s \right) \underset{n \rightarrow \infty}\longrightarrow 0.
$$
We observe first that
\[
\mathrm{d}_{\mathrm{GHP}} \left(n^{1/\alpha - 1} \mathsf T_{s,n}^{s} , \alpha \cdot \mathcal{T}_{s,n}^s\right) \underset{n \rightarrow \infty}\longrightarrow 0
\]
almost surely.  This is proved for $s = 0$ in \cite[Section 2.4]{CurienHaas} and may be transferred to $s \ge 1$ by absolute continuity.  The $s = 0$ case is proved in \cite{CurienHaas} by using a natural correspondence which we introduce here for general $s$ and call $\mathcal{R}_n^s$.  It is a correspondence between $n^{\frac{1}{\alpha}-1} \mathsf T^s_{s,n}$ and $\alpha \cdot \mathcal T^s_{s,n}$. The leaves with the same labels correspond to one another, and the internal vertices of $\mathsf T^s_{s,n}$ are put in correspondence with the branch-points of $\mathcal{T}_{s,n}^s$ in the obvious way.  Finally, the edges of $\mathcal{T}_{s,n}^s$ (which have real-valued lengths and which we think of as line-segments) are put in correspondence with the vertex or vertices of $\mathsf T_{s,n}^s$ corresponding to their end-points.  From \cite{CurienHaas} we obtain that the distortion $\mathrm{dist}(\mathcal{R}_n^s)$ of the correspondence $\mathcal{R}_n^s$ tends to 0 almost surely as $n \to \infty$. To deal with the gluing, we use the fact already observed above that for $n$ sufficiently large, $\mathcal G_n^s$ is obtained from $\mathcal T^s_{s,n}$ by an appropriate gluing of the $s$ leaves $\pi^s(V_1^s), \ldots, \pi^s(V_s^s)$ to its internal vertices; similarly $\mathsf G_n^s$ is obtained by the gluing of the corresponding leaves of $\mathsf T^s_{s,n}$ to the corresponding internal vertices of this tree. It then follows from Lemma~4.2 of \cite{ABBrGoMi} that 
\[
\mathrm{d}_{\mathrm{GHP}}\left(\frac{\mathsf G_n^s}{n^{1-1/\alpha}}, \alpha \cdot \mathcal G_n^s \right) \le \frac{(s+1)}{2} \mathrm{dist}(\mathcal{R}_n^s)
\]
and the claimed almost sure convergence follows easily.
\end{proof}

\subsection{Construction from randomly scaled stable trees glued to the kernel}
\label{sec:distrGs}

We now turn to the proof of Theorem~\ref{thm:distrGs} which states that in $\left(\mathscr C, \mathrm{d}_{\mathrm{GHP}}\right)$, we have the identity in distribution of the measured compact metric spaces
\begin{equation}
\label{distrGs}
\mathcal G^s \ \overset{\mathrm{d}}= \ \mathcal G(\mathsf{K}^s)
\end{equation}
(with the notation used in Section \ref{cons1:gluing}). We will also prove Proposition~\ref{cor:lengthskernel} in this section. 

\begin{proof}[Proof of Theorem~\ref{thm:distrGs}] Using (\ref{cv:GHPgraph}), we just need to prove that
$$
\frac{\mathsf G_n^s}{n^{1-1/\alpha}} \underset{n \rightarrow \infty}{\overset{\mathrm{d}}\longrightarrow} \alpha \cdot \mathcal G(\mathsf{K}^s)
$$ 
for the Gromov-Hausdorff-Prokhorov topology, when the graph $\mathsf G_n^s$ is equipped with the uniform distribution on its leaves. (We will prove the compactness of the object on the right-hand side below.) As discussed earlier, the graph $\mathsf G_n^s$ may be viewed as a collection of trees glued to the kernel $\mathsf{K}^s$.  We will show that each of these tree-blocks converges after rescaling to its continuous counterpart used in the construction of $\mathcal G(\mathsf{K}^s)$.  Our argument and notation are similar to those used in the proof of Proposition~\ref{prop:jointlaws} concerning the stable tree.

We work conditionally on $\mathsf{K}^s$. Let $m$ denote the number of edges of $\mathsf{K}^s$, which are arbitrarily labelled as $e_1,\ldots, e_m$. Let $v_1,\ldots,v_{m-s}$ denote the internal vertices of $\mathsf{K}^s$, again in arbitrary order, and $d_1,\ldots,d_{m-s}$ their respective degrees. 
For each $n \geq 0$, we interpret these edges (resp. vertices) as edges of $\mathsf G^s_n$ with edge-lengths (resp. vertices).
For each $k$, we write $T_n(e_k)$ for the subtree of $\mathsf G^s_n$ induced by the vertices closer to $e_k$ than to any other edge $e_i,i\neq k$, including the two end-points of $e_k$. These end-vertices are interpreted as leaves of $T_n(e_k)$ and count as distinct leaves even if $e_k$ is a loop. (These formulation may seem arbitrary but it is the one needed to initiate properly the urn model we will use below.) The number of leaves of $T_n(e_k)$ is then denoted by  $M_n(e_k)$. Similarly we let $T_n(v_{i})$ denote the subtree of $\mathsf G^s_n$ induced by the set of all vertices closer to $v_i$ than to any edge $e_k, 1 \leq k \leq m$, including $v_i$ which is considered as its root. Then $M_n(v_i)$ denotes its number of leaves (here $v_i$ is \emph{not} considered to be a leaf so that, in particular, $M_n(v_i)=0$ if $T_n(v_{i})$ has vertex-set $\{v_i\}$). Next, for each $1 \leq i \leq m-s$, let $T_{n}(v_{i},j),j \geq 1$ denote the connected components of $T_n(v_{i})\backslash \{v_i\}$. We think of these subtrees as planted (and we again call the root of each $v_i$), so that if we identify their roots we recover $T_n(v_i)$. The number of such subtrees is finite (possibly zero) for each $n$ but tends to infinity as $n\rightarrow \infty$. We label them $T_{n}(v_{i},1), T_{n}(v_{i},2),\ldots$ in order of appearance, with the convention that $T_{n}(v_i,j)$ is the empty set if there are strictly fewer than $j$ subtrees at step $n$. Let $M_{n}(v_i,j)$ be the number of leaves of $T_{n}(v_{i},j), j \geq 1$.

\smallskip

\emph{$\bullet$ Scaling limits of the numbers of leaves.} It is easy to see using the algorithmic construction of the sequence $(\mathsf G_n^s,n\geq 0)$ from Theorem~\ref{thm:MarchalAlg} that the process 
\[
\left(\alpha M_n(e_1)-\alpha-1, \ldots, \alpha M_n(e_m)-\alpha-1,  \alpha M_n(v_1)+d_1-1-\alpha, \ldots, \alpha M_n(v_{m-s}) + d_{m-s} - 1 - \alpha \right)_{n\geq 0}
\]
evolves according to P\'olya's urn (see Theorem~\ref{thm:urn2}) with $2m-s$ colours of initial weights \[\left(\alpha-1,\ldots,\alpha-1,d_1-1-\alpha,\ldots,d_{m-s}-1-\alpha\right)\] respectively, and weight parameter $\alpha$. Hence, there exists a random variable $(M_1,\ldots,M_{2m-s})$ with the Dirichlet distribution of parameters specified at (\ref{distr1}) such that
\[
\left(\frac{M_n(e_1)}{n},\ldots, \frac{M_n(e_m)}{n},\frac{M_n(v_1)}{n},\ldots, \frac{M_n(v_{m-s})}{n}\right) \underset{n \rightarrow \infty}{\overset{\mathrm{a.s.}}\longrightarrow} (M_1,\ldots,M_{2m-s}).
\]
Next we observe that for all $i$ the jumps of $((M_{n}(v_{i},j))_{j\geq 1},n \geq 0)$ follow the same dynamics as a Chinese restaurant process with parameters $1/\alpha$ and $(d_i-1-\alpha)/\alpha$, independently of everything else. Since the total number of jumps at step $n$ is $M_n(v_i)$, Theorem~\ref{thm:urn3} yields
\[
\left(\frac{M^{\downarrow}_{n}(v_{i},j)}{M_n(v_i)},j\geq 1\right)  \underset{n \rightarrow \infty}{\overset{\mathrm{a.s.}}\longrightarrow} (\Delta_{i,j},j\geq 1),
\]
where $(M^{\downarrow}_{n}(v_{i},j),j\geq 1)$ denotes the decreasing reordering of $(M_{n}(v_{i},j),j\geq 1)$ and the limit $(\Delta_{i,j},j\geq 1)$ follows a Poisson-Dirichlet $\mathrm{PD}(1/\alpha, (d_i-1-\alpha)/\alpha)$ distribution, independent of the random variable $(M_1,\ldots,M_{2m-s})$. (The convergence holds in $\ell^1$ equipped with its usual metric.)

\smallskip

\emph{$\bullet$  Scaling limits of the trees $T_n(e_k),T_{n}(v_{i},j)$.} Given the processes $(M_n(e_k),n\geq 0)$, $(M_{n}(v_{i},j), n \geq 0)$, for all $k,i,j$, the jump evolutions of the trees $T_n(e_k)$, $T_{n}(v_{i},j)$, $n \geq 0$ are independent and all follow Marchal's algorithm. Then writing $e_k=\{x_k,y_k\}$ for $1\leq k \leq m$, we know by (\ref{cv:GHPtrees}) that there exist \emph{rescaled} (measured) $\alpha$-stable trees $\mathcal T_k,\mathcal T_{i,j}$, $k,i,j$ such that, given $(M_1,\ldots,M_{2m-s})$ and $(\Delta_{i,j},j\geq 1)$, the trees are independent, $\mathcal T_k$ has total mass $M_k$, $\mathcal T_{i,j}$ total mass $M_{i+m} \cdot \Delta_{i,j}$ and, furthermore,
\begin{enumerate}
\item[(a)] for all $k$,
$$
\left(\frac{T_n(e_k)}{n^{1-1/\alpha}},x_k,y_k\right)=\left(\left(\frac{M_n(e_k)}{n}\right)^{1-1/\alpha} \cdot\frac{T_n(e_k)}{M_n(e_k)^{1-1/\alpha}}  ,x_k,y_k\right) \underset{n \rightarrow \infty}{\overset{\mathrm{a.s.}}\longrightarrow}\left( \alpha \cdot \mathcal T_k,\rho_k,L_k\right)
$$
for the 2-pointed Gromov-Hausdorff-Prokhorov topology, the tree $T_n(e_k)$ being implicitly endowed with the measure that assigns weight $1/n$ to each of its leaves (here, $\rho_k$ denotes the root of $\mathcal T_k$ and $L_k$ a uniform leaf);   
\item[(b)] for all $i,j$,
$$
\left(\frac{T_{n}(v_{i},j)}{n^{1-1/\alpha}},v_i\right)=\left( \left(\frac{M_{n}(v_{i},j)}{n} \right)^{1-1/\alpha}\cdot \frac{T_{n}(v_{i},j)}{M_{n}(v_{i},j)^{1-1/\alpha}} ,v_i \right)\underset{n \rightarrow \infty}{\overset{\mathrm{a.s.}}\longrightarrow} \left(\alpha \cdot \mathcal T_{i,j},\rho_{i,j} \right)
$$
for the pointed Gromov-Hausdorff-Prokhorov topology, where again $T_{n}(v_{i},j)$ is endowed with the measure that assigns weight $1/n$ to each of its leaves, and $\rho_{i,j}$ is the root of $\mathcal T_{i,j}$.
\end{enumerate}

\emph{$\bullet$  Scaling limits of the trees $T_n(v_{i})$, and the compactness of the limit.} Fix $i \ge 1$ and recall that $T_n(v_{i})$ is obtained by identifying the roots of the trees $T_{n}(v_{i},j),j\geq 1$. We now show that $n^{-(1-1/\alpha)} T_n(v_i)$ converges in probability for the pointed GHP-topology to the measured $\R$-tree $\mathcal T_{(i)}$ obtained by identifying the roots of the trees $\alpha \cdot \mathcal T_{i,j}$. 

Let us first show that $\mathcal T_{(i)}$ is compact and is the almost sure GHP-limit as $j_0 \rightarrow \infty$ of the $\R$-tree $\mathcal T_{(i)}^{j_0}$ obtained by gluing the first $j_0$ trees $\mathcal T_{i,j},j\leq j_0$ together at their roots. (For different values of $j_0$ we think of the underlying spaces as being nested and all contained within $\mathcal T_{(i)}$.)   For a rooted $\R$-tree $\mathsf t$, we write $\mathrm{ht}(\mathsf t)$ for its height.
Let $\mathcal T$ denote a standard $\alpha$-stable tree (of total mass 1).  Then by the scaling property of the stable tree we have
\[
\E{ \left(\sup_{j > j_0}\mathrm{ht}(\mathcal T_{i,j})\right)^{\alpha/(\alpha-1)}}
\le \sum_{j > j_0} \mathbb E\big[\mathrm{ht}(\mathcal T_{i,j})^{\alpha/(\alpha-1)}\big] = \mathbb E\big[\mathrm{ht}(\mathcal T)^{\alpha/(\alpha-1)}\big] \mathbb E\big[M_{i+m}\big] \sum_{j > j_0} \mathbb E\big[\Delta_{i,j}\big].
\]
Since $\mathrm{ht}(\mathcal{T})$ has finite exponential moments (see, for example, equation (2) of \cite{Kortchemski} for a convenient statment) the right-hand side is finite, and clearly tends to 0 as $j_0 \to \infty$.  Hence the decreasing sequence $\sup_{j > j_0}\mathrm{ht}(\mathcal T_{i,j})$ converges a.s.\ to 0 as $j_0 \to \infty$. This implies in particular that $\mathcal T_{(i)}$ is a.s.\ compact. Then, note that 
\[\mathrm{d}_{\mathrm{GHP}}\left(\mathcal T_{(i)},\mathcal T_{(i)}^{j_0}\right)\le \max\Bigg(\sup_{j > j_0}\mathrm{ht}(\mathcal T_{i,j}), M_{i+m} \cdot \sum_{j > j_0} \Delta_{i,j} \Bigg)
\]
since $M_{i+m} \cdot \sum_{j > j_0} \Delta_{i,j}$ is the total mass of $\mathcal T_{(i)} \setminus \mathcal T_{(i)}^{j_0}$. This total mass also converges to 0.  Hence, $\mathcal T_{(i)}^{j_0} \to \mathcal{T}_{(i)}$ almost surely as $j_0 \to \infty$ with respect to the GHP-topology.

Next, note that for $j_0 \in \mathbb N$,
\begin{align*}
\mathrm{d}_{\mathrm{GHP}}\left(\frac{T_n(v_i)}{n^{1-1/\alpha}},\alpha \cdot \mathcal T_{(i)} \right) &\leq  \sum_{j=1}^{j_0} \mathrm{d}_{\mathrm{GHP}}\left(\frac{T_n(v_{i}, j)}{n^{1-1/\alpha}},\alpha \cdot \mathcal T_{i,j} \right) +\alpha \cdot \mathrm{d}_{\mathrm{GHP}}\left(\mathcal T_{(i)}^{j_0} ,\mathcal T_{(i)} \right) \\
& \qquad +\sup_{j > j_0} \mathrm{ht}\left(\frac{T_n(v_{i}, j)}{n^{1-1/\alpha}}\right)  + \sum_{j > j_0} \frac{M_n(v_{i},j)}{n}.
\end{align*}
We already know that the first term on the right-hand side converges a.s.\ to 0 as $n \rightarrow \infty$ (for $j_0$ fixed) and that the second term converges a.s.\ to 0 as $j_0 \rightarrow \infty$. Moreover, since $M_n(v_i) \le n$, by dominated convergence we have
\[
\mathbb E\left[ \sum_{j > j_0} \frac{M_n(v_{i},j)}{n}\right]=\mathbb E\left[ \frac{M_n(v_i)}{n}-\sum_{j \leq j_0} \frac{M_n(v_{i},j)}{n}\right] \underset{n \rightarrow \infty}\longrightarrow \mathbb E\left[M_{i+m}\left(1- \sum_{j \leq j_0} \Delta_{i,j}\right)\right]
\]
and then
\[
\lim_{j_0 \rightarrow \infty} \lim_{n \rightarrow \infty} \mathbb E\left[ \sum_{j > j_0} \frac{M_n(v_{i},j)}{n}\right]=0.
\]
Now note that 
\begin{align*}
\limsup_{n \rightarrow \infty}  \sum_{j> j_0}\mathbb E\left[ \frac{(\mathrm{ht}(T_{n}(v_{i},j))^{\alpha/(\alpha-1)}}{n}\right] & \leq  \limsup_{n \rightarrow \infty} \sum_{j> j_0} \mathbb E\left[ \frac{(\mathrm{ht}(T_{n,j}(v_{i},j))^{\alpha/(\alpha-1)}}{M_{n}(v_{i},j)} \cdot \frac{M_{n}(v_{i},j)}{n} \right] \\
&\underset{}\leq  C_{\alpha} \limsup_{n \rightarrow \infty}  \sum_{j> j_0} \mathbb E\left[ \frac{M_{n}(v_{i},j)}{n} \right],
\end{align*}
by \cite[Lemma~33]{HaasMiermont}, where  $C_{\alpha}$ is a finite constant depending only on $\alpha$. So by Markov's inequality, we get
\[
\lim_{j_0 \rightarrow \infty} \limsup_{n \rightarrow \infty}\mathbb P\left(\sup_{j > j_0} \mathrm{ht}\left(\frac{T_n(v_{i},j)}{n^{1-1/\alpha}}\right)>\varepsilon\right)=0
\]
for all $\varepsilon>0$.
Putting everything together, we obtain the convergence in probability
\[
\mathrm{d}_{\mathrm{GHP}}\left(\frac{T_n(v_i)}{n^{1-1/\alpha}},\alpha \cdot \mathcal T_{(i)} \right) \convprob 0.
\]

\emph{$\bullet$  Final gluing.} Finally, the graph $\mathsf G_n^s$ is obtained by gluing appropriately the $2m-s$ trees $T_n(e_k), T_n(v_{i}), 1 \leq k \leq m, 1 \leq i \leq m-s$ along the kernel $\mathsf{K}^s$.  Using the results above, it therefore converges in probability, after multiplication of distances by $n^{-(1-1/\alpha)}$, to a version of $\alpha \cdot \mathcal G(\mathsf{K}^s)$.
\end{proof}

From this we immediately obtain the joint distribution of the edge-lengths of the continuous kernel $\mathcal{K}^s$. Given that the number of edges of $\mathcal{K}^s$ is $m$ and keeping the notation of the proofs, we see that the lengths of the $m$ edges are given by $M_i^{1-1/\alpha} \cdot \Lambda_i, 1 \leq i\leq m$ where the $\Lambda_i$ are i.i.d.\ $\mathrm{ML}(1-1/\alpha,1-1/\alpha)$ random variables (this is the distribution of the distance between a uniform leaf and the root in a standard $\alpha$-stable tree) and independent of $(M_1,\ldots,M_{2m-s})$. We may combine Remark~\ref{rem:identityMLDir} and Lemma~\ref{lem:DirichletBeta} to check that the distribution of this $m$-tuple of random variables coincides with the one of Proposition~\ref{cor:lengthskernel} when $n=0$. More generally, we could deduce from (\ref{distrGs}) the joint distribution of the edge-lengths of the  continuous marginals $\mathcal G_n^s$, $n \geq 0$. However, it is simpler to prove this directly using urn arguments similar to those above.

\begin{proof}[Proof Proposition~\ref{cor:lengthskernel}] Fix $n_0 \geq 0$. We work conditionally on $\mathsf G^s_{n_0}=(V,E)$. For each edge $e \in E$ and each $n \geq n_0$, let $L_n(e)$ denote the length of $e$ in $\mathsf G^s_n$ and let $L_n^{\mathrm{tot}}:=\sum_{e \in E} L_n(e)$. 
From the algorithmic construction of $(\mathsf G^s_n,n \geq n_0)$ we get that
\begin{enumerate}
\item[(a)] the process
\[(L_n^{\mathrm{tot}}, n \geq n_0)\] is a triangular urn scheme as defined in Theorem~\ref{thm:urn1} with initial weights $$a=|E|, \quad b=\frac{(n_0+s)\alpha+s-1}{\alpha-1}-|E|$$ ($b$ is the initial total weight of the vertices of $\mathsf G^s_{n_0}$, divided by $\alpha-1$) and additional weight parameters $\gamma=1$ and $\beta=\alpha/(\alpha-1)$;
\item[(b)] the jumps of the process $((L_n(e),e \in E),n\geq n_0)$ evolve according to P\'olya's urn with initial weights $a_i=1,1 \leq i \leq |E|$, and additional weight parameter $\beta=1$, independently of $L_n^{\mathrm{tot}}$.
\end{enumerate}
Theorem~\ref{thm:urn1} and Theorem~\ref{thm:urn2} therefore imply that 
$
\left(L_n(e)/n^{1-1/\alpha}, e \in E\right)
$
converges almost surely to a random vector with distribution (\ref{ref:BMLD}). The conclusion then follows from the convergence (\ref{cv:GHPgraph}).
\end{proof}
  
\subsection{The line-breaking construction}
\label{sec:LB}

The proof of Theorem~\ref{th:linebreaking} for $s \geq 1$ is inspired by the approach used in  \cite{GHlinebreaking} to obtain a line-breaking construction of the stable trees. As we have already mentioned, we rely again on the algorithmic construction of the sequence $(\mathsf G_n^s,n\geq 0)$. The notation below coincides with that of Section~\ref{cons2:LB}. 
Moreover, for each $n$, we let $\mathsf H^s_{n}$ denote the combinatorial shape of $\mathcal H^s_n$. 
The metric space $\mathcal H^s_{n}$ is then interpreted as a finite graph (the graph $\mathsf H^s_{n}$) with edge-lengths. We let 
 $L_{n}$ denote this sequence of edge-lengths, ordered arbitrarily, and let $W_{n}$ denote the sequence of weights at internal vertices of $\mathcal H^s_{n}$ (i.e. the weights attributed by the measure $\eta_n$ to each of these vertices), also ordered arbitrarily. We start with a preliminary lemma.  
 
\begin{lem} 
\label{lem:lengthslinebreaking}
Given $\mathsf H^s_{k}$, $0 \leq k \leq n$, and in particular that $\mathsf H^{s}_n$ has $m$ edges and $m-(n+s)$ internal vertices with degrees $d_1,\ldots,d_{m-(n+s)}$, we have 
$$
\big(L_{n}, W_{n}\big) \ \overset{\mathrm{(d)}}= \ \mathrm{ML}\left(1-\frac{1}{\alpha},\frac{(n+s) \alpha +(s-1)}{\alpha} \right) \cdot \mathrm{Dir}\bigg(\underbrace{1,\ldots,1}_m,\frac{d_1-1-\alpha}{\alpha-1},\ldots, \frac{d_{m-(n+s)}-1-\alpha}{\alpha-1}\bigg),
$$
the random variables on the right-hand side being independent. In particular,
$$
L_{n} \ \overset{\mathrm{(d)}}= \ \mathrm{ML}\left(1-\frac{1}{\alpha},\frac{(n+s) \alpha +(s-1)}{\alpha} \right) \cdot \mathrm{Beta}\left( m, \frac{(n+s)\alpha+s-1}{\alpha-1}-m \right) \cdot \mathrm{Dir}\left(1,\ldots,1\right).
$$
\end{lem}

\begin{proof}
For $n=0$, the first identity in distribution holds by definition of $(\mathcal H^s_{0},\eta_0)$ in the line-breaking construction. The rest of the proof proceeds by induction on $n$, and is based essentially on manipulations of Dirichlet distributions. The steps  are exactly the same as those of Proposition~3.2 in \cite{GHlinebreaking}, to which we refer the interested reader. The only slight change to highlight is that here the degrees $d_1,\ldots,d_{m-{(n+s)}}$ of the internal vertices of a graph in $\mathbb M_{s,n}$ with $m$ edges necessarily satisfy
$$
\sum_{i=1}^{m-(n+s)}\frac{d_i-1-\alpha}{\alpha-1}=\frac{(n+s)\alpha+s-1}{\alpha-1}-m,
$$
as already observed in (\ref{identity:degrees}). This fact is also used, together with Lemma~\ref{lem:DirichletBeta}, to deduce the distribution of $L_{n}$ from that of the pair $(L_{n},W_{n})$.
\end{proof}

\begin{proof}[Proof of Theorem~\ref{th:linebreaking}]
Note that the metric spaces $\mathcal H^s_n,n\geq 0$ have implicit leaf-labels, given by their order of appearance in the construction. The metric spaces  $\mathcal G^s_n,n\geq 0$ are also leaf-labelled by construction. Both models are sampling consistent: the metric space indexed by $n$ is obtained from the metric space indexed by $n+1$ by removing the leaf labelled $n+1$ and the adjacent line-segment (this description is a little informal but hopefully clear). Hence, we only need to prove that, for all $n \geq 0$,
\begin{equation}
\label{id:marginalsLB}
\mathcal H^s_n  \overset{\mathrm{d}}= {\mathcal G}^s_n,
\end{equation}
these compact metric spaces being implicitly endowed with the uniform distribution on their leaves, and still leaf-labelled.
Together with the sampling consistency, this will imply that the processes  of compact measured metric spaces  $(\mathcal H^s_n,n\geq 0)$ and  $(\mathcal G^s_n,n\geq 0)$ have the same distribution. Since $\mathcal G^s$ is the almost sure GHP-scaling limit of $\mathcal G^s_n$ (Lemma~\ref{lem:density}) and since $(\mathscr C,\mathrm{d}_{\mathrm{GHP}})$ is complete, this will in turn entail that $\mathcal H^s_n$ converges a.s.\ to a random compact measured metric space distributed as $\mathcal G^s$.

To prove (\ref{id:marginalsLB}), we first check that the sequence of finite graphs $(\mathsf H^s_n,n\geq 0)$ evolves according to Marchal's algorithm, as does $(\mathsf G^s_n,n\geq 0)$.
This relies on Lemmas~\ref{lem:lengthslinebreaking} and \ref{lem:biasDir} which imply that for each $n$, given $(\mathsf H^s_k, 0\leq k \leq n)$,  the probability that the new segment in the line-breaking construction is attached to a given edge of $\mathsf H^s_n$ is proportional to 1, whereas the probability that it is attached to a given vertex with degree $d_i\geq 3$ is proportional to $(d_i-1-\alpha)/(\alpha-1)$. Hence, the sequences of graphs $(\mathsf H^s_n,n \geq 0)$ and $({\mathsf G}^s_n, n \geq 0)$ have the same distribution since $\mathsf G_0^s=\mathsf H_0^s=\mathsf{K}^s$, including leaf-labels. Then we get (\ref{id:marginalsLB}) by simply noticing that the distribution of the edge-lengths of $\mathcal H^s_n$ given $(\mathsf H^s_k,0\leq k \leq n)$ is the same as that of the edge-lengths of ${\mathcal G}^s_n$ 
given  $(\mathsf G^s_k,0\leq k \leq n)$, by Lemma~\ref{lem:lengthslinebreaking} and Proposition~\ref{cor:lengthskernel}. 
\end{proof}

\section{Appendix: distributions, urn models and applications}
\label{sec:urns}


We detail in this appendix some classical asymptotic results on urn models that are needed at various points in the paper. We first recall the definitions and some properties of several distributions that are related to these asymptotics. 

\subsection{Some probability distributions}

For more detail on the material in this section, we refer to Pitman~\cite{PitmanStFl}.

\subsubsection{Definitions and moments}

\textbf{Beta distributions.} For parameters $a,b >0$, the $\mathrm{Beta}(a,b)$ distribution has density
\[
\frac{\Gamma(a+b)}{\Gamma(a) \Gamma(b)} x^{a-1} (1-x)^{b-1}
\]
with respect to the Lebesgue measure on $(0,1)$. If $B \sim \mathrm{Beta}(a,b)$, then for $p,q \in \mathbb R_+,$ 
\begin{align}\label{eq:moments beta}
\Ec{B^p(1-B)^q}=\frac{\Gamma(a+b)}{\Gamma(a+b+p+q)}\frac{\Gamma(a+p)}{\Gamma(a)}\frac{\Gamma(b+q)}{\Gamma(b)}.
\end{align}

\medskip

\textbf{Dirichlet distributions.}  For parameters $a_1, a_2, \ldots, a_n > 0$, the Dirichlet distribution $\mathrm{Dir}(a_1, a_2, \ldots, a_n)$ has density
\[
\frac{\Gamma(\sum_{i=1}^n a_i)}{\prod_{i=1}^n \Gamma(a_i)} \prod_{j=1}^{n} x_i^{a_j-1}
\]
with respect to the Lebesgue measure on the simplex $\{(x_1, \ldots, x_n) \in [0,1]^n: \sum_{i=1}^n x_i = 1\}$.
When $(X_1, \dots, X_n)\sim \mathrm{Dir}(a_1,\dots, a_n)$, for $k_1,\ldots,k_n \in \mathbb R_+$,
\begin{equation}\label{moments dirichlet}
	\Ec{X_1^{k_1}X_2^{k_2}\dots X_n^{k_n}}= \frac{\Gamma\left(\sum_{i=1}^na_i\right)}{\Gamma(\sum_{i=1}^n (a_i+k_i))}\cdot \prod_{i=1}^n\frac{\Gamma(a_i+k_i)}{\Gamma(a_i)}.
\end{equation}

\medskip

\textbf{Generalized Mittag-Leffler distributions.} Let $0<\beta<1$, $\theta>-\beta$. An $\R_+$-valued random variable $M$ has the {generalized Mittag-Leffler distribution} $\mathrm{ML}(\beta,\theta)$ if, for all suitable test functions $f$, we have
\begin{equation}
\label{def:ML}
\mathbb E\left[ f(M)\right]= \frac{\mathbb E\left[ \sigma_{\beta}^{-\theta}f\big(\sigma_{\beta}^{-\beta}\big)\right]}{\mathbb E\big[ \sigma_{\beta}^{-\theta}\big]},
\end{equation}
where $\sigma_{\beta}$ is a stable random variable with Laplace transform $\mathbb E[e^{-\lambda \sigma_{\beta}}] =  \exp(-\lambda^{\beta}),\lambda \geq 0$.
For $p \in \mathbb R_+$, 
\[
\mathbb E\left[ M^p\right]=\frac{\Gamma(\theta) \Gamma(\theta/\beta + p)}{\Gamma(\theta/\beta) \Gamma(\theta + p \beta)}=\frac{\Gamma(\theta+1) \Gamma(\theta/\beta + p+1)}{\Gamma(\theta/\beta+1) \Gamma(\theta + p \beta+1)}.
\]

\medskip

\textbf{Poisson-Dirichlet distributions.} Let $0<\beta<1$, $\theta>-\beta$ and for $i \ge 1$, let $B_i \sim \mathrm{Beta}(1-\beta, \theta + i \beta)$ independently.  Then the decreasing sequence $(P_i)_{i\geq 1}=(Q_i^{\downarrow})_{i \ge 1}$ where $Q_j = B_j \prod_{i=1}^{j-1} (1 - B_i)$ has the $\mathrm{PD}(\beta,\theta)$ distribution. The almost sure limit $ W := \Gamma(1 - \beta) \lim_{i \to \infty} i (P_i^{\downarrow})^{\beta}$ has the $\mathrm{ML}(\beta,\theta)$ distribution. 

\subsubsection{Distributional properties}

\begin{lem}
\label{lem:DirichletBeta}
If $(X_1,\ldots,X_n) \sim \mathrm{Dir}(a_1,\ldots,a_n)$ then for all $1 \leq m \leq n-1$,
$
(X_1,\ldots,X_m)
$
is distributed as the product of two independent random variables:
$$
\mathrm{Beta}\left(\sum_{i=1}^m a_i,\sum_{i=m+1}^n a_i\right) \cdot \mathrm{Dir}(a_1,\ldots,a_m).
$$
\end{lem}

\medskip

\begin{lem}
\label{lem:biasDir}
Suppose that $(X_1, X_2, \ldots, X_n) \sim \mathrm{Dir}(a_1, a_2, \ldots, a_n)$.  Let $I$ be the index of a size-biased pick from amongst the co-ordinates i.e.\ $\Prob{I=i|X_1, X_2, \ldots, X_n} = X_i$, for $1 \le i \le n$.  Then
\[
\Prob{I = i} = \frac{a_i}{a_1 + a_2 + \ldots + a_n}
\]
for $1 \le i \le n$ and, conditionally on $I=i$,
\[
(X_1, X_2, \ldots, X_n) \sim \mathrm{Dir}(a_1, \ldots, a_{i-1}, a_i + 1, a_{i+1}, \ldots, a_n).
\]
\end{lem}

\medskip

\begin{lem}\label{lem:PD factorisation}
Let $0<\beta<1$, $\theta>-\beta$, and let $(P_i)_{i\geq 1}$ have distribution $\mathrm{PD}(\beta,\theta)$. Let $J$  be the index of a size-biased pick from this sequence, i.e.\ $\Ppsq{J=j}{(P_i)_{i\geq 1}}=P_j$, for $j\geq 1$. We let $(P_i')_{i\geq 1}$ be the decreasing sequence $(1-P_J)^{-1}\cdot (P_i)_{i\ge 1, i\neq J}$, reindexed by $\mathbb N$. Then
\begin{align*}
P_J\sim \mathrm{Beta}(1-\beta,\theta+\beta) \quad \text{and} \quad  (P_i')_{i\geq 1} \sim \mathrm{PD}(\beta,\theta+\beta),
\end{align*}
and these two random variables are independent.
\end{lem}

\medskip

Let $\N^{n,\neq}:=\enstq{(i_1,\dots,i_n)\in \mathbb N^n }{i_1,\dots,i_n \text{ are distinct}}$.
\begin{lem}\label{lem:pd}
	Let $(P_i)_{i\geq 1} \sim \mathrm{PD}(\beta,\theta)$ with $0<\beta<1$ and $\theta>-\beta$. Then for all $k_1, k_2, \dots, k_n \in \intervallefo{1}{\infty}$, 
	\begin{equation}
	\label{eq:momentsPD}
	\Ec{\sum_{(i_1,\dots,i_n)\in \N^{n,\neq}}P_{i_1}^{k_1} \dots P_{i_n}^{k_n}}= \left(\prod_{i=1}^n \beta \frac{\Gamma(k_i-\beta)}{\Gamma(1-\beta)}\right)\frac{\Gamma(\theta)}{\Gamma(\theta+\sum_{j=1}^n k_j)} \frac{\Gamma(\theta/\beta+n)}{\Gamma(\theta/\beta)}.
	\end{equation}
	In particular, for $(P_i)_{i\geq 1} \sim \mathrm{PD}(\alpha-1,\alpha -1)$ with $\alpha\in\intervalleoo{1}{2}$, and $k_1,\dots k_n \in \N$, we have
	\begin{align*}\Ec{\sum_{(i_1,\dots,i_n)\in \N^{n,\neq}}P_{i_1}^{k_1} \dots P_{i_n}^{k_n}}&= \left(\prod_{i=1}^n (\alpha-1) \frac{\Gamma(k_i+1-\alpha)}{\Gamma(2-\alpha)}\right)\frac{\Gamma(\alpha-1) \  n!}{\Gamma(\alpha-1+\sum_{j=1}^n k_j)} \\
	&= \left(\prod_{i=1}^n w_{k_i+1}\right) \frac{\Gamma(\alpha-1) \ n!}{\Gamma(\alpha-1+\sum_{j=1}^n k_j)} ,
	\end{align*}
	where the weights $w_1,w_2,\dots$ are defined in \eqref{Marchal's weights}. 
\end{lem}

\begin{proof}
	We proceed by induction on $n$. For $n=0$ we use the convention that the left-hand side of (\ref{eq:momentsPD}) is 1 and so the identity is true.
	Let $n\geq 1$ and suppose that the identity is true for $n-1$. Then letting $J$ be such that $\Ppsq{J=j}{(P_i)_{i\geq 1}}=P_j$, we have
\begin{align*}
& \Ec{ \sum_{(i_1,\ldots,i_n)\in \N^{n,\neq}} P_{i_1}^{k_1} \dots P_{i_n}^{k_n} } \\
& \qquad =\Ec{ P_J^{k_n-1} (1-P_J)^{k_1+\cdots+k_{n-1} } \underset{\in (\N\setminus\{J\})^{n-1,\neq}}{\sum_{(i_1,\dots,i_{n-1})}}\left( \frac{P_{i_1}}{1-P_J} \right)^{k_1} \cdots \left(\frac{P_{i_{n-1}}}{1-P_J}\right) ^{k_{n-1}}}\\
& \qquad  = \Ec{ P_J^{k_n-1}(1-P_J)^{k_1+\cdots+k_{n-1}}} \cdot \Ec{\sum_{(i_1,\cdots,i_{n-1})\in \N^{n-1,\neq}}(P'_{i_1})^{k_1} \cdots (P'_{i_{n-1}})^{k_{n-1}}},
\end{align*}
by Lemma~\ref{lem:PD factorisation}, where $(P'_i)_{i\geq 1}\sim \mathrm{PD}(\beta,\beta+\theta)$ and $P_J\sim\mathrm{Beta}(1-\beta,\theta+\beta)$. Using \eqref{eq:moments beta}, we have
\begin{align*}
\Ec{ P_J^{k_n-1}(1-P_J)^{k_1+\dots+k_{n-1}}}&=\frac{\Gamma(1+\theta)\Gamma(1-\beta+k_n-1)\Gamma(\theta+\beta+\sum_{i=1}^{n-1}k_i)}{\Gamma(\theta+\beta)\Gamma(1-\beta)\Gamma(1+\theta+\sum_{i=1}^{n}k_i-1)}\\
&=\left(\beta\frac{\Gamma(k_n-\beta)}{\Gamma(1-\beta)}\right)\frac{\Gamma(\theta)\Gamma(\theta+\beta+\sum_{i=1}^{n-1}k_i)}{\Gamma(\theta+\beta)\Gamma(\theta+\sum_{j=1}^{n}k_j)}\frac{\theta}{\beta}.
\end{align*}
The induction hypothesis applied to the sequence $(P'_i)_{i\geq 1}$, which has distribution $\mathrm{PD}(\beta,\beta+\theta)$, then yields
\begin{align*}
& \Ec{\sum_{(i_1,\dots,i_{n-1})\in \N^{n-1,\neq}}(P'_{i_1})^{k_1} \dots (P'_{i_{n-1}})^{k_{n-1}}} \\
&\qquad =\left(\prod_{i=1}^{n-1} \beta \frac{\Gamma(k_i-\beta)}{\Gamma(1-\beta)}\right)\frac{\Gamma(\theta+\beta)}{\Gamma(\theta+\beta+\sum_{j=1}^{n-1} k_j)} \frac{\Gamma((\theta+\beta)/\beta+n-1)}{\Gamma((\theta+\beta)/\beta)}\\
&\qquad =\left(\prod_{i=1}^{n-1} \beta \frac{\Gamma(k_i-\beta)}{\Gamma(1-\beta)}\right)\frac{\Gamma(\theta+\beta)}{\Gamma(\theta+\beta+\sum_{j=1}^{n-1} k_j)} \frac{\Gamma(\theta/\beta+n)}{(\theta/\beta)\Gamma(\theta/\beta)},
\end{align*}
and the result for $n$ follows by multiplying together the last display and the preceding one.
\end{proof}

\subsection{P\'olya's urn, Chinese restaurant processes and triangular urn schemes} \label{sec:urnproof}

We gather here some classical results for urn models. 

\begin{theorem}[P\'olya's urn]
\label{thm:urn2}
Consider an urn model with $k$ colours, with initial weights $a_1,\ldots,a_k>0$ respectively. At each step, draw a colour with a probability proportional to its weight and add an extra weight $\beta>0$ to this colour. Let $W^{(1)}_n,\ldots, W^{(k)}_n$ denote the weights of the $k$ colours after $n$ steps. Then
$$
\left(\frac{W^{(1)}_n}{\beta n},\ldots,\frac{W^{(k)}_n}{\beta n}\right)  \underset{n \rightarrow \infty}{\overset{\mathrm{a.s.}}\longrightarrow} (W^{(1)},\ldots,W^{(k)})
$$ 
where $(W^{(1)},\ldots,W^{(k)}) \sim \mathrm{Dir}(a_1/\beta, \ldots, a_k/\beta)$. 
\end{theorem}

\medskip

\begin{theorem}[The Chinese restaurant process]
\label{thm:urn3}
Fix two parameters $\beta \in (0,1)$ and $\theta>-\beta$. The process starts with one table occupied by a single customer and then evolves in a Markovian way as follows: given that at step $n$ there are $k$ occupied tables with $n_i$ customers at table $i$, a new customer is placed at table $i$ with probability $(n_i-\beta)/(n+\theta)$ and placed at a new table with probability $(\theta+k\beta)/(n+\theta)$. Let $N_i(n),i\geq 1$ be the number of customers at table $i$ at step $n$ and let $(N^{\downarrow}_i(n),i\geq 1)$ be the decreasing rearrangement of these terms. Let $K(n)$ denote the number of occupied tables at step $n$. Then
$$
\left(\frac{N^{\downarrow}_i(n),i \geq 1}{n}\right)  \underset{n \rightarrow \infty}{\overset{\mathrm{a.s.} \text{\emph{ in} }\ell^1}\longrightarrow} \left(Y_i,i\geq 1\right) \quad \text{and}\quad \frac{K(n)}{n^{\beta} } \underset{n \rightarrow \infty}{\overset{\mathrm{a.s.}}\longrightarrow} W
$$ 
where $\left(Y_i,i\geq 1\right) \sim \mathrm{PD}(\beta,\theta)$ and $W \sim \mathrm{ML}\left(\beta,\theta \right)$.
\end{theorem}

We refer to Pitman's book \cite[Chapter 3]{PitmanStFl} for more detail on these first two theorems.

\begin{theorem}[Triangular urn schemes]
\label{thm:urn1}
Consider an urn model with two colours, red and black. Suppose that initially red has weight $a>0$ and black has weight $b \geq 0$.  At each step, we sample a colour with probability proportional to its current weight in the urn.  Let $\beta>\gamma > 0$ and assume that when red is drawn then weight $\gamma$ is added to red and weight $\beta-\gamma$ to black, whereas when black is drawn then weight $\beta$ is added to black (and nothing to red). Let $R_n$ denote the red weight after $n$ steps. Then, 
\[
\frac{R_n}{n^{\gamma/\beta}} \underset{n \rightarrow \infty}{\overset{\mathrm{a.s.}}\longrightarrow} R
\]
where the random variable $R$ is such that $R\sim \gamma \cdot \mathrm{Beta}(\frac{a}{\gamma},\frac{b}{\gamma})\cdot \mathrm{ML}\left(\frac{\gamma}{\beta},\frac{(a+b)}{\beta}\right)$ with the Beta and Mittag-Leffler random variables being independent, and the convention that $\mathrm{Beta}(a,0) = 1$ a.s.
\end{theorem}

(Note that, since the total weight in the urn at step $n$ is $a + b + n\beta$, we trivially deduce that the black weight $B_n = a + b + n\beta -R_n$ satisfies $B_n/n \to \beta$ almost surely.)  There is a vast literature on triangular urn schemes, which give rise to profoundly different asymptotic behaviour. We refer to Janson \cite{Janson} for an overview, and in particular to Theorems~1.3 and 1.7  therein which together imply the convergence of Theorem~\ref{thm:urn1} (but only in distribution). The almost sure convergence can, in fact, be deduced from Theorems~\ref{thm:urn2} and \ref{thm:urn3}. Observe first that we may reduce to the case $\gamma = 1$ by scaling. Now note that in the context of Theorem~\ref{thm:urn1} when $\gamma=1$ and $b=0$, the red weight evolves as $a$ plus the number of occupied tables in a Chinese restaurant process  with parameters $(1/\beta,a/\beta)$, and so the almost sure limit has $\mathrm{ML}(1/\beta,a/\beta)$ distribution. To treat the case $b > 0$, consider a refinement of the urn model in which the red colour comes in two variants, light and dark. Start with $a$ light red weight, $b$ dark red weight and 0 black weight. Sample a colour with probability proportional to its current weight in the urn. When black is drawn, add weight $\beta$ to black. When red is drawn in either of its variants, add weight 1 to that variant and weight $\beta-1$ to black. Clearly, light red and dark red + black taken together follow the $\beta$-triangular urn scheme with respective initial weights $a$ and $b$. Moreover, (1) the proportion of the total red weight which is light red converges almost surely to a random variable with $\mathrm{Beta}(a,b)$ distribution by Theorem~\ref{thm:urn2}, and (2) this evolution holds independently of that of the total proportion of red weight in the urn, which converges to a $\mathrm{ML}(1/\beta,(a+b)/\beta)$-distributed random variable, by the Chinese restaurant process as noted above.

We finally turn to the proof of Proposition~\ref{prop:urn}. The notation is introduced in the vicinity of its statement in Section \ref{sec:marginalsstable}.
 
\begin{proof}[Proof of Proposition~\ref{prop:urn}]
Imagine first not distinguishing between the different types of a colour, i.e.\ consider the evolution of
\[
X^{a,b,c}_i(n) = X_i^a(n) + X^b_i(n) + X^c_i(n), \quad 1 \le i \le k.
\]
Then $(X^{a,b,c}_1(n), \ldots, X^{a,b,c}_k(n))_{n \ge 0}$ performs a classical P\'olya's urn in which we always add weight $\alpha$ of the colour picked, and which is started from
\[
(X^{a,b,c}_1(0), \ldots, X^{a,b,c}_k(0)) = (\gamma_1, \ldots, \gamma_k).
\]
So we have
\begin{equation} \label{eqn:(abc)limit}
\frac{1}{\alpha n} (X^{a,b,c}_1(n), \ldots, X^{a,b,c}_k(n)) \to (D_1, \ldots, D_k)
\end{equation}
almost surely as $n \to \infty$, where $(D_1, \ldots, D_k) \sim \mathrm{Dir}(\gamma_1/\alpha, \ldots, \gamma_k/\alpha)$.  Observe that $(X^{a,b,c}_i(n) - \gamma_i)/\alpha$ is the number of times by step $n$ that colour $i$ has been picked.

Now consider the triangular sub-urn which just watches the evolution of colour $i$, which doesn't distinguish between types $a$ and $b$, but does distinguish type $c$.  In particular, at each step we pick either type $\{a,b\}$ or type $c$ with probability proportional to its current weight.  If we pick $\{a,b\}$, we add $1$ to its weight and $\alpha-1$ to the weight of $c$; if we pick $c$, we simply add weight $\alpha$ to c. Write $Y^{a,b}_i(n)$ and $Y^c_i(n)$ for the weights after $n$ steps within this urn, with $Y_i^{a,b}(0) = \gamma_i$ and $Y^c_i(0) = 0$. Then by Theorem \ref{thm:urn1}, we have
\begin{equation} \label{eqn:(ab)climit}
\frac{1}{n^{1/\alpha}} Y_i^{a,b}(n) \to R_i, \quad \frac{1}{\alpha n} Y^c_i(n) \to 1,
\end{equation}
almost surely as $n \to \infty$, where $R_i \sim \mathrm{ML}(1/\alpha, \gamma_i/\alpha)$. Moreover, the number of times we add to type $a$ or $b$ is $Y_i^{a,b}(n) - \gamma_i$.

Now consider the sub-urn which just watches the evolution of types $a$ and $b$ of colour $i$. So if we pick $a$, we add weight $\alpha-1$ to $a$ and $2-\alpha$ to $b$, whereas if we pick $b$ we just add weight $1$ to $b$.  Write $Z^a_i(n)$ and $Z^b_i(n)$ for the weights of types $a$ and $b$ after $n$ steps of this sub-urn, with $Z^a_i(0) = \gamma_i$ and $Z^b_i(0) = 0$. Then again by Theorem \ref{thm:urn1} we have
\begin{equation} \label{eqn:ablimit}
\frac{1}{(\alpha-1) n^{\alpha-1}}Z^a_i(n) \to \bar{R}_i, \quad \frac{1}{n} Z^b_i(n) \to 1
\end{equation}
almost surely, where $\bar{R}_i\sim  \mathrm{ML}(\alpha-1, \gamma_i)$. Finally, observe that the full urn process may be decomposed as follows:
\begin{align*}
X_i^a(n) & = Z_i^a\left(Y_i^{a,b}\left(\frac{X^{a,b,c}_i(n) -\gamma_i}{\alpha}\right) - \gamma_i\right) \\
X_i^b(n) & = Z_i^b\left(Y_i^{a,b}\left(\frac{X^{a,b,c}_i(n) -\gamma_i}{\alpha}\right) - \gamma_i\right) \\
X_i^c(n) & = Y_i^{c}\left(\frac{X^{a,b,c}_i(n) -\gamma_i}{\alpha}\right),
\end{align*}
where the processes $(X_1^{a,b,c}(n), \ldots, X_k^{a,b,c}(n))_{n \ge 0}$,
$(Y_i^{a,b}(n), Y_i^{c}(n))_{n \ge 0}$ \ for $1 \le i \le k$, and $(Z_i^a(n), Z_i^b(n))_{n \ge 0}$ for $1 \le i \le k$,
are all independent.  The claimed results then follow by composing the limits (\ref{eqn:(abc)limit}), (\ref{eqn:(ab)climit}) and (\ref{eqn:ablimit}). 
\end{proof}

\begin{rem}
\label{rem:identityMLDir}
The following statements follow using similar arguments:
\[
(D_1^{1/\alpha} R_1, \ldots, D_k^{1/\alpha} R_k) \equidist R \cdot (\tilde{D}_1, \ldots, \tilde{D}_k),
\]
where $R \sim \mathrm{ML}(1/\alpha, \gamma/\alpha)$ is independent of $(\tilde{D}_1, \ldots, \tilde{D}_k) \sim \mathrm{Dir}(\gamma_1, \ldots, \gamma_k)$, and
\[
R_i^{\alpha-1} \bar{R}_i  \sim \mathrm{ML}(1-1/\alpha,\gamma_i/\alpha) 
\]
for $1 \le i \le k$.
\end{rem}

\section*{Acknowledgements}

B.H. and D.S.'s research is partially supported by the ANR GRAAL ANR-14-CE25-0014.
C.G.'s research is supported by EPSRC Fellowship EP/N004833/1.  She is grateful to LIX, \'Ecole polytechnique for two professeur invit\'e positions in November 2016 and November 2017, and for funding from Universit\'e Sorbonne Paris Nord as part of the Projet MathStic to pay for a week's visit in March 2018 and for a professeur invit\'e position in September 2018, all of which facilitated work on this project.

\bibliographystyle{siam}
\bibliography{stable}
\end{document}